\documentclass[11pt,a4paper]{article}

\usepackage[utf8]{inputenc}
\usepackage[french,english]{babel}
\usepackage[T1]{fontenc}
\usepackage[mono=false]{libertine}
\usepackage{microtype}

\usepackage[margin=2.cm]{geometry}
\usepackage[cmintegrals,libertine]{newtxmath} 
   
\usepackage[cal=euler, scr=boondoxo]{mathalfa}
\usepackage{booktabs}
\useosf

\usepackage{wrapfig,flafter,xcolor}
\usepackage{float}

\usepackage{tabu}
\usepackage[font=small]{caption}
\usepackage{subcaption}

\usepackage{enumerate}
\usepackage[shortlabels]{enumitem}

\usepackage{graphicx}
\usepackage{mathsymb-2021-03-21}

\usepackage{hyperref}
\hypersetup{
	colorlinks=true,
	citecolor=cyan,
	linkcolor=red,
	urlcolor=cyan,
	filecolor=yellow!50!black,
	breaklinks=true,
	pdfpagemode=UseNone,
	bookmarksopen=false,
}

\usepackage{amsthm} % must be loaded before {cleveref} since {cleveref} uses {amsthm} to determine the names of the theorems.

\usepackage{nameref}
\usepackage{cleveref} % must be loaded after {hyperref} to be effective
\usepackage{thm-2021-03-21}

\makeatletter
\DeclareRobustCommand{\labeltext}[2]{%
  \phantomsection
  #1\def\@currentlabel{\unexpanded{#1}}\label{#2}%
}
\makeatother

\newcommand*{\tree}{\mathfrak{t}}
\newcommand*{\ptree}[1][]{(\tree#1,\ell)}
\newcommand*{\FPT}{\mathcal{FT}}
\newcommand{\note}[1]{} % turn off notes

\renewcommand*{\Re}{\mathfrak{Re}}
\renewcommand*{\Im}{\mathfrak{Im}}
\newcommand*{\supp}{\mathrm{supp}\,}

\newcommand{\pictoSymb}[2][]{\mathchoice{%
    \raisebox{-1pt}{\includegraphics[#1,scale=1.0]{#2}}%
  }{\raisebox{-1pt}{\includegraphics[#1,scale=1.0]{#2}}%
  }{\raisebox{-1pt}{\includegraphics[#1,scale=0.7]{#2}}%
  }{\raisebox{-1pt}{\includegraphics[#1,scale=0.5]{#2}}
}}
\newcommand*{\pacman}{\pictoSymb[page=1]{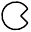}}

\DeclareRobustCommand*{\droplet}{\pictoSymb[page=1]{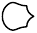}}

\newcommand*{\disk}{\mathbb{D}}
\newcommand*{\cdisk}{\overline{\mathbb{D}}\!\,}
\newcommand*{\V}{\mathbb{V}}
\newcommand*{\cV}{\overline{\mathbb{V}}\!\,}
\newcommand*{\tear}{\droplet}
\newcommand*{\ctear}{\overline{\droplet}\!\,}
\newcommand*{\ddom}{\pacman}
\newcommand*{\cdom}{\overline{\pacman}\!\,}
\newcommand*{\marg}{\epsilon}
\newcommand*{\ang}{\delta}

\newcommand*{\cone}[1][\ang]{K\!\,_{#1}}
\newcommand*{\Cone}[1][\ang]{\overline K\!\,_{#1}}

\newcommand*{\cst}{\mathtt C}

\usepackage{slashed,cancel}
\newcommand{\partials}{\slashed{\partial}}

\usepackage{esint} % includes \oiint

\usepackage[titles]{tocloft}
\title{Enumeration of {fully parked trees}}
\author{Linxiao~Chen}

\newcommand{\Addresses}{{% additional braces for segregating \footnotesize
  \bigskip\footnotesize

\noindent  \textsc{ETH Z\"urich, Department of Mathematics, R\"amistrasse 101, 8092 Z\"urich, Switzerland}\par\nopagebreak
  \textit{E-mail address}: \texttt{linxiao.chen@math.ethz.ch}	\medskip
}}

\begin{document}

\maketitle

\begin{abstract}
We enumerate a class of \emph{fully parked trees}.
In a probabilistic context, this means computing the partition function $F(x,y)$ of the parking process where an i.i.d.\ number of cars arrives at each vertex of a Galton-Watson tree with a geometric offspring distribution, conditioned to have no unoccupied vertex in the end.
The variables $x$ and $y$ count the number of vertices in the tree and the number of cars exiting from the root, respectively.

For any car arrival distribution $\mathbf b$, we obtain an explicit parametric expression of $F(x,y)$ in terms of the probability generating function $B(y)$ of $\mathbf b$. 
We show that the model has a \emph{generic} phase where the singular behavior of $F(x,y)$ is essentially independent of $B(y)$, and a \emph{non-generic} phase where it depends sensitively on the singular behavior of $B(y)$. The non-generic phase is further divided into two cases, which we call \emph{dilute} and \emph{dense}. We give a simple algebraic description of the phase diagram, and, under mild additional assumptions on $\mathbf b$, carry out detailed singularity analysis of $F(x,y)$ in the generic and the dilute phases.
The singularity analysis uses the classical transfer theorem, as well as its generalization for bivariate asymptotics. In the process, we develop a variational method for locating the dominant singularity of the inverse of an analytic function, which is of independent interest.

The phases defined in this paper are closely related to the phases in the transition of macroscopic runoff described in \cite{CurienHenard2019} and related works. The precise relation is discussed in \Cref{sec:intro/background}.
\end{abstract}

\tableofcontents

\section{Introduction}

This paper studies the exact and asymptotic enumeration of the parking configurations on a Galton-Watson tree with geometric offspring distribution, conditioned to have no unoccupied vertices at the end. However, the main definitions and results can be stated conveniently  without explicit reference to parking processes. We will proceed in this manner, and explain the context on parking processes later in \Cref{sec:intro/background}.

\subsection{Definitions and main results}\label{sec:intro/def}

We consider finite rooted plane trees. Let $V(\tree)$ denote the set of vertices of a tree $\tree$. Given a nonnegative integer labeling $\ell\!:V(\tree)\to \natural$ of the tree $\tree$, we define the \emph{surplus} of a subtree $\tree'$ as 
\begin{equation}\label{eq:def surplus}
\mathscr s\ptree['] = \sum_{v\in V(\tree')} (\ell(v)-1) \,.
\end{equation}
We say that a \emph{labeled tree} $\ptree$ is \emph{fully packed} if $\mathscr s\ptree[']\ge 0$ for all subtree $\tree'$ of $\tree$. Let $\FPT$ denote the set of all fully packed trees. 
Consider a non-negative sequence $\mathbf b=\seq bl$, encoded by the generating series $B(y) = \sum_{l=0}^\infty b_l y^l$. We assign to each labeled tree $\ptree$ a weight:
\begin{equation}\label{eq:def weight}
w_\mathbf{b} \ptree = \prod_{v\in V(\tree)} b_{\ell(v)}
\end{equation}
We define the generating function
\begin{equation}\label{eq:def F}
F(x,y) \equiv F(x,y;\mathbf{b}) = \sum_{\ptree \in \FPT} w_\mathbf{b} \ptree \cdot x^{\abs{V(\tree)}} y^{\mathscr s\ptree}
\end{equation}
Since a rooted tree contains at least the root vertex, we have $F(0,y) = 0$. 

Our first result is an explicit parametric expression of the generating function $F(x,y)$ for a general weight sequence $\seq bl$.

\begin{proposition}[Parametrzation of $F(x,y)$]\label{prop:parametrization}
For any nonnegative sequence $\mathbf b$, there exists a power series $Y\equiv \hat Y(x)$ with nonnegative coefficients such that $F(x,y)$ satisfies the following equations in the sense of formal power series:
\begin{equation}\label{eq:F:parametrization}
x=\hat x(Y) := \frac{YB(Y)}{(B(Y)+YB'(Y))^2} \qtq{and}
F(x,y)=\hat F(Y,y) := \frac12 + \frac{(Y-y)\sqrt{q(Y,y)} - \phi(Y)}{2y}
\end{equation}
where $\phi(Y) = Y\frac{B(Y)-YB'(Y)}{B(Y)+YB'(Y)}$, and $q(Y,y)=\frac{Q(Y,y)}{(Y-y)^2}$ with $Q(Y,y) = (\phi(Y)+y)^2 -4yB(y)\cdot \hat x(Y)$.
\end{proposition}

Notice that if the labeling $\ell$ is restricted to take nonzero values, then all labeled trees are full packed. Therefore when $b_0=0$, the model of fully packed trees is reduced to that of the rooted plane trees with arbitrary positive labeling.
Inversely, when $b_l=0$ for all $l\ge 2$, the condition of being fully packed forces the labeling to be $1$ on every vertex.
On the other hand, let $\rho$ be the radius of convergence of the weight generating function $B(y)$. From the functional equation \eqref{eq:F:combinatorial} on $F(x,y)$ in \Cref{sec:parametric solution}, it is not hard to see that $\rho=0$ implies $F(x,y)=\infty$ for all $x,y>0$. We will avoid these problematic cases in the following:
\begin{center}
From now on, we assume that $b_0>0$, $b_l>0$ for at least one $l\ge 2$, and $\rho>0$.
\end{center}

Our second main result is a description of the phase diagram of this model.
It is clear that the function $\hat x$ defined in \Cref{prop:parametrization} is analytic on $[0,\rho)$ and satisfies $\hat x(0)=0$ and $\hat x'(0)>0$. The singularity behavior of its inverse $\hat Y$, and hence of $F(x,y)=\hat F(\hat Y(x),y)$, depends crucially on whether $\hat x$ has a critical point on $(0,\rho)$. This motivates the following definition:

\begin{definition*}[The generic, dilute and dense\footnote{\,The names \emph{dilute} and \emph{dense} are borrowed from the terminology for the $O(n)$-loop model on random maps (see \cite{BorotBouttierDuplantier2016} and the references therein) because of similarities of the corresponding phases on the enumerative level. They are \emph{not} used to convey any geometric property of our model.} 
phases]
We say that the weight sequence $\mathbf b$ is
\begin{itemize}[itemsep=-0.5ex,topsep=0.5ex]
\item    \emph{generic} if $\hat x'(Y)=0$ for some $Y\in (0,\rho)$. In the case, let $Y_c = \min \set{Y\in (0,\rho)}{\hat x'(Y)=0}$.
\item    \emph{non-generic} if $\hat x'(Y)>0$ for all $Y\in (0,\rho)$. In this case, let $Y_c=\rho$.\\
In this case, we say that the weight sequence is \emph{dilute} if $\hat x'(\rho) := \lim_{Y\to \rho^-}\hat x'(Y) = 0$, and \emph{dense} otherwise.%\footnote{This terminology is borrowed from the phase diagram of loop-decorated maps based on some similarity on the level of critical exponents, and does not have any geometric interpretation (yet).}
\end{itemize}
~~~~In both cases, we define $x_c = \lim_{Y\to Y_c^-}\hat x(Y)$. 
\end{definition*}

\noindent It is clear that the generic, (non-generic) dilute, and (non-generic) dense phases form a partition of the phase space $\set{\mathbf b \in \real_{\ge 0}^\natural}{b_0>0,\rho>0}$. The following result gives a simpler characterization of the phases.

\begin{proposition}[Characterization of the phases]\label{prop:phase diagram}
The model is generic if $\rho=\infty$ or $B''(\rho)=\infty$. When $\rho<\infty$ and $B''(\rho)<\infty$, the model is generic (resp.\ dilute, dense) \Iff\ $\hat x'(\rho)<0$ (resp.\ $\hat x'(\rho)=0$, $\hat x'(\rho)>0$).
\end{proposition}

In \Cref{sec:intro/discussion}, we will give a probabilistic interpretation for the above result (\Cref{cor:phase diagram/probabilistic}), as well as a simple way to construct one-parameter families $(\mathbf b\0p,p\in[0,1])$ of weight sequences such that $\mathbf b\0p$ is generic, dilute, and dense when $p\in[0,p_c)$, $p=p_c$ and $p\in(p_c,1]$ respectively, for some $0<p_c<1$ (\Cref{cor:phase diagram/1D transition}).

It turns out that some weight sequences in the dilute phase will lead to the same leading order asymptotics of the coefficients of $F(x,y)$ as those in the generic phase. It is convenient to regroup them together:

\begin{definition*}[The generic$^+$ and the dilute$^-$ phases]
We say that the weight sequence $\mathbf b$ is
\begin{itemize}[itemsep=-0.5ex,topsep=0.5ex]
\item    in the \emph{generic$^+$} phase if it is either generic, or non-generic dilute with $B'''(\rho)<\infty$.
\item    in the \emph{dilute$^-$} phase if it is non-generic dilute and $B'''(\rho)=\infty$.
\end{itemize}
\end{definition*}

\Cref{fig:phase-diagram} summarizes the characterization the various phases and illustrates the relation between them.

\begin{figure}
\centering
\includegraphics[scale=1]{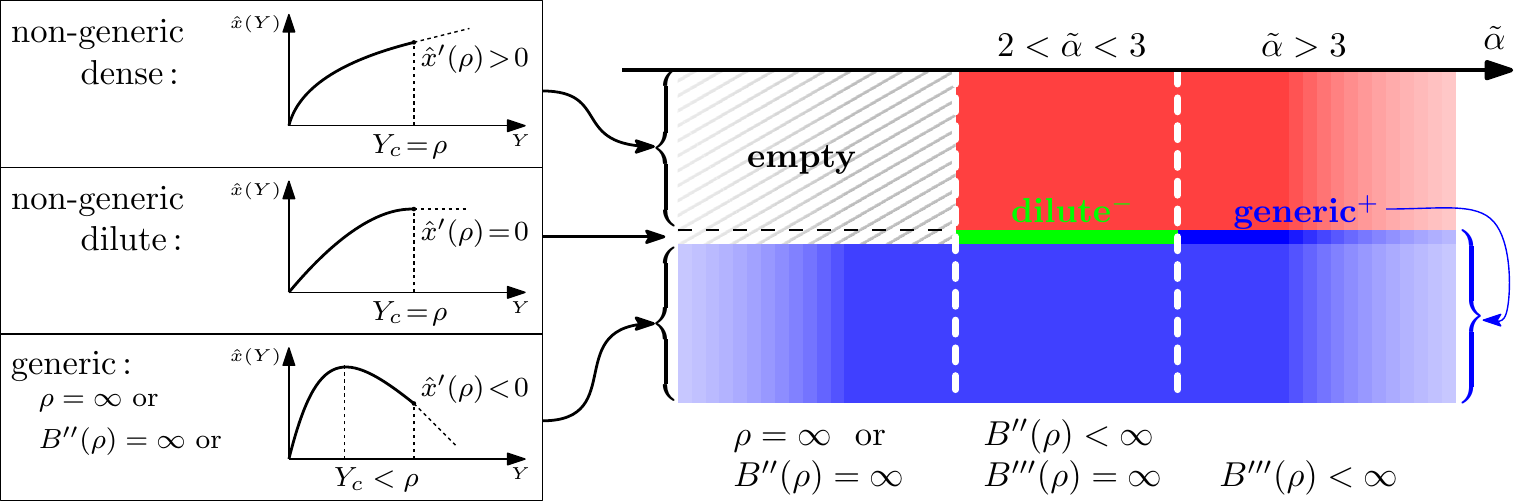}
% ? redraw diagram using the "gradient" feature of IPE ?
% It seems that the official IPE user's manual is not the only resource now: https://alga.win.tue.nl/tutorials/ipe/
\caption{Left: the characterizations of the generic, dilute and dense phases. While the graphs seem to suggest that $\hat x$ is concave, this is not necessarily true. We only prove a weaker property: see \Cref{lem:concave critical point}. Right: a phase diagram indicating the relation between generic$^+$, dilute$^-$ and the other phases of the model. The conditions at the bottom give the general defintions of the three columns of the diagram, while the axis at the top gives the corresponding ranges of $\tilde a$ under the assumption \eqref{*}.}
\label{fig:phase-diagram}
\end{figure}

\paragraph{Notations.}
For $r\in (0,\infty)$, let $\disk_r = \Set{z\in \complex}{|z|<r}$ be the disk of radius $r$ centered at $0$. We say that a function is \emph{$\Delta$-analytic at $r$} if it is analytic in a \emph{$\Delta$-domain at $r$} of the form
\begin{equation}
\ddom_r^{\marg,\ang} := \set{r z}{z\in \disk_{1+\marg},\,z\ne 1 \text{ and }\arg(z-1) \in (\pi/2 -\ang,3\pi/2 +\ang)}
\end{equation}
where $\marg>0$ and $\ang \in (0,\pi/2)$. We will write $\ddom_r$ instead of $\ddom_r^{\marg,\ang}$ when the values of $\marg$ and $\ang$ are unimportant. We denote by $\cdisk_r$ and $\cdom_r$ the closures of $\disk_r$ and $\ddom_r$, respectively.

For a formal power series $f(x)=\sum_{n\ge 0}f_n x^n$, we denote by $\supp f=\set{n\ge 0}{f_n\ne 0}$ the \emph{support} of (the coefficients of) $f$. We say that $f$ is \emph{aperiodic} if $\supp f$ is not contained in $m\integer+n$ for any $m\ge 2$ and $n\in \integer$. 
\bigskip

As we will see below, the asymptotics of the coefficients of $F(x,y)$ are mostly independent of the weight sequence $\seq bl$ in the generic phase, thus the name. On the other hand, they depend sensitively on the asymptotics of $\seq bl$ in the non-generic phase. In order to obtain interesting quantitative results on these asymptotics, we make the following additional assumptions on the weight sequence $\seq bl$:

\paragraph{Assumption (\labeltext{$*$}{*}).}~\!\!\!\!
We assume that $B(y)$ is aperiodic and $\abs{\supp B}=\infty$ (i.e.\ $B(y)$ not a polynomial). In addition,
in the non-generic phase, we assume that $B(y)$ is $\Delta$-analytic at $\rho$ and has the following asymptotic expansions when $y\to \rho$ in $\ddom_{\rho}$:
\begin{equation}\label{eq:non-generic B asymptotic}
\begin{aligned}
B  (y) &= B\1r  (y) + B\1s  (y)\cdot (1+o(1)) \,,\\
B' (y) &= B\1r' (y) + B\1s' (y)\cdot (1+o(1)) \,,\\
B''(y) &= B\1r''(y) + B\1s''(y)\cdot (1+o(1)) \,,
\end{aligned}
\end{equation}
where $B\1r$ is an analytic function at $\rho$, and 
$B\1s(y) = \cst_B\cdot (1-y/\rho)^{\tilde \alpha}$ for some $\cst_B \ne 0$ and $\tilde \alpha \in (2,\infty) \setminus \integer$.

We will discuss the significance and necessity of the above assumptions in \Cref{sec:intro/discussion}.

\bigskip

The third main result of this paper concerns the asymptotics of the coefficients $F_p(x) := [y^p]F(x,y)$ and $F_{n,p} := [x^ny^p] F(x,y)$ in various regimes of the limit $n,p\to \infty$. 
For the moment, we focus on the asymptotics in the generic and the dilute phases, because these are the two phases which are relevant in the probabilistic study of the model (see discussions in \Cref{sec:intro/background}), and they can mostly be treated in a unified way. A discussion about the asymptotics in the dense phase would create some additional hurdles, and is left to future work. 

Let $\alpha=3$ in the generic$^+$ phase, and $\alpha = \tilde \alpha \in (2,3)$ in the dilute$^-$ phase. 

\begin{theorem}[Coefficient asymptotics]\label{thm:coeff asymptotics}
Under Assumption~\eqref{*} and in the generic and the dilute phases, we have
\begin{align}\label{eq:coeff y asymptotics}
F_p(x_c) &\ \eqv{p}\ 
          \frac{\cst_F}{\Gamma (-\gamma_0)} 
          \cdot Y_c^{-p} \cdot p^{-\gamma_0-1} &
\partial_x F_p(x_c) &\ \eqv{p}\ 
                     \frac{\alpha}{2\mu x_c} 
                     \frac{\cst_F}{\Gamma(-\gamma_1)} 
                     \cdot Y_c^{-p} \cdot p^{-\gamma_1-1}
\\ \label{eq:coeff 2-step asymptotics}
F_{n,p} &\ \eqv{n}\ 
         \frac{G_p}{\Gamma(-\beta_0)} 
         \cdot x_c^{-n} \cdot n^{-\beta_0-1} &
G_p &\ \eqv{p}\ 
     \frac{\alpha-1}{2\,\mu^{\beta_0}} 
     \frac{\cst_F}{\Gamma(-\beta_1)} 
     \cdot Y_c^{-p} \cdot p^{-\beta_1-1}
\\ \label{eq:coeff bivariate asymptotics}
\text{and when }&n \sim v\cdot p^{1/\theta}
\text{ for some }v \in (0,\infty): &
F_{n,p}\ &\eqv{n,p}\ 
          \mu\, C_F \cdot I_\alpha(\mu v) 
          \cdot x_c^{-n} Y_c^{-p} 
          \cdot p^{-(\gamma_0+1+1/\theta)} \,.
\end{align}
where the exponents $\gamma_0,\gamma_1,\beta_0,\beta_1$ and $\theta$ are universal (i.e.\ they only depend on $\alpha$), and are given by
\begin{align}\label{eq:exponents values}
\gamma_0=  \frac{\alpha}{2}        \,,\qquad
\gamma_1=1-\frac{\alpha}{2}        \,,\qquad
\beta_0 =  \frac{\alpha}{\alpha-1} \,,\qquad
\beta_1 = -\frac{\alpha}{2}        \qtq{and}
\theta  =  \frac{1}{\alpha-1}      \,.
\end{align}
The scaling function $I_\alpha:\real_{>0}\to \real_{>0}$ is also universal. Its expression is given (without proof) in the remark below.
The non-universal constants $\mu$ and $\cst_F$ in the expansions \eqref{eq:coeff y asymptotics}--\eqref{eq:coeff bivariate asymptotics} are given by
\begin{equation}
\mu = \begin{cases}
-\frac{Y_c^2}2 \frac{\hat x''(Y_c)}{x_c} 
             & \text{in the generic$^+$ phase} \\
 \frac{2\tilde \alpha \, \cst_B}{B(Y_c)+Y_c B'(Y_c)} 
             & \text{in the dilute$^-$ phase}
\end{cases}
\qtq{and}
\cst_F = \frac12 \sqrt{ \frac{2\mu}{\alpha \m({ 1+\frac{Y_c B'(Y_c)}{B(Y_c)} }} } \,.
\end{equation}
\end{theorem}

\begin{remark*}
In an upcoming paper, we will explain how to carry out singularity analysis and to compute $I_\alpha$ for a fairly large class of bivariate generating functions. By applying this method to $F(x,y)$, we obtain that
\begin{equation}
I_\alpha(\lambda) =\,\sum_{n=0}^\infty \, \frac{c_n}{\Gamma(\sigma_n-\gamma_0) \Gamma(-\theta \sigma_n)} \,\lambda^{-\theta \sigma_n-1}
\end{equation}
where the constants $\sigma_n,c_n \in \real$ are determined by $\sum \limits_{n=0}^\infty \!c_n x^{\sigma_n} = \sqrt{1 -\alpha x^{\alpha-1} +(\alpha-1) x^\alpha}$ as $x\to 0$. Or, explicitly
\begin{equation}
I_\alpha(\lambda) 
=\, \sum_{p=0}^\infty \sum_{q=1}^\infty \,
  \frac{(-1)^{q+1}}{2\sqrt \pi}
  \frac{\Gamma \m({ p+q-\frac12 } 
      }{\Gamma \m({ (\alpha-1)p+\alpha(q-\frac12) } 
        \Gamma \m({ -p -\frac{\alpha}{\alpha-1} q }  
      } \frac{\alpha^p (\alpha-1)^q}{p!\,q!}
 \, \lambda^{-(p+\frac{\alpha}{\alpha-1} q)-1} \,.
\end{equation}
Assuming the increasing order $\sigma_0<\sigma_1<\cdots$, it is an elementary exercise to show that $\limsup_{n\to \infty} |c_n|^{\frac1{\sigma_n}} <\infty$. On the other hand, by Euler's reflexion formula, we have $\frac1{\Gamma(-\theta \sigma_n)} = \frac{\sin(\pi \theta \sigma_n)}{\pi} \Gamma(\theta \sigma_n+1)$. It follows that
\begin{equation}
\limsup_{n\to \infty} \abs{ \frac{c_n}{\Gamma(\sigma-\gamma_0) 
\Gamma(-\theta \sigma_n)} }^{\frac1{\sigma_n}}
\le 
\limsup_{n\to \infty} |c_n|^{\frac1{\sigma_n}} \ \cdot\
\limsup_{n\to \infty} \abs{ \frac{\Gamma(\theta \sigma_n+1)}{\Gamma(\sigma_n-\gamma_0)} }^{\frac1{\sigma_n}}
\end{equation}
Since $\theta=\frac{1}{\alpha-1}<1$, Stirling's formula implies that the second limsup on the \rhs\ is equal to zero.
By a harmless generalization of the root test, we see that the series of power functions which defines $I_\alpha$ is absolutely convergent for all $\lambda\in \complex\setminus \{0\}$.

The various asymptotic formulas in \Cref{thm:coeff asymptotics} are related to each other by a number of heuristic scaling relations. For instance, by plugging the second asymptotics of \eqref{eq:coeff 2-step asymptotics} into the first one, we see that
\begin{equation}
F_{n,p} \eqv{n,p} \frac{\alpha-1}{2} \frac{\mu\,\cst_F}{\Gamma(-\beta_0) \Gamma(-\beta_1)} \cdot x_c^{-n} Y_c^{-p} \cdot (\mu n)^{-\beta_0-1} p^{-\beta_1-1}
\end{equation}
when $n$ tends to $\infty$ sufficiently fast compared to $p$. On the other hand, the asymptotics \eqref{eq:coeff bivariate asymptotics} can be rewritten as
\begin{equation}
F_{n,p}\ \eqv{n,p}\ \mu\, C_F
       \cdot (\mu v)^{(\gamma_0-\beta_1)\theta+1}
             I_\alpha(\mu v)
       \cdot x_c^{-n} Y_c^{-p} 
       \cdot (\mu n)^{-(\gamma_0-\beta_1)\theta-1}
             p^{-\beta_1-1}
\end{equation}
when $n \sim v \cdot p^{1/\theta}$ for $v \in (0,\infty)$.
Although the two asymptotics are valid for different regimes of the limit $n,p\to \infty$, they sugguest heuristically the scaling relations $\beta_0=(\gamma_0-\beta_1)\theta$ and 
$\lim_{\lambda \to \infty} \lambda^{\beta_0+1} I_\alpha(\lambda) = \frac{\alpha-1}{2 \Gamma(-\beta_0) \Gamma(-\beta_1)}$. Both relations can be verified using the explicit expression of the exponents and of $I_\alpha$. Another heuristic scaling relation is $\gamma_1 = \gamma_0 - \frac1\theta$. It is a bit harder to explain, and will be discussed in the upcoming paper containing the derivation of the expression of $I_\alpha$. 
\end{remark*}

One last result that we would like to mention here is a variational method for finding equations which constrain the dominant singularities of the inverse of an analytic function.
It is used in the proof of \Cref{thm:coeff asymptotics}, but applies in a general setting. We explain this method in detail in \Cref{sec:variational method}.

\subsection{Discussions and corollaries}\label{sec:intro/discussion}

\paragraph{About on the technical assumption~\eqref{*}.\!}
Since we assumed $b_0>0$, the series $B(y)$ is periodic \Iff\ $\supp B\subseteq m\integer$ for some $m\ge 2$. 
A simple rewriting of the definition of the surplus gives that
\begin{equation}\label{eq:def surplus equiv}
\sum_{v\in V(\tree)} \ell(v) = |V(\tree)|+\mathscr s \ptree \,.
\end{equation}
So if $\supp B\subseteq m\integer$, then all fully packed tree $\ptree$ such that $|V(\tree)|+\mathscr s \ptree \not \in m\integer$ would have zero weight. This would cause complications in the asymptotic analysis of the coefficients of $F(x,y)$, which we prefer to avoid. 
Both the aperiodicity of $B(y)$ and the assumption $|\supp B|=\infty$ are only used in \Cref{sec:technical lemma} to prove the uniqueness of dominant singularity of the series $\hat Y=\hat x^{-1}$. The above discussion shows that the aperiodicity is necessary for that conclusion to be true. On the other hand, we do not believe that the condition $|\supp B|=\infty$ is necessary. But currently we do not have a proof that bypasses it.

The assumptions in the non-generic phase contain two parts: First, we assume that $B$ is $\Delta$-analytic and that the expansions \eqref{eq:non-generic B asymptotic} hold in $\ddom_\rho$. This is necessary for having the corresponding $\Delta$-analyticity and asymptotic expansion in $\ddom_\rho$ of the function $y\mapsto F(x,y)$. Our asymptotic analysis of $F_{n,p}$ relies heavily on these ingredients.
Second, we assume that the dominant singular term $B\1s(y)$ in the asymptotic expansion of $B$ is a power function. While our method is applicable to more general $B\1s(y)$ (e.g.\ power function with logarithmic corrections), allowing such terms would greatly complicate the singularity analysis of $F(x,y)$ with little benefit. So we choose not to do so.
Finally, remark that the assumption $\tilde \alpha>2$ is not restrictive, since according to \Cref{prop:phase diagram}, we must have $B''(\rho)<\infty$ in the non-generic phase.

\paragraph{Random fully packed trees and equivalent weight sequences.\!}
When $F(x,y;\mathbf{b})<\infty$, we can define a probability measure on $\FPT$ by
\begin{equation}\label{eq:def random FTP}
\prob^{x,y}_\mathbf{b}\ptree = \frac{ w_\mathbf{b}\ptree \cdot x^{\abs{V(\tree)}} y^{\mathscr s \ptree} }{ F(x,y;\mathbf b) } .
\end{equation}
The definition of $w_{\mathbf b}$ and the relation \eqref{eq:def surplus equiv} imply that for any $\lambda,r>0$, the weight sequence $\tilde b_l = \lambda r^l \cdot b_l$ satisfies
\begin{equation}
F(x,y; \mathbf{\tilde b}) = F(\lambda r\cdot x,\,r\cdot y;\,\mathbf{b})
\qtq{and}
\prob^{x,y}_\mathbf{\tilde b} = \prob^{\lambda r \cdot x,r \cdot y}_{\mathbf b} \,.
\end{equation}
That is, the weight sequences $\mathbf{\tilde b}$ and $\mathbf b$ define the same family of probability measures up to a change of indices. We say that they are \emph{equivalent}. Alternatively, two weight sequences are equivalent \Iff\ $\tilde B(y) = \lambda B(r y)$ for some $\lambda,r>0$.
It is not hard to see that equivalent weight sequences always belong to the same phase.

\paragraph{Probabilistic characterization of the phases.\!}
Under the assumptions $\rho>0$ and $b_0>0$, every weight sequence has an equivalent in exactly one of the three categories: $(\rho=\infty$ and $B(1)=1)$, $(\rho=1$ and $B(1)=\infty)$, or $\rho=1=B(1)$.
\Cref{prop:phase diagram} ensures that the weight sequences in the first two categories are always in the generic phase. On the other hand, each weight sequence satisfying $\rho=1=B(1)$ defines a probability measure on $\{0,1,2,\cdots\}$ \emph{with sub-exponential tail} (in the sense that $b_n=o(r^n)$ and $b_n\ne O(r^{-n})$ as $n\to \infty$ for all $r>1$). Moreover, this probability distribution has a finite second moment \Iff\ $B''(1)<\infty$, and when this is the case, the condition $\hat x'(1)<0$ simplies to $2(B''(1)+B'(1)-B'(1)^2) + B'(1)^2 > 1$. This gives the following probabilistic reformulation of \Cref{prop:phase diagram}.

\begin{corollary}\label{cor:phase diagram/probabilistic}
Assume that $\mathbf b\equiv \seq bl$ is a probability distribution on $\{0,1,2,\cdots \}$ with a sub-exponential tail.
If $\,\mathbf b$ has infinite second moment, the it is in the generic phase. Otherwise, it is in the generic (resp.\ dilute, dense) phase \Iff\ \,$2\sigma^2 + m^2>1$ (resp.\ $=1$, $<1$), where $m$ and $\sigma^2$ are the mean and the variance of $\,\mathbf b$.
\end{corollary}

Using \Cref{cor:phase diagram/probabilistic}, one can easily construct examples of weight sequences in each of the three phases, or a continuous family of weight sequences that passes through the generic, dilute and dense phases consecutively. We leave the reader to verify the following particular construction.

\begin{corollary}\label{cor:phase diagram/1D transition}
Let $\mathbf b\00$ and $\mathbf b\01$ be two probability distributions on $\{0,1,2,\cdots\}$ with sub-exponential tails such that $\mathbf b\00$ is generic and $\mathbf b\01$ is dense. Then there exists $p_c\in (0,1)$, such that the weight sequence $\mathbf b\0p:= (1-p)\mathbf b\00 + p \mathbf b\01$ is generic, dilute and dense when $p\in [0,p_c)$, $p=p_c$ and $p\in (p_c,1]$, respectively.
\end{corollary}

If we drop the condition of sub-exponential tail (that is, $\rho=1$) in \Cref{cor:phase diagram/probabilistic}, then the sign of $2\sigma^2+m^2-1$ and the phase of $\mathbf b$ no longer determine each other. However, not all combinations of these two properties are possible: 
For any $\mathbf b$ representing a probability distribution, we have $\rho\ge 1$.
By definition, $\hat x'(Y)>0$ for all $Y\in [0,\rho)$ in the non-generic phase, which implies that either $\hat x'(1)>0$, or $\rho=1$ and $\hat x'(0)=0$ and $\mathbf b$ is dilute. 
After simple rearrangements, the previous statement is equivalent to: for any probability distribution $\mathbf b$ on $\natural$\,:

\begin{itemize}[topsep=2pt,itemsep=0pt]
\item
If $2\sigma^2+m^2<1$, that is, $\hat x'(1)>0$, then $\mathbf b$ can be in the generic, dilute or dense phase.
\item
If $2\sigma^2+m^2=1$, that is, $\hat x'(1)=0$, then $\mathbf b$ can be in the generic or the dilute phase.
\item
If $2\sigma^2+m^2>1$ (or if $B''(1)=\infty$), then $\mathbf b$ can only be in the generic phase.
\end{itemize}
It is also worth noting that while all the other five cases can be realized by a probability distribution with exponential tail (that is, $\rho>1$), we must have $\rho=1$ to realize the case where $2\sigma^2+m^2=1$ and $\mathbf b$ is dilute.
We will explain in the next subsection the significance of the above observations in the context of the phase transition of parking processes on trees.

\subsection{Motivation and background}\label{sec:intro/background}

\paragraph{Fully packed trees as fully parked trees.\!}
This work is motivated by the following interpretation of labeled  trees as the initial configuration of a \emph{parking process} on trees: Given a labeled tree $\ptree$, we view each vertex $v\in V(\tree)$ as a parking spot that can accommodate at most one car. The label $\ell(v)$ represents the number of cars that arrive at the vertex $v$. The parking process assumes that each car attempts to park at its vertex of arrival, and if that vertex is occupied, travels towards the root until it finds an unoccupied vertex. 
If all vertices on its way are occupied, then the car exits the tree through the root. The final configuration of the parking process is encoded by the function $\chi:V(t)\to \natural$, where $\chi(v)$ is the total number of cars that visited the vertex $v$ (either parking there, or passing by) after all the cars have either parked or exited the tree. An important observation is that $\chi$ does not depend on the order in which one chooses to park the cars. Indeed, one can check that $\chi$ satisfies the recursion relation
\begin{equation}
\chi(v) = \ell(v) + \sum_{u \in \mathfrak C_v} (\chi(u)-1)_+ 
\end{equation}
where $\mathfrak C_v$ denotes the set of children of the vertex $v$, and $(x)_+ \!\equiv\! \max(x,0)$ is the positive part of a real number~$x$. Since the tree is finite, the above recursion relation completely determines $\chi$.
Another presentation of the final configuration consists of recording whether each vertex $v$ is occupied at the end of the parking process, and the \emph{flux} $\varphi(v)$ of cars that went from $v$ to its parent vertex during the process. 
The relation between the two presentations is simple: a vertex $v$ is occupied at the end \Iff $\,\chi(v)\ge 1$, and we have $\varphi(v) = (\chi(v)-1)_+$. 
The flux of cars $\varphi(\emptyset)$ going out from the root vertex $\emptyset$ is called the \emph{overflow} of the parking process.

With a bit of thought, it is not hard to see that a labeled tree is \emph{fully packed} \Iff\ the corresponding parking configuration $\chi$ is \emph{fully parked}, that is, every vertex is occupied, or equivalently, $\chi(v)\ge 1$ for all $v$.
In~this case, the parking configuration $\chi$ and the flux $\varphi$ are related to the surplus by $\chi(v)-1 = \varphi(v) = \mathscr s(\tree_v,\ell)$, where $\tree_v$ is the subtree rooted at $v$.

\paragraph{Previous works on the parking process on trees.}
The parking problem was first introduced by Konheim and Weiss \cite{KonheimWeiss1966} to model the linear probing scheme of hash collision resolution in computer science.
% about "linear probing scheme for hash collision resolution":
% https://en.wikipedia.org/wiki/Hash_table#Open_addressing
%
%In their model, the parking takes place on a linear directed graph, i.e.\ a rooted tree consisting of a single branch.
%This model of parking process and its generalizations ... we refer to the survey \cite{Yan15} and references therein.
In their model, the parking process takes place on a directed  linear graph (i.e.\ a rooted tree with a single branch).
Parking processes on non-degenerate trees was introduced more recently by Lackner and Panholzer~\cite{LacknerPanholzer2015}, who enumerated the \emph{parking functions} on Cayley trees of size $n$. In our terminology, a parking function is an initial configuration of the parking process which produces no overflow at the root, and in which the cars are labeled from $1$ to $m$. It is represented by a function from $\{1,\ldots,m\}$ to the vertex set of the tree, thus the name. Using analytic combiantorics methods, it was shown in \cite{LacknerPanholzer2015} that when $m=\floor{\alpha n}$ labeled cars arrive independently at uniformly chosen vertices of a random Cayley tree with $n$ vertices, the probability that there is no macroscopic overflow at the root undergoes a continuous phase transition. More precisely, as $n\to\infty$, this probability converges to a continuous limit $p(\alpha)$ which is positive if $\alpha<\alpha_c$, and zero if $\alpha\ge \alpha_c$, for some $\alpha_c\in (0,1)$. This result was later generalized to many other classes of trees and to variants of the parking functions \cite{Panholzer2020}.

A probabilistic explanation of the phase transition in \cite{LacknerPanholzer2015} was given by Goldschmidt and Przykucki~\cite{GoldschmidtPrzykucki2016} using the \emph{objective method} \cite{AldousSteele2004}. Their key observation is that the parking process of \cite{LacknerPanholzer2015} has a nice local limit in distribution, and that the probability of having no macroscopic overflow at the root is continuous \wrt\ this limit. More precisely, the limit parking process lives on a \emph{Kesten's tree} (i.e.\ \emph{critical} Galton-Watson tree conditioned to survive forever, see \cite{AbrahamDelmas2015}), and an i.i.d.\ number of cars arrives at each vertex of this tree. Chen and Goldschmidt~\cite{ChenGoldschmidt2019} later used the same idea to study the parking process on uniform random rooted plane trees, which also gives rise to a limit process on a Kesten's tree, but with a different offspring distribution. This motivates the study of parking processes with i.i.d.\ car arrivals on general critical Galton-Watson trees. Interestingly, the same type of parking processes was also proposed and studied independently as a good model of rainfall runoff from hillsides, where the aforementioned phase transition is of practical importance. See Jones \cite{Jones2018} and the references therein.

In both \cite{GoldschmidtPrzykucki2016} and \cite{ChenGoldschmidt2019}, the derivation of the phase transition relies on computing explicitly the probability of macroscopic overflow in the limit model. The offspring distributions of critical Galton-Watson trees involved are Poissonian and geometric respectively, while the car arrival distribution is Poissonian in both cases.
Using a more flexible argument involving the \emph{spinal decomposition} of Galton-Watson trees \cite[Chapter 12.1]{LyonsPeres2016}, Curien and H\'enard \cite{CurienHenard2019} generalized  these phase transition results to parking processes on critical Galton-Watson trees of any offspring distribution $\nu$ and with any car arrival distribution $\mu$. They also found a simple algebraic characterization the phase transition involving only the first and second moments of $\nu$ and $\mu$. 
The results in \cite{CurienHenard2019} are later  further generalized by Contat \cite{Contat2020} to the case where the car arrival distribution may depend on the degree of the vertex, and refined with some large deviation estimates for the sharpness of the phase transition.
There has also been a recent work \cite{BahlBarnetJunge2019} focusing on parking processes with i.i.d.\ car arrivals on a \emph{supercritical} Galton-Watson tree, which makes an interesting link with the \emph{Derrida-Retaux model} \cite{DerridaShi2020}.

\paragraph{Phases of the unconditioned parking process.}
The fully parked trees in this paper are derived from a special case of the parking model of \cite{CurienHenard2019} described in the previous paragraph. More precisely, if we choose the geometric offspring distribution $\nu_k = 2^{-k-1}$ and let $\mu_k = b_k$ for the law of car arrivals, then the parking process of \cite{CurienHenard2019}, conditioned on the event that its final configuration is \emph{fully parked}, follows the law $\prob^{x,y}_{\mathbf b}$ defined in \eqref{eq:def random FTP} with $(x,y)=(1/4,1)$.
(Notice that when $x<1/4$, the model of fully parked trees in this paper is derived from a parking process with i.i.d.\ car arrivals on a \emph{subcritical} Galton-Watson tree with geometric offspring distribution. But we shall not pursue this link further here.)

In \cite{CurienHenard2019}, the parking process is called \emph{supercritical} if the overflow of cars at the root of a Galton-Watson tree conditioned to have $n$ vertices (without the conditioning of being fully parked) scales linearly with \mbox{$n$ as $n\to \infty$}. The process is called \emph{critical} if the overflow scales sublinearly but is unbounded, and \emph{subcritical} if it is bounded.
(Other descriptions of the phase transition are also available in \cite{CurienHenard2019}.) In our special case of geometric offspring distribution, the characterization of phase transition in \cite{CurienHenard2019} simplifies to:

\begin{citetheorem}[\cite{CurienHenard2019}]
~\hspace{-2mm}The model is subcritical (resp.\ critical, supercritical) \Iff\ $2\sigma^2+m^2<\!1$ (resp.\
$=\!1$, $>\!1$), where $m$ and $\sigma^2$ denote the mean and the variance of the car arrival distribution.
\end{citetheorem}

\noindent
Comparing this result to \Cref{cor:phase diagram/probabilistic}, we see that when the law $\mathbf b$ of car arrivals has a subexponential tail, the subcritical, critical, and supercritical phases described in \cite{CurienHenard2019} correspond precisely to the generic, dilute, and dense phases defined in this paper. When $\mathbf b$ does not have a subexponential tail, the possible combinations of the two notions of phase are given in the discussion after \Cref{cor:phase diagram/1D transition}.

\paragraph{Relation to upcoming works, and the reason to skip the dense phase.}
Of course, the behaviors of the parking process decribed in \cite{CurienHenard2019} do not directly apply to the model conditioned to be fully parked. 
Instead, fully parked trees appear as geometric building blocks of the final configuration of an unconditioned parking process on Galton-Watson trees.
More precisely, the \emph{clusters of occupied vertices} in a such configuration are distributed according to the law of a fully parked tree with no overflow (i.e.\ $\prob^{x,0}_{\mathbf b}$ for some $x>0$, the cluster of the root requires some special treatment since it may have a nonzero overflow). The full configuration can then be generated as a multitype Galton-Watson tree whose vertices represent either an unoccupied vertex or a fully parked cluster of the original parking process. 

This decomposition will be used in an upcoming work to study the scaling limit of the parking process on Galton-Watson trees. This paper provides the necessary asymptotic enumeration results in order to understand the law of the fully parked clusters in its final configuration.
As explained in the concluding remarks of \cite{CurienHenard2019}, this scaling limit of the parking process is most interesting when the car arrival distribution $\nu=\mathbf b$ is critical. According to the discussion below \Cref{cor:phase diagram/1D transition}, the fully parked trees can only be in the generic or the dilute phases in this case. This explains why we decide to skip the dense phase at first approach: 
while interesting from a combinatorial point of view, the asymptotic enumeration of fully parked trees in the dense phase is not relevant for the study of critical parking processes.

\paragraph{Outline of sections.}
\Cref{sec:parametric solution} derives the parametrization of $F(x,y)$ in \Cref{prop:parametrization} from its combinatorial definition.
\Cref{sec:phase diagram} proves the characterization of the generic, dilute and dense phase given in \Cref{prop:phase diagram}.
\Cref{sec:algebraic properties} gathers some useful algebraic properties (\Cref{lem:algebraic properties}) of the parametrization of $F(x,y)$.
Based on these properties, \Cref{sec:func asymptotics} derives the asymptotic expansions (\Cref{prop:func asymptotics}) of $F(x,y)$, which is then used in \Cref{sec:coeff asymptotics} to prove the coefficient asymptotics in \Cref{thm:coeff asymptotics}. The proof assumes that $F(x,y)$ has a \emph{double $\Delta$-analyticity property} (\Cref{prop:Delta-analyticity}). This assumption is verified in \Cref{sec:Delta-analyticity}, with the proof of a technical lemma (\Cref{lem:technical}) being deferred to \Cref{sec:technical lemma}.
Finally, \Cref{sec:variational method,sec:generalized IFTs} 
contains some analysis results that are used in the proofs of \Cref{lem:technical,lem:def tear}.
As mentioned before, \Cref{sec:variational method} provides a variational method for finding equations which constrain the dominant singularities of the inverse of an analytic function, which is considered another main result of this paper. \Cref{sec:generalized IFTs} provides modified versions of the (analytic) inverse function theorem and implicit function theorem, in situations where the conditions of the classical versions break down.

\paragraph{Acknowledgement.}
I would like to thank Nicolas~Curien and Olivier~H\'enard for introducing the parking problem on random trees to me and for sharing the progress of their recent and ongoing works. Their insight into the probabilistic properties of the model greatly helped the formulation of this work. I am also grateful to Mireille Bousquet-M\'elou for explaining to me how the generalized kernel method can be applied to this problem, and for many other discussions. The author was affiliated to the Univerity of Helsinki during the initial stage of this work, and would like to thank the hospitality of his colleagues there. This work has been supported by the ERC Advanced Grant (QFPROBA) and Swiss National Science Foundation (SNF) Grant 175505.

\section{Parametrization of the generating function $F(x,y)$}%
\label{sec:parametric solution}
%\label{sec:F functional equation}

In this section, we first derive the following functional equation on $F(x,y)$ from the recursive decomposition of labeled trees:
\begin{equation}\label{eq:F:combinatorial}
F(x,y) = \frac{x}{y} 
\m({ \frac{B(y)}{1-F(x,y)} - \frac{b_0}{1-F_0(x)} }
\end{equation}
where $F_0(x) = F(x,0)$. Then, we solve the above equation using the \emph{generalized kernel method} explained in \cite{BousquetMelouJehanne2005} in order to deduce the parametrization of $F(x,y)$ in \Cref{prop:parametrization}.

\paragraph{Derivation of \eqref{eq:F:combinatorial}.\!\!}
Recall that $\FPT$ is the set of all \emph{fully packed trees}, i.e., labeled (rooted plane) trees $\ptree$ such that $\mathscr s\ptree['] \ge 0$ for all $\tree'\subseteq \tree$, where $\mathscr s\ptree['] := \sum_{v\in V(\tree')} (\ell(v)-1)$ is the surplus of a subtree $\tree'$. To expand the generating function of $\FPT$ using the recursive decomposition of trees, let us consider the slightly larger class 
\begin{equation}
\FPT^\dagger := \set{\ptree}{\,\mathscr s\ptree[']\ge 0 \text{ for all \emph{proper} subtree }\tree' \varsubsetneq \tree \,} .
\end{equation}
It is clear that a labeled tree belongs to $\FPT^\dagger$ \Iff\ the subtrees rooted at the children of the root vertex are all fully packed. Therefore 
\begin{equation}\label{eq:FT dagger decomposition}
\FPT^\dagger \cong (\{\emptyset\} \times \natural) \times \mathtt{SEQ}(\FPT)
\end{equation}
where $\{\emptyset\} \times \natural$ represents the root vertex $\emptyset$ with its integer label $\ell(\emptyset)$, and $\mathtt{SEQ}(\FPT)$ is the class of (finite) sequences of fully packed trees.
The $\cong$ sign denotes an equivalence of combinatorial classes, that is, there is a bijection between the two sides that preserves the vertex count $\ptree \mapsto |V(\tree)|$, the surplus $\ptree \mapsto \mathscr s\ptree$, and the weight function $\ptree \mapsto w_{\mathbf b}\ptree$. In terms of generating functions, \eqref{eq:FT dagger decomposition} translates to
\begin{equation}\label{eq:FT dagger combinatorial}
\mH[{ \ \sum_{\ptree \in \FPT^\dagger} w_{\mathbf b}\ptree \cdot x^{|V(\tree)|} y^{\mathscr s\ptree} =: \ }\  F^\dagger(x,y)= \frac{xB(y)}{y} \cdot \frac{1}{1-F(x,y)}
\end{equation}
where $\frac{xB(y)}{y}$ is the generating function of the class $\{\emptyset\} \times \natural$ with the surplus being defined by $\mathscr s(\emptyset,l) =l-1$. (We refer readers unfamiliar with the formalism to \cite[Chapter~I.2]{FlajoletSedgewick2009}.)

On the other hand, $\FPT$ is simply the subset of $\FPT^\dagger$ defined by the condition $\mathscr s\ptree \ge 0$. Moreover, the elements in its complement all satisfy $\mathscr s\ptree =-1$. Therefore $\FPT = \FPT^\dagger \setminus \setn{\ptree \in \FPT^\dagger}{\mathscr s\ptree = -1}$, or in terms of the generating function, $F(x,y) = F^\dagger(x,y) - [y^{-1}]F^\dagger(x,y)$. With \eqref{eq:FT dagger combinatorial}, this gives \eqref{eq:F:combinatorial} in the sense of formal power series.

\paragraph{Solution of \eqref{eq:F:combinatorial}.\!\!}
First, notice that the coefficient of $[x^n]$ of the \rhs\ of \eqref{eq:F:combinatorial} only depends on the coefficients of $F(x,y)$ up to order $[x^{n-1}]$ in $x$. Therefore \Cref{eq:F:combinatorial} uniquely determines the series $F(x,y)$ order by order in $x$ (with the initial condition $F(0,y)=0$). 

\Cref{eq:F:combinatorial} involves not only the unknown function $F(x,y)$, but also its specialization at $y=0$. Equations of this form are called \emph{equations with one catalytic variable} $y$, and a general method for solving them --- which is a generalization of the \emph{kernel method} and the \emph{quadratic method} --- is given in \cite{BousquetMelouJehanne2005}. In the following, we apply this method to solve \Cref{eq:F:combinatorial}, while keeping the presentation self-contained.

Let $\Phi(f,f_0,x,y) = \frac{B(y)}{1-f} - \frac{b_0}{1-f_0} - x^{-1} y f$. Then \Cref{eq:F:combinatorial} is equivalent to $\Phi(F(x,y), F_0(x),x,y)=0$. Its partial derivative \wrt\ $y$ gives
\begin{equation}
\partial_y F(x,y)\cdot \partial_f \Phi\mb({ F(x,y),F_0(x),x,y } 
                     + \partial_y \Phi\mb({ F(x,y),F_0(x),x,y } = 0\,.
\end{equation}
Assume that there exists a formal power series $\hat Y(x) \in \complex[[x]]$ such that
\begin{equation}\label{eq:Yhat def}
\partial_f \Phi\mb({ F(x,\hat Y(x)) , F_0(x), x, \hat Y(x) } = 0.
\end{equation}
Then the formal power series $F \equiv F(x,\hat Y(x))$, $F_0 \equiv F_0(x)$ and $Y \equiv \hat Y(x)$ must satisfy the system of equations
\begin{equation}
\partial_f \Phi (F,F_0,x,Y) = 0\,,\qquad
\partial_y \Phi (F,F_0,x,Y) = 0\,,\qtq{and}
           \Phi (F,F_0,x,Y) = 0\,,
\end{equation}
or, explicitly
\begin{equation}\label{eq:kernel method explicit}
\qquad          \frac{B(Y)}{(1-F)^2} = x^{-1} Y \,,\qquad
                \frac{B'(Y)}{1-F}    = x^{-1} F \,,\qtq{and}
\frac{B(Y)}{1-F} - \frac{b_0}{1-F_0} = x^{-1} YF\,.
\end{equation}
The first equation, which is equivalent to \eqref{eq:Yhat def}, determines the coefficients of $\hat Y(x)$ inductively starting from $\hat Y(0)=0$ in the same way that \eqref{eq:F:combinatorial} determines $F(x,y)$. This shows the existence of the series $\hat Y(x)$ assumed above. Moreover, since the expansion of $\frac{B(Y)}{(1-F)^2} \equiv \frac{B(Y)}{(1-F(x,Y))^2}$ as a power series of $x$ and $Y$ has nonnegative coefficients, the above inductive definition shows that all the coefficients of $\hat Y(x)$ are nonnegative.

One can eliminate $F=F(x,Y)$ from the system \eqref{eq:kernel method explicit}, and express $x$ and $F_0(x)$ explicitly in terms of $Y=\hat Y(x)$:
\begin{equation}\label{eq:F0 parametrization}
x     = \hat x(Y)   := \frac{Y}{B(Y)\cdot (1+\psi(Y))^2} \qtq{and}
F_0(x)= \hat F_0(Y) := 1-\frac{b_0}{B(Y)\cdot (1-\psi(Y)^2)} 
\end{equation}
where $\psi(Y) = \frac{YB'(Y)}{B(Y)}$.
Plugging these into the original combinatorial equation \eqref{eq:F:combinatorial} gives a quadratic equation for $F(x,y)$, whose two solutions are
\begin{equation}
F(x,y) = \frac12 + \frac{ \pm \sqrt{ Q(Y,y)} - \phi(Y) }{2y}
\qtq{where}
Q(Y,y):=(\phi(Y)+y)^2 - 4y B(y)\cdot \hat x(Y)
\end{equation}
and $\phi(Y) = Y \frac{1- \psi(Y)}{1+ \psi(Y)}$.
One can check directly that $Q(Y,Y)=\partial_y Q(Y,Y)=0$ for all $Y$ (see also Lemma~\refp{2}{lem:algebraic properties}). This means that $Q(Y,y)=(Y-y)^2 q(Y,y)$ for some series $q(Y,y)\in \complex[[Y]][[y]]$. Moreover, we have $q(0,0)= \frac12 \partial_y^2 Q(0,0) =1$. Therefore the square root of $Q(Y,y)$ has two analytic determinations in a neighborhood of $(0,0)$, given by $\pm (Y-y)\sqrt{q(Y,y)}$. Since $F(x,y)$ is a power series of $y$, we must choose the plus sign, which gives the parametrization of $F(x,y)$ in \Cref{prop:parametrization}.
This finishes the proof of \Cref{prop:parametrization}.

Notice that $\hat F(Y,0) = 1-b_0 \frac{\hat x(Y)}{\phi(Y)}$, and this is in agreement with the second equation in \eqref{eq:F0 parametrization}.

From now on, we make the following distinction between the notations $Y$ and $\hat Y(x)$: we treat $Y$ as an independent formal or complex variable, and treat $\hat Y(x)$ as a formal power series or complex function of $x$. Notice that $\hat Y(0)=0$, and $\hat Y$ is the functional inverse of $\hat x$ in the sense that $\hat Y(\hat x(Y))=Y$ and $\hat x(\hat Y(x))=x$ as formal power series.

\section{The phase diagram}\label{sec:phase diagram}

In this short section, we prove the characterization of the phases given in \Cref{prop:phase diagram}.
As illustrated in \Cref{fig:phase-diagram}, \Cref{prop:phase diagram} would follow almost directly from the definition of the phases if the function $\hat x$ was concave. While $\hat x$ is not always concave, the following lemma gives a weaker property  (i.e.\ local concavity at the critical points) of $\hat x$, which suffices for the proof of \Cref{prop:phase diagram}. It will also be used in the proof of \Cref{lem:x asymptotics} to show that the asymptotic expansion of $\hat x$ in the generic phase is indeed \emph{generic}.

Recall that $\hat x(Y) = \frac{YB(Y)}{(B(Y)+YB'(Y))^2}$ and the model is said to be in the generic phase if $\hat x'$ vanishes on $(0,\rho)$.

\begin{lemma}\label{lem:concave critical point}
If $\hat x'(Y)=0$ for some $Y\in (0,\rho)$, then $\hat x''(Y)<0$. When $\rho<\infty$, the same implication holds for $Y=\rho$.
\end{lemma}

\begin{proof}
Notice that $\hat x$ is a rational function of $Y$, $B(Y)$ and $B'(Y)$. Therefore $\hat x'(Y)$ depends linearly on $B''(Y)$. More precisely,
\begin{equation}\label{eq:x rational derivative}
\hat x'(Y) = \frac{(B(Y) - Y B'(Y))^2}{(B(Y) + Y B'(Y))^3} -\frac{2 Y^2 B(Y)}{(B(Y) + Y B'(Y))^3} \cdot B''(Y)
\end{equation}
By solving $B''(Y)$ from the equation $\hat x'(Y)=0$ and plugging the result into the expression of $\hat x''(Y)$, we obtain after simplification:
\begin{equation*}
\hat x''(Y) \, =\, -\, \frac{ 3 (B(Y)-Y B'(Y))^2 + 2 Y^3 B'''(Y) }{ Y B(Y) \cdot (B(Y)+Y B'(Y))^3 } \, <\, 0 \,.
\qedhere
\end{equation*}
\end{proof}

\begin{proof}[Proof of \Cref{prop:phase diagram}]
\Cref{lem:concave critical point} implies that $\hat x'$ vanishes at most once on $(0,\rho)$, and when it does, it changes sign. Therefore the model is in the generic (resp.\ dilute, dense) phase \Iff\ $\hat x(\rho^-) := \lim_{x\to \rho^-} \hat x'(Y)<0$ (resp.\ $=0$, $>0$).
This proves \Cref{prop:phase diagram} when $\rho<0$ and $B''(\rho)<\infty$, where $\hat x'(\rho)$ is well-defined and finite.

When $B''(\rho)=\infty$ but $B'(\rho)<\infty$, the expression \eqref{eq:x rational derivative} shows that $\hat x'(\rho^-)=-\infty$. When $B'(\rho)=\infty$ or $\rho=\infty$, we have $B'(\rho^-)=\infty$. (Recall that the case $B(Y) = b_0+b_1 Y$ is excluded by assumption.) This implies $\hat x(\rho^-)=0$, because $\hat x(Y) = \frac{\psi}{(1+\psi)^2} \frac{1}{B'(Y)} \le \frac{1}{4}\frac{1}{B'(Y)}$ for all $Y\in (0,\rho)$, where $\psi = \frac{YB'(Y)}{B(Y)}$. Combining the two cases, we see that $\hat x'$ vanishes at least once on $(0,\rho)$ when $B''(\rho)=\infty$ or $\rho=\infty$, so the model is in the genric phase.
\end{proof}

\section{Basic algebraic properties of the parametrization of $F(x,y)$}%
\label{sec:algebraic properties}

In this section, we gather some useful algebraic properties of the parametrization of $F(x,y)$ in \Cref{prop:parametrization}. 
All of these properties can be verified by direct computation.
However, we will provide a proof that relies as little as we can on the explicit expressions of the parametrization, with the hope that it would shed some light on the combinatorial origin of these properties.

To help organize the statement and the proof of these properties, we introduce several differential operators: 
For any function $\hat f(Y,y)$, define 
\begin{equation}
\partials_x \hat f(Y,y) := \frac{\partial_Y \hat f(Y,y)}{\hat x'(Y)} \,.
\end{equation}
This operator has a simple meaning: if a function $f(x,y)$ is parametrized by $x=\hat x(Y)$ and $f(x,y) = \hat f(Y,y)$, then its $x$-derivative is parametrized by $x=\hat x(Y)$ and $\partial_x f(x,y) = \partials_x \hat f(Y,y)$.

Now consider a function of the form 
\begin{equation}\label{eq:Rf representation}
f(Y,y) = R_f(Y,B(Y),B'(Y),\ldots,y,B(y),B'(y),\ldots)
\end{equation}
where  $R_f(Y,U_0,U_1,\ldots,y,u_0,u_1,\ldots)$ is some algebraic function that depends only on finitely many of the variables $U_k$ and $u_k$. (In practice we will only need $k\le 2$.) 
We define
\begin{equation}
\partials_{U_k} f(Y,y) := \partial_{U_k} R_f(Y,B(Y),\ldots; y,B(y),\ldots)
\end{equation}
For generic values of $Y$, $y$ and $\seq bl$, the variables $Y, B(Y), B'(Y), \ldots$, and $y, B(y), B'(y), \ldots$ are algebraically independent. Hence the representation \eqref{eq:Rf representation} of $f$ is unique, and the above definition of $\partials_{U_k}$ is non-ambiguous. We define $\partials_Y f$, $\partials_{u_k} f$ and $\partials_y f$ similarly. We have the operator relations
\begin{equation}
\partial_Y = \partials_Y + U_1 \partials_{U_0} + U_2 \partials_{U_1} + \cdots
\qtq{and}
\partial_y = \partials_y + u_1 \partials_{u_0} + u_2 \partials_{u_1} + \cdots
\end{equation}
Notice that while the operators $\partials_{U_0}$, $\partials_{U_1}$, $\ldots$ commute with each other and with $\partials_Y$, they do not commute with $\partial_Y$.
The same remark holds for the operators $\partials_{u_0}$, $\partials_{u_1}$, $\ldots$ and $\partials_y$, $\partial_y$.

\begin{lemma}[Algebraic properties of the parametrization of $F(x,y)$]~\label{lem:algebraic properties}
\begin{enumerate}[(\arabic*)]
\item
$\partials_x \phi(Y) = B(Y)+Y\cdot B'(Y)$.
\item
For all $Y$, we have~
$Q(Y,Y) = \partials_{U_1} Q(Y,Y) = \partials_x Q(Y,Y) = \partial_Y Q(Y,Y) = \partial_y Q(Y,Y) = 0$.\\
On the other hand, ~$\partial_Y \partials_x Q(Y\!,\!Y)
=-\partial_y \partials_x Q(Y\!,\!Y)
= 2\partials_x \phi(Y)$, ~so~ $\partial_Y^2 Q(Y,Y) = -\partial_y\partial_Y Q(Y,Y) = 2\phi'(Y)$.
\item
$\displaystyle q(Y,y) =
- \int_0^1 \m({ \int_0^{\lambda_1} \partial_Y \partial_y
            Q \mb({ Y+\lambda_2 (y-Y), Y+\lambda_1 (y-Y)} 
    \cdot \dd \lambda_2 } \dd \lambda_1$.
In particular, ~$q(Y,Y)= \phi'(Y)$.
\item
$ \partial_Y \mb({ (Y-y)\sqrt{q(Y,y)} }
= \frac{\partial_Y Q(Y,y)}{2(Y-y)\sqrt{q(Y,y)}}$ \,and \ 
$\partial_y \mb({ (Y-y)\sqrt{q(Y,y)} }
= -\frac{\partial_y Q(Y,y)}{2(Y-y)\sqrt{q(Y,y)}}$.
\end{enumerate}
\end{lemma}

\begin{proof}
\begin{enumerate}[(\arabic*)]
\item
By definition, we have $\phi = Y\frac{1-\psi}{1+\psi}$ and $\hat x=\frac{Y}{(1+\psi)^2 B}$ with $\psi=\frac{YB'}{B}$. 
By comparing the logarithmic derivatives
\begin{equation}
\frac{\phi'}{\phi} = \frac{1}{Y} - \frac{2}{1-\psi^2} \cdot \psi'
\qtq{and}
\frac{\hat x'}{\hat x} = \frac{1-\psi}{Y} - \frac{2}{1+\psi} \cdot \psi'\,,
\end{equation}
we see that $\partials_x \phi = \frac{\phi}{(1-\psi) \hat x} = B+YB'$.

\item
Notice that $Q(Y,y)$ is the discriminant of the quadratic equation \eqref{eq:F:combinatorial} satisfied by $\hat F(Y,y) = F(\hat x(Y),y)$. 
When applying the generalized kernel method in \Cref{sec:parametric solution}, we have seen that this quadratic equation and its derivative share the same solution $F\equiv \hat F(Y,Y)$ if $y$ is set to $Y$. It follows that $Q(Y,Y)\equiv 0$.

By differentiating the function $\Delta(Y) := Q(Y,Y) \equiv 0$, we see that
\begin{equation}\label{eq:Q diagonal derivatives}
\Delta'(Y)=\partials_Y Q(Y,Y) +\partials_y Q(Y,Y) \equiv 0 
\qtq{and}
\partials_{U_1} \Delta(Y)= \partials_{U_1} Q(Y,Y) + \partials_{u_1} Q(Y,Y) \equiv 0
\,.
\end{equation}
Since $Q(Y,y)$ is independent of the variable $u_1=B'(y)$, the second identity is reduced to $\partials_{U_1}Q(Y,Y)=0$. 
By plugging the operator relation $\partial_Y = \partials_Y + U_1 \partials_{U_0} + \cdots$ into the identity
$\partial_Y Q = \partials_x Q \cdot \hat x'$, we get
\begin{equation}
(\partials_Y + U_1 \partials_{U_0} + U_2 \partials_{U_1}) Q = \partials_x Q \cdot (\partials_Y + U_1 \partials_{U_0} + U_2 \partials_{U_1}) \hat x \,.
\end{equation}
According Point (1), $\partials_x Q(Y,y) = 2(\phi(Y)+y)\cdot \partials_x \phi(Y)+ 4yB(y)$ is independent of the variable $U_2=B''(Y)$. 
Hence we can extract the coefficient of $U_2$ from both sides of the equation, which gives $\partials_{U_1} Q = \partials_x Q \cdot \partials_{U_1} \hat x$. 
Since $\partials_{U_1} \hat x$ is not identically zero, we must have $\partials_x Q(Y,Y)\equiv 0$.
It follows that $\partial_Y Q(Y,Y) = \partials_x Q(Y,Y) \cdot \hat x'(Y) \equiv 0$. Then, the first identity of \eqref{eq:Q diagonal derivatives} shows that $\partial_y Q(Y,Y) \equiv 0$ as well.

The total derivative of the identity $\partials_x Q(Y,Y) \equiv 0$ gives that $\partial_Y \partials_x Q(Y,Y) = - \partial_y \partials_x Q(Y,Y)$ for all $Y$. 
Thanks to the general formula $\partial_y \partials_x Q(Y,y) = 2 \partials_x \phi(Y) - 4 \partials_x \phi(y)$, we have
$\partial_y \partials_x Q(Y,Y) = -2\partials_x \phi(Y)$.
Finally, the formula $\partial_Y^2 Q(Y,Y) = -\partial_y \partial_Y Q(Y,Y) = 2\phi'(Y)$ follows from the previous one by the definition of $\partials_x$. 

\item
Thanks to the identities $Q(Y,Y) = \partial_Y Q(Y,Y)\equiv 0$, we have
\begin{align*}
&- \int_0^1 \m({ \int_0^{\lambda_1} \partial_Y \partial_y
            Q \mb({ Y+\lambda_2 (y-Y), Y+\lambda_1 (y-Y)} 
    \cdot \dd \lambda_2 } \dd \lambda_1 \\
=&\ \int_0^1 \frac{ 
          \partial_Y Q(Y,Y\!+\!\lambda_1 (y-Y)) - 
          \cancel{ \partial_Y Q(Y \!+\! \lambda_1 (y-Y} 
                     \cancel{), Y \!+\! \lambda_1 (y-Y)) } 
     }{y-Y} \dd \lambda_1 
= \frac{Q(Y,y)- \cancel{Q(Y,Y)} }{(Y-y)^2}
= q(Y,y) \,.
\end{align*}

\item
The first identity is simply the correct analytic branch of the formula $\partial_Y \sqrt{Q(Y,y)} = \frac{\partial_Y Q(Y,y)}{2\sqrt{Q(Y,y)}}$. It can be obtained by dividing $\partial_Y Q(Y,y) = 2(Y-y)q(Y,y) + (Y-y)^2 \partial_Y q(Y,y)$ \,by\, $2(Y-y)\sqrt{q(Y,y)}$. The second identity is proved similarly. \qedhere
\end{enumerate}
\end{proof}

\begin{remark*}
Most of the above proof does not rely on the explicit expression of parametrization of $F(x,y)$.
But the precise formulas of $\hat x$ and $\phi$ are crucial for $\partials_x \phi$ to \emph{not} depend on $U_2=B''(Y)$ in Point (1). Indeed, for a generic function $\tilde \phi(Y)$ that depends on $(Y,U_0,U_1)$, the derivative $\partials_x \tilde \phi$ will depend on $U_2$\,:
\begin{equation}
\partials_x \tilde \phi \,\equiv\, \frac{\tilde \phi'}{\hat x'} \,=\, \frac{ 
    (\partials_Y + U_1 \partials_{U_0}) \tilde \phi 
           + U_2 \cdot \partials_{U_1}  \tilde \phi 
}{  (\partials_Y + U_1 \partials_{U_0}) \hat x 
           + U_2 \cdot \partials_{U_1}  \hat x    } \,.
\end{equation}
So the fact that $\partials_x \phi$, and therefore $\partials_x \hat F(Y,y)$, does not depend on $U_2$ reflects some property that is proper to our model. Combinatorially, this means that the generating function $\partial_x F(x,y)$ of fully packed trees \emph{with a distinguished vertex} also has a relatively simple expression under the parametrization $x=\hat x(Y)$.
\end{remark*}

\section{Asymptotic expansions of $F(x,y)$}%
\label{sec:func asymptotics}

In this section we compute the asymptotic expansions of $F(x,y)$ necessary for establishing the coefficient asymptotics in \Cref{thm:coeff asymptotics}. 
We start with the corresponding asymptotic expansions of $\hat x(Y)$ and $\hat F(Y,y)$, then combine them to get the desired expansions of $F(x,y)$.

To this end, we first need to locate the singularities of $F(x,y)$ that are relevant for its coefficient asymptotics. The definition of the generic and non-generic phases already hints that the values $x_c$ and $Y_c$ play a role. We will prove in \Cref{sec:Delta-analyticity} that the bivariate function $F(x,y)$ actually has a unique dominant singularity at $(x_c,Y_c)$. In this section, we take this information as granted, and focus on the asymptotics of $F(x,y)$ when $(x,y)\to (x_c,Y_c)$ in $\cdom_{x_c} \times \cdom_{Y_c}$. 
By definition, $x=x_c$ is parametrized by $Y=Y_c$. We will use the following change of variables 
\begin{equation}
x=x_c \cdot (1-s)\,,\qquad
y=Y_c \cdot (1-t)\,,\qquad    Y=Y_c \cdot (1-S)\,,
\end{equation}
so that the limit to be taken becomes $(s,t)\to (0,0)$, or $(S,t)\to (0,0)$ under the parametrization $x=\hat x(Y)$.

Recall that the $\Delta$-domain $\ddom_r\equiv \ddom_r^{\marg,\ang}$ depends on two positive parameters $\marg$ and $\ang$. But we choose to often omit them from the notation, and their values may change from one place to another.
Define the cone
\begin{equation}
\cone = \set{z\in \complex}{ z\ne 0 \text{ and } \arg(z) \in (-\pi/2-\ang, \pi/2+\ang) } \,.
\end{equation}
For $|s|$ and $|t|$ small enough, we have $x\in \cdom^{\marg,\ang}_{x_c}$ \Iff\ $s\in \Cone$, and $y\in \cdom^{\marg,\ang}_{Y_c}$ \Iff\ $t\in \Cone$.

Recall that we restrict our attention to the generic$^+$ and the dilute$^-$ phases, in both of which $\hat x'(Y_c)=0$. Recall also that we define $\alpha=3$ in the generic$^+$ phase and $\alpha=\tilde \alpha \in (2,3)$ in the dilute$^-$ phase, where $\tilde \alpha$ is the singular exponent in Assumption~\eqref{*}.

\begin{lemma}[Asymptotics of $\hat x(Y)$]%
\label{lem:x asymptotics}
When $S\to 0$ in the closed cone $\Cone$, we have
\begin{equation}
1 - \frac{\hat x(Y)}{x_c} \sim \mu \cdot S^{\alpha-1}
\qtq{and}
\frac{\hat x'(Y)}{x_c} \sim \frac{\alpha-1}{Y_c}\mu \cdot S^{\alpha-2}
\end{equation}
where $\mu=-\frac{Y_c^2}2 \frac{\hat x''(Y_c)}{x_c}$ in the generic$^+$ phase and $\mu=\frac{\tilde \alpha \cst_B}{Y_c} \cdot \frac{\partials_{U_1} \hat x(Y_c)}{x_c}$ in the dilute$^-$ phase. In both phases, $\mu>0$.
\end{lemma}

\begin{proof}
The definition of $\hat x$ can be written as $\hat x(Y)= R_{\hat x}(Y,B(Y),B'(Y))$ with $R_{\hat x}(Y,U_0,U_1) = \frac{Y U_0}{(U_0+Y U_1)^2}$. 

\textbf{In the generic$^+$ phase,} $B(Y)$ is $C^3$-continuous as $Y\to Y_c$ in $\cdom_{Y_c}$. (Indeed, $B(Y)$ is analytic at $Y_c$ in the generic case, while in the dilute but generic$^+$ case, $C^3$-continuity follows from Assumption~\eqref{*} and $B'''(\rho)<\infty$.) It follows that $\hat x$ is $C^2$-continuous at $Y_c$. Since $\hat x(Y_c)\!=\!x_c$ and $\hat x'(Y_c)\!=\!0$, the Taylor expansion of $\hat x$ gives
\begin{equation}
\hat x(Y) = x_c + \frac12 \hat x''(Y_c) \cdot (Y_c S)^2 + o(S^2) \,.
\end{equation}
That is, $1 - \frac{\hat x(Y)}{x_c} \sim \mu \cdot S^2$ with $\mu = -\frac{Y_c^2}{2} \frac{\hat x''(Y_c)}{x_c}$. Similarly, the Taylor expansion of $\hat x'(Y)$ gives $\frac{\hat x'(Y)}{x_c} \sim \frac{2}{Y_c}\mu \cdot S$.

\textbf{In the dilute$^-$ phase,} recall the decomposition $B(Y) = B\1r(Y)+B\1s(Y) + o(B\1s(Y))$ in Assumption~\eqref{*}. Let $\hat x\1r(Y)= R_{\hat x}(Y,B\1r(Y),B\1r'(Y))$. 
By expanding the rational function $R_{\hat x}$ around $(Y,B\1r(Y),B\1r'(Y))$, we get
\begin{align*}
\hat x(Y) &= \hat x\1r(Y) 
+ \partials_{U_1} \hat x\1r(Y)\cdot B\1s'(Y) + o(B\1s'(Y))
+ \partials_{U_0} \hat x\1r(Y)\cdot B\1s (Y) + o(B\1s (Y))
       \\ &= \hat x\1r(Y)
           + \partials_{U_1} \hat x\1r(Y_c)\cdot B\1s'(Y) 
           + o((Y_c-Y)^{\tilde \alpha-1})
\end{align*}
as $Y\to Y_c$. The function $\hat x\1r$ is analytic at $Y_c$. It is not hard to see that $\hat x\1r(Y_c)\!=\! \hat x(Y_c)\!=\! x_c$ and $\hat x\1r'(Y_c)\!=\! \hat x'(Y_c) \!=\! 0$. Also, we have $B\1s'(Y) = -\frac{\tilde \alpha \cst_B}{Y_c}\cdot  S^{\tilde \alpha-1}$. It follows that
\begin{equation}
\hat x(Y) = x_c - \partials_{U_1} \hat x(Y_c) \cdot \frac{\tilde \alpha \cst_B}{Y_c} \cdot S^{\tilde \alpha-1} + o(S^{\tilde \alpha-1}) \,,
\end{equation}
that is, $1-\frac{\hat x(Y)}{x_c} \sim \mu \cdot S^{\tilde \alpha-1}$ with $\mu = \frac{\tilde \alpha \cst_B}{Y_c} \cdot \frac{\partials_{U_1}\hat x(Y_c)}{x_c}$. A similar computation shows that $\frac{\hat x'(Y)}{x_c} \sim \frac{ \alpha-1}{Y_c}\mu \cdot S^{\tilde \alpha-2}$.

All of the above asymptotics are valid when $Y\to Y_c$ in $\cdom_{Y_c}$, or equivalently $S\to 0$ in $\Cone$, thanks to the domain of validity of the expansions in Assumption~\eqref{*}. 
In the generic$^+$ phase, we have $\hat x''(Y_c)<0$ by \Cref{lem:concave critical point}. In the dilute$^-$ phase, one can check that the asymptotic positivity of the coefficients $b_l=[Y^n]B(Y)$ implies that $\cst_B<0$ when $2<\alpha<3$. It follows that $\mu>0$ in both phases.
This completes the proof of the lemma.
\end{proof}

\Cref{lem:x asymptotics} is the only place where we do calculations separately in the generic$^+$ phase and the dilute$^-$ phase. From now on, the two phases will be treated in a unified way (except for a technical proof in \Cref{sec:Delta-analyticity/uniqueness} which verifies the $\Delta$-analyticity of $\hat Y= \hat x^{-1}$).

Now we turn to the asymptotic expansions of $\hat F(Y,y)$.
Since \Cref{thm:coeff asymptotics} contains both univariate and bivariate asymptotics of the coefficients of $F(x,y)$, we need both univariate and bivariate asymptotic expansions of $\hat F(Y,y)$ to derive it.
The univariate expansions are straightforward to compute, and it is not hard to see --- given our assumption~\eqref{*} on $B(y)$ --- that the dominant singular term must be of the classical \mbox{power-law type}.
In the bivariate case, the classification of dominant singular terms is much less studied.
For the particular example of $\hat F(Y,y)$,\footnote{\,--- and also for some examples related to Ising-decorated planar maps, see \cite{ChenTurunen2020} ---} it seems that the correct generalization of power functions in the multivariate world is the following concept of \emph{homogenous functions}:

\begin{definition*}
We say that a function $H$ defined on some domain $\Omega\subseteq \complex \times \complex$ is \emph{homogenous of degree $\gamma$} if for all $\sigma \in \complex \setminus \{0\}$, we have $H(\sigma z,\sigma w) = \sigma^\gamma \cdot H(z,w)$ whenever both $(z,w)$ and $(\sigma z, \sigma w)$ are in $\Omega$.
\end{definition*}

In the formula of $\hat F(Y,y)$, the only term where $Y$ and $y$ cannot be easily separated is the square root of $q(Y,y)$.
\Cref{lem:q asymptotics} below gives the homogenous function $H_\alpha$ that is asymptotic equivalent to $q$ as $(Y,y)\to (Y_c,Y_c)$. 
The next lemma (\Cref{lem:H_alpha bounds}) provides uniform bounds of $H_\alpha$ by a power function of the vector norm $\norm{(S,t)}$.

Notice that when $x$ is tied to $Y$ by $x\equiv (1-s)x_c  = \hat x(Y)$, the asymptotics of $\hat x$ in \Cref{lem:x asymptotics} becomes $s\sim \mu \cdot S^{\alpha-1}$. Therefore, at first order approximation, $x\to x_c$ in $\cdom_{x_c}$ is equivalent to $S\to 0$ in 
$\Cone^\theta := \setn{z^\theta}{z\in \Cone}$,
where $\theta=\frac{1}{\alpha-1}$ as defined in \Cref{thm:coeff asymptotics}. Hence the right domain for taking the limit $(S,t) \to (0,0)$ is $\Cone^\theta \times \Cone$.
%Notice that $\Cone^\theta$ (resp.\ $\Cone$) is a cone of angle $\frac{2}{\alpha-1}(\frac\pi2+\ang)$ (resp.\ $2(\frac\pi2+\ang)$) at the origin.

\begin{lemma}[Asymptotics of $q(Y,y)$]\label{lem:q asymptotics}
When $(S,t)\to (0,0)$ in $\Cone^\theta \times \Cone$ for some $\ang=\ang(\alpha)>0$, we have
\begin{equation}\label{eq:q asymptotics}
q(Y,y) \sim \cst_q \cdot H_\alpha(S,t)
\end{equation}
where the constant $\cst_q := 2\partials_x \phi(Y_c) \cdot \frac{\mu x_c}{\alpha Y_c}$ is positive, and $H_\alpha$ is the homogenous function of degree $\alpha-2$ defined by
\begin{equation}
H_\alpha(S,t) = \frac{t^\alpha - S^\alpha - \alpha S^{\alpha-1} (t-S)}{(t-S)^2} \,.
\end{equation}
\end{lemma}

\begin{proof}
By Lemma~\refp{4}{lem:algebraic properties}, we have $\partials_x \partial_y Q(Y_c,Y_c) = -2\partials_x \phi(Y_c) \ne 0$. Hence \Cref{lem:x asymptotics} implies that
\begin{equation}
\partial_Y \partial_y Q(Y,y) 
 \,=\,\partials_x \partial_y Q(Y,y) \cdot \hat x'(Y)
\,\eqv[Y_c]{Y}\,-2\partials_x \phi(Y_c) \cdot 
          \frac{\alpha-1}{Y_c}\mu x_c \cdot S^{\alpha-2}
 \,=\,-\alpha(\alpha-1) \cst_q \cdot S^{\alpha-2}\,.
\end{equation}
With the change of variables $S=1-\frac{Y}{Y_c}$ and $t=1-\frac{y}{Y_c}$, this gives
\begin{align*}
\partial_Y \partial_y Q\mb({ Y+\lambda_2(y-Y),
                             Y+\lambda_1(y-Y) }
  \,&=\,- \alpha(\alpha-1) \cst_q \cdot 
        \m({ 1- \frac{Y+ \lambda_2(y-Y)}{Y_c} }^{\alpha-2}
       \!\!\!\! \cdot \m({ 1+o(1) }
\\\,&=\,- \alpha(\alpha-1) \cst_q \cdot
          \mb({S+\lambda_2(t-S)}^{\alpha-2} 
        + o\m({ \norm{(S,t)}^{\alpha-2} }
\end{align*}
where $\norm{\,\cdot\,}$ is any norm on the vector space $\complex^2$, and the little-o is uniform over $(\lambda_1,\lambda_2)\in [0,1]^2$. Plug this into the integral formula of $q(Y,y)$ in Lemma~\refp{3}{lem:algebraic properties}, we get
\begin{align*}
q(Y,y) 
&= \cst_q \cdot \int_0^1 \m({
      \int_0^{\lambda_1} \alpha(\alpha-1) \,
        \mb({ S+\lambda_2(t-S) }^{\alpha-2} \dd \lambda_2 
    } \dd \lambda_1 
    + o\m({ \norm{(S,t)}^{\alpha-2} }   \\
&= \cst_q \cdot \int_0^1 \alpha \,
     \frac{ (S+\lambda_1(t-S))^{\alpha-1} - S^{\alpha-1}
         }{ t-S }
    \dd \lambda_1
    + o\m({ \norm{(S,t)}^{\alpha-2} }   \\
&= \cst_q \cdot 
   \frac{ t^\alpha - S^\alpha - \alpha S^{\alpha-1}(t-S)
       }{ (t-S)^2 }
    + o\m({ \norm{(S,t)}^{\alpha-2} } \,.
\end{align*}
Thanks to the lower bound of $H_\alpha(t,S)$ in \Cref{lem:H_alpha bounds} below, this implies the asymptotic equivalence \eqref{eq:q asymptotics}.
\end{proof}

\begin{lemma}[`` $H_\alpha(S,t) \asymp \norm{(S,t)}^{\alpha-2}$ '']\label{lem:H_alpha bounds}
For each $\alpha\in (2,3]$, there exist $\ang>0$ and $c,c'>0$ such that
\begin{equation}\label{eq:H_alpha bounds}
    c  \cdot \norm{(S,t)}^{\alpha-2} 
\le \abs{H_\alpha(S,t)} 
\le c' \cdot \norm{(S,t)}^{\alpha-2}
\end{equation}
for all $(S,t) \in \Cone^\theta \times \Cone$.
\end{lemma}

\begin{proof}
When $S=0$, we have $H_\alpha(0,t) = t^{\alpha-2}$ and \eqref{eq:H_alpha bounds} is obvious.
When $S\ne 0$, because $H_\alpha$ is homogenous of degree $\alpha-2$, \eqref{eq:H_alpha bounds} is equivalent to
\begin{equation}\label{eq:h_alpha bounds}
    c \cdot (1+|z|^{\alpha-2})
\le \abs{h_\alpha(z)} 
\le c'\cdot (1+|z|^{\alpha-2})
\end{equation}
for all $z\in \mathcal K:= \Setn{t/S}{(S,t) \in \Cone^\theta \times \Cone,\, S\ne 0}$, where $h_\alpha(z) := H_\alpha(1,z) = \frac{z^\alpha -1 -\alpha(z-1)}{(z-1)^2}$.
Due to the $z^\alpha$ term, $\mathcal K$ should be viewed as a subdomain of the universal cover of $\complex \setminus \{0\}$, completed by a single point at $0$.
Notice that $h_\alpha$ is continuous on this completed universal cover, because $\lim_{z\to 0} h_\alpha(z) = \alpha$ and $\lim_{z\to 1}h_\alpha(z) = \frac{\alpha(\alpha-1)}{2}$.

Since $h_\alpha(z) \sim z^{\alpha-2}$ when $z\to \infty$, for any $c<1<c'$, there exists $R>0$ such that \eqref{eq:h_alpha bounds} holds for all $|z|>R$. On the other hand, the continuity of $h_\alpha$ implies that it is bounded on the compact set $\Set{z\in \mathcal K}{|z|\le R}$. This proves the upper bound in \eqref{eq:h_alpha bounds}. For the lower bound, it suffices to show that $h_\alpha$ have no zeros in $\mathcal K$. 

From its definition, we see that $\mathcal K = \setn{r e^{i \tau}}{r\ge 0\,, |\tau|\le (\theta+1)(\frac\pi2 +\ang) }$. For all $\alpha\in (2,3]$, since $\theta = \frac{1}{\alpha-1}<1$, we can choose $\ang = \ang(\alpha)>0$ such that $(\theta+1)(\frac\pi2 +\ang) <\pi$. Then $\mathcal K$ is contained in $\complex \setminus (-\infty,0)$, the principal branch of the universal cover of $\complex \setminus\{0\}$.

Now let us show that $h_\alpha$ has no zero in $\complex \setminus (-\infty,0)$ for all $\alpha\in (2,3]$. 
This is clear for $\alpha=3$, \mbox{since $h_3(z)=z+2$}. 
Assume that $h_\alpha$ has a zero on $\complex \setminus (-\infty,0)$ for some $\alpha\in (2,3)$. Let $\alpha^*$ be the infimum of such $\alpha$. By definition, there exists a sequence of pairs $(z_n,\alpha_n) \in (\complex \setminus (-\infty,0)) \times (2,3)$ such that $h_{\alpha_n}(z_n)=0$ for all $n$, and $\alpha_n\searrow \alpha_*$ as $n\to \infty$. Using the equation $h_{\alpha_n}(z_n)=0$, it is not hard to see that the sequence $(|z_n|)_{n\ge 0}$ is bounded. Thus up to extracting a subsequence, we can assume that $z_n\to z_*$ as $n\to \infty$ for some $z_*$ in the closure of $\complex \setminus (-\infty,0)$. By the continuity of $(z,\alpha) \mapsto h_\alpha(z)$, we have $h_{\alpha_*}(z_*)=0$.
However, we can check $h_\alpha$ has no zero on the boundary of $\complex \setminus (-\infty,0)$: we have $h_\alpha(0)=\alpha$, and $h_\alpha(re^{\pm i\pi}) = \frac{r^\alpha e^{\pm i\alpha\pi} -1+\alpha(r+1)}{(r+1)^2} \ne 0$ for all $\alpha \in (2,3)$ and $r\in (0,\infty)$ because the imaginary part of the \lhs\ is nonzero. It follows that $z_*\in \complex \setminus (-\infty,0)$. 
In addition, we have $\alpha_*>2$ because $h_2(z)\equiv 1$. 
Since the mapping $(z,\alpha) \mapsto h_\alpha(z)$ is analytic in $z$, and jointly continuous in both variables, 
a version of the implicit function theorem (Lemma~\ref{lem:generalized ImplicitFT}) implies that there exists a continuous function $\hat z: (\alpha_*-\varepsilon,\alpha_*+\varepsilon) \to \complex \setminus (-\infty,0)$ such that $\hat z(\alpha_*)=z_*$ and $h_\alpha(\hat z(\alpha))=0$ for all $\alpha$. This contradicts the minimality of $\alpha_*$. Therefore $h_\alpha$ has no zero in $\complex \setminus (-\infty,0)$ for all $\alpha\in (2,3]$, and this concludes the proof.
\end{proof}

With the asymptotic expansions of $\hat x(Y)$ and $q(Y,y)$ in \Cref{lem:x asymptotics,lem:q asymptotics}, we can now derive the desired asymptotic expansions of $\hat F(Y,y)$ and $F(x,y)$ by elementary calculations.

\begin{lemma}[Asymptotics of $\hat F(Y,y)$]%
\label{lem:Fhat asymptotics}
Let $\cst_F = \sqrt{\cst_q}/2$.
\begin{align}\label{eq:Fhat bivariate asymptotics}
&\text{When }(S,t) \to (0,0) 
 \text{ in } \Cone^\theta \times \Cone:
&  \hat F(Y,y) - \m({ \tfrac12 - \tfrac{\phi(Y)}{2y} }
&\,\sim\, \cst_F \cdot (t-S) \sqrt{ H_\alpha(S,t) } \,.
\\ \label{eq:Fhat Y asymptotics}
&\text{When }S \to 0 \text{ in }\Cone^\theta 
 \text{ for fixed }y\in \cdom_{Y_c}:
&  \partials_x \hat F(Y,y) - \partials_x \hat F(Y_c,y) 
&\,\sim\,-Y_c \!\cdot 
         \partial_Y \partials_x \hat F(Y_c,y) \cdot S \,.
\\ \label{eq:Fhat y asymptotics}
&\text{When }t\to 0 \text{ in }\Cone:
&    \hspace{-5mm} 
  Y_c \!\cdot \partial_Y \partials_x \hat F(Y_c,y)
\,\sim\, \tfrac{\alpha \cdot \cst_F}{2\mu x_c} \!\cdot\! t^{-\frac \alpha2} \quad & \text{and}\quad
         \partials_x \hat F(Y_c,y) 
\,\sim\, \tfrac{\alpha \cdot \cst_F}{2\mu x_c}\!\cdot\! t^{1-\frac \alpha2} \,.
\end{align}
\end{lemma}

\begin{proof}
The asymptotic expansion \eqref{eq:Fhat bivariate asymptotics} follows directly from the definition \eqref{eq:F:parametrization} of $\hat F(Y,y)$ and \Cref{lem:q asymptotics}:
\begin{equation}
         \hat F(Y,y) -\mB({ \frac12 - \frac{\phi(Y)}{2y} }
\,=\,    \frac{Y-y}{2y} \sqrt{q(Y,y)} 
\,\sim\, \frac{t-S}{2} \sqrt{\cst_q H_\alpha(S,t)} 
\end{equation}
as $(S,t)\to (0,0)$ in $\Cone^\theta \times \Cone$. Relation \eqref{eq:Fhat Y asymptotics} is simply the first order Taylor expansion of $Y\mapsto \partials_x \hat F(Y,y)$ at $Y_c$.
For the two asymptotics in \eqref{eq:Fhat y asymptotics}, we first use the derivative formula $\partials_x \mb({(Y-y)\sqrt{q(Y,y)}} = \frac{\partials_x Q(Y,y)}{2(Y-y)\sqrt{q(Y,y)}}$ from Lemma~\refp{5}{lem:algebraic properties} to compute $\partials_x \hat F(Y,y)$:
\begin{equation}\label{eq:partials x Fhat}
\partials_x \hat F(Y,y) = \frac{1}{2y} \m({ 
      \frac{\partials_x Q(Y,y)}{2(Y-y) \sqrt{q(Y,y)}} -
      \partials_x \phi(Y) } \,.
\end{equation}
The same derivative formula implies that $\partial_Y \frac{1}{(Y-y)\sqrt{q(Y,y)}} \big|_{Y=Y_c} \!= - \frac{\partials_x Q(Y_c,y)}{2 \m({ (Y_c-y) q(Y_c,y)}^{3/2}} \cdot \hat x'(Y_c) = 0$. Hence the $Y$-derivative of $\partials_x F(Y,y)$ simplifies to
\begin{equation}
\partial_Y \partials_x \hat F(Y_c,y) = 
\frac{1}{2y} \m({ 
    \frac{ \partial_Y \partials_x Q(Y_c,y)
        }{ 2(Y_c-y) \sqrt{q(Y_c,y)}
        } - \partial_Y \partials_x \phi(Y_c) } \,.
\end{equation}
On the one hand, Lemma~\refp{2}{lem:algebraic properties} tells us that $\partials_x Q(Y_c,y) \,\sim\, -\partial_y \partials_x Q(Y_c,Y_c) \cdot  (Y_c-y) \,=\, 2\partial_x \phi(Y_c) Y_c \cdot t$ as $y\to Y_c$, and $\partial_Y \partials_x Q(Y_c,Y_c) = 2\partial_x \phi(Y_c)\ne 0$. 
On the other hand, \Cref{lem:q asymptotics} implies
$(Y_c-y) \sqrt{q(Y_c,y)} \sim Y_c \sqrt{\cst_q} \cdot t^{\alpha/2}$ in the special case where $S=0$.
Plugging these asymptotics into the expressions of $\partials_x \hat F(Y_c,y)$ and $\partial_Y \partials_x \hat F(Y_c,y)$ gives \eqref{eq:Fhat y asymptotics}. The expression of the constant follows from the identity $\frac{\partials_x \phi(Y_c)}{2Y_c \sqrt{\cst_q}} = \frac{\alpha \sqrt{\cst_q}}{4\mu x_c} = \frac{\alpha \cst_F}{2\mu x_c}$.
\end{proof}

Recall the definitions $\beta_0=\frac{\alpha}{\alpha-1}$, $\beta_1=-\frac{\alpha}{2}$ and $\gamma_0=\frac{\alpha}{2}$, $\gamma_1 = 1-\frac{\alpha}{2}$ from \Cref{thm:coeff asymptotics}. The following proposition translates the asymptotic expansions of $\hat F(Y,y)$ in \Cref{lem:Fhat asymptotics} to asymptotic expansions of $F(x,y)$.

\begin{proposition}[Asymptotics of $F(x,y)$]%
\label{prop:func asymptotics}
Let $F\1{reg}(x,y)\!=\! \frac12 - \frac{\phi(\hat Y(x))}{2y}$ and $F\1{hom}(S,t)\!=\!(t-S) \sqrt{ H_\alpha(S,t) }$.
\begin{align}\label{eq:F:bivariate asymptotics}
&\text{When }(x,y)\to (x_c,Y_c)
 \text{ in } \cdom_{x_c}\times \cdom_{Y_c}:&
F(x,y) &= F\1{reg}(x,y) + \cst_F \cdot
          F\1{hom} \m({ (s/\mu)^\theta,t }
      + o \m({ \|(s^\theta,t)\|^{\gamma_0} } \,.
\\ \label{eq:F:x asymptotics}
&\text{When }x\to x_c \text{ in }\cdom_{x_c}
 \text{ for fixed }y\in \cdom_{Y_c}\!:&
F(x,y) &= F(x_c,y) - \partial_x F(x_c,y)(x_c-x) 
        + G(y) \cdot s^{\beta_0} + o\mb({ s^{\beta_0} }\,.
\\\label{eq:F:y asymptotics}
&\text{When }y\to Y_c \text{ in }\cdom_{Y_c}:&
G(y) &\sim \tfrac{\alpha-1}{2\mu^{\beta_0}} \cst_F \cdot t^{\beta_1}  \!\!\qtq{and}
\partial_x F(x_c,y) \sim \tfrac{\alpha}{2\mu x_c} \cst_F \cdot t^{\gamma_1} \,,
\end{align}
where $G(y):=\frac{\mu x_c}{\beta_0 \cdot \mu^{\beta_0}} \cdot \partial_Y \partials_x \hat F(Y_c,y)$.
\end{proposition}

\begin{proof}
Under the change of variable $x=\hat x(Y)$, \Cref{eq:Fhat bivariate asymptotics} in \Cref{lem:Fhat asymptotics} reads: $F(x,y)-F\1{reg}(x,y) \sim \cst_F \cdot F\1{hom}(S,t)$.
To prove \eqref{eq:F:bivariate asymptotics}, we just need to show that the error induced when replacing $F\1{hom}(S,t)$ by $F\1{hom}((s/\mu)^\theta,t)$ is of order $o\mn({ \|(s^\theta,t)\|^{\gamma_0} }$. 
Recall from \Cref{lem:x asymptotics} that $S\sim (s/\mu)^\theta$ when $s\to 0$.
For general values $S_1,S_2\in \complex$, we have:
\begin{equation}
    \mb|{ F\1{hom}(S_1,t) - F\1{hom}(S_2,t) } \,
\le\, \abs{S_1-S_2} \cdot \!\! \sup_{S\in [S_1,S_2]} \!
                        \mb|{ \partial_S F\1{hom}(S,t) }
\,=\, \abs{S_1-S_2} \cdot \!\! \sup_{S\in [S_1,S_2]} \!
    \frac{\alpha(\alpha-1)}{2} \m|{ 
        \frac{ S^{\alpha-2} }{ \sqrt{H_\alpha(S,t)} } }
\end{equation}
When $S_1,S_2\to 0$ and $S_1/S_2\to 1$, we have $|S_1-S_2|=o(S_1)$, whereas the supremum on $[S_1,S_2]$ is bounded by a constant times $\frac{|S_1|^{\alpha-2} }{\norm{(S_1,t)}^{\alpha/2-1}}$ (the denominator is estimated using \Cref{lem:H_alpha bounds}). It follows that 
\begin{equation}
F\1{hom}(S_1,t) - F\1{hom}(S_2,t) = \frac{o(|S_1|^{\alpha-1})}{ \norm{(S_1, t)}^{\alpha/2-1}} = o(\norm{(S_1,t)}^{\alpha/2}) \,.
\end{equation}
Taking $S_1=(s/\mu)^\theta$ and $S_2=S$ in the above formula gives the necessary estimate for proving \eqref{eq:F:bivariate asymptotics}.

Under the change of variable $x=\hat x(Y)$, the asymptotics \eqref{eq:Fhat Y asymptotics} in \Cref{lem:Fhat asymptotics} reads:
\begin{equation}
\partial_x F(x,y) - \partial_x F(x_c,y)  
= - Y_c \cdot \partial_Y \partials_x \hat F(Y_c,y) 
    \cdot (s/\mu)^\theta + o\mb({ (s/\mu)^\theta }\,.
\end{equation}
Since $s = 1-\frac{x}{x_c}$ and $\beta_0 = \theta+1$, by integrating the above equation from $x$ to $x_c$, we get
\begin{equation}
F(x_c,y) - F(x,y) - \partial_x F(x_c,y) \cdot (x_c-x)
= - Y_c \cdot \partial_Y \partials_x \hat F(Y_c,y) \cdot
    \frac{\mu x_c}{\beta_0} \m({\frac s\mu}^{\beta_0}
    + o\m({s^{\beta_0}} \,.
\end{equation}
This is \eqref{eq:F:x asymptotics} after rearrangement.
Finally, \eqref{eq:F:y asymptotics} is the direct translation of \eqref{eq:Fhat y asymptotics} in  \Cref{lem:Fhat asymptotics}.
\end{proof}

\begin{remark*}
In the bivariate expansion \eqref{eq:F:bivariate asymptotics} of $F(x,y)$, the mapping $y\mapsto F\1{reg}(x,y)$ is analytic at $Y_c$, and $F\1{hom}$ is a homogenous function of degree $\gamma_0$.
These are the essential features of \eqref{eq:F:bivariate asymptotics} that will be used to prove the bivariate asymptotics \eqref{eq:coeff bivariate asymptotics} of $F_{n,p}$ in \Cref{thm:coeff asymptotics}.
\end{remark*}

\section{Proof of \Cref{thm:coeff asymptotics}}%
\label{sec:coeff asymptotics}

In this section, we prove the coefficient asymptotics stated in \Cref{thm:coeff asymptotics} under the assumption that $F(x,y)$ is $\Delta$-analytic in both variables. More precisely, we assume that $F$ has an analytic continuation in some double $\Delta$-domain $\ddom_{x_c} \times \ddom_{Y_c}$ which is continuous on the boundary, and that $y\mapsto F(x_c,y)$ and $y\mapsto \partial_x F(x_c,y)$ are analytic in $\ddom_{Y_c}$. Using \eqref{eq:F:x asymptotics}, it is not hard to see that these assumptions imply that $G(y)$ is also analytic in $\ddom_{Y_c}$. We will verify the above $\Delta$-analyticity assumptions in \Cref{sec:Delta-analyticity}.

\begin{proof}[Proof of \Cref{thm:coeff asymptotics}]
\textbf{Asymptotics of $F_p(x_c)$ and $F_p'(x_c)$.}
When $x=x_c$, the bivariate asymptotics \eqref{eq:F:bivariate asymptotics} reads $F(x_c,y) = F\1{reg}(x_c,y) + \cst_F \cdot t^{\gamma_0} + o(t^{\gamma_0})$, where $y\mapsto F\1{reg}(x_c,y)$ is analytic at $y=Y_c$. Together with the second asymptotics in \eqref{eq:F:y asymptotics}, this gives the asymptotic expansion of $y\mapsto F(x_c,y)$ and $y\mapsto \partial_x F(x_c,y)$ at their domiannt singularity $Y_c$. By assumption, these functions are analytic in $\ddom_{Y_c}$. Thus, by the classical transfer theorem:
\begin{equation}
F_p (x_c)\ \eqv{p}\ \frac{\cst_F}{\Gamma(-\gamma_0)} 
                    \cdot Y_c^{-p} \cdot p^{-\gamma_0-1}
\qtq{and}
F_p'(x_c)\ \eqv{p}\ \frac{\alpha}{2\mu x_c} 
                    \frac{\cst_F}{\Gamma(-\gamma_1)} 
                  \cdot Y_c^{-p} \cdot p^{-\gamma_1-1} \,.
\end{equation}

\paragraph{Asymptotics of $F_{n,p}$ as $n\to \infty$ for fixed $p$, and then $p\to \infty$.} 
For each $y\in \ddom_{Y_c}$, \eqref{eq:F:x asymptotics} and the $\Delta$-analyticity of $x\mapsto F(x,y)$ imply that
\begin{equation}
F\0n(y)\ \eqv{n}\ \frac{G(y)}{\Gamma(-\beta_0)} 
                  \cdot x_c^{-n} \cdot n^{-\beta_0-1} \,,
\end{equation}
where $F\0n(y) := [x^n]F(x,y) = \sum_{p=0}^\infty F_{n,p} y^p$. Dividing the above asymptotics by its special case at $y=Y_c$ gives
\begin{equation}
\frac{F\0n(y)}{F\0n(Y_c)} \cv[]n \frac{G(y)}{G(Y_c)}
\end{equation}
According to Vitali's theorem \cite[p. 624]{FlajoletSedgewick2009}, the uniform convergence of a sequence of analytic functions in a neighborhood of zero implies the convergence of each coefficient in their Taylor expansions. Therefore
\begin{equation}
[y^p]\m({ \frac{F\0n(y)}{F\0n(Y_c)} } 
\,=\, \frac{F_{n,p}}{F\0n(Y_c)}
\ \cv[]n\ \frac{G_p}{G(Y_c)}
\end{equation}
for each fixed $p$, where $G_p=[y^p]G(y)$. Multiply this by the asymptotics of $F\0n(Y_c)$, and we obtain
\begin{equation}
F_{n,p}\ \eqv{n}\ \frac{G_p}{\Gamma(-\beta_0)} 
                  \cdot x_c^{-n} \cdot n^{-\beta_0-1}
\end{equation}
for each fixed $p$. And thanks to the first asymptotics of \eqref{eq:F:y asymptotics} and the $\Delta$-analyticity of $G(y)$, we have
\begin{equation}
G_p\ \eqv{p}\ \frac{\alpha-1}{2\, \mu^{\beta_0}}
              \frac{\cst_F}{\Gamma(-\beta_1)} 
              \cdot Y_c^{-p} \cdot p^{-\beta_1-1} \,.
\end{equation}

\paragraph{Asymptotics of $F_{n,p}$ as $n,p\to \infty$ while $n \sim v \cdot p^{1/\theta}$.}
According to the Cauchy integral formula, we have
\begin{equation}
F_{n,p} = \m({ \frac{1}{2\pi i} }^2 \oiint
\frac{F(x,y)}{x^{n+1} y^{p+1}} \dd x\, \dd y \,,
\end{equation}
where the integral is performed on the product of two small circles around the origin.
Since $F$ is analytic in $\ddom_{x_c} \times \ddom_{Y_c}$ and continuous on the boundary, we can deform the contour of integration to $\partial \ddom_{x_c} \times \partial \ddom_{Y_c}$.
The contour $\partial \ddom_1$ can be decomposed into a circular part $\mathcal C := \partial \ddom_1^{\marg,\ang} \cap \partial \disk_{1+\marg}$ and a $V$-shaped part $\mathcal V := \partial \ddom_1 \setminus \mathcal C$.
For $x$ on the circular part $x_c\! \cdot \mathcal C$ of its contour, we have 
\begin{equation}
\abs{ \frac{F(x,y)}{x^{n+1} y^{p+1}} } \le \frac{ \sup_{\partial \ddom_{x_c}\times \partial \ddom_{Y_c}} |F| }{ x_c^{n+1}(1+\marg)^{n+1} Y_c^{p+1} } = x_c^{-n} Y_c^{-p} \cdot O\m({ (1+\marg)^{-n} } \,.
\end{equation}
Similarly, when $y\in Y_c\! \cdot \mathcal C$, the integrand decays exponentially fast \wrt\ $p\to \infty$. It follows that
\begin{equation}
x_c^n Y_c^p \cdot F_{n,p}
= \m({ \frac{1}{2\pi i} }^2 
  \iint_{\mathcal V \times \mathcal V}
  \frac{F(x_c u,Y_c v)}{u^{n+1} v^{p+1}}\dd u \, \dd v 
  + O\m({ (1+\marg)^{-n} } + O\m({ (1+\marg)^{-p} }
\end{equation}
when $n,p\to \infty$. Thanks to \eqref{eq:F:bivariate asymptotics}, we have
\begin{equation}
\m({ \frac{1}{2\pi i} }^2 
  \iint_{\mathcal V \times \mathcal V}
  \frac{F(x_c u,Y_c v)}{u^{n+1} v^{p+1}}\dd u \, \dd v 
= I\1{reg} + \cst_F \cdot I\1{hom} + I\1{rem}\,,
\end{equation}
where $I\1{reg}$, $I\1{hom}$ and $I\1{rem}$ are defined by replacing $F(x,y)$ in the integral on the \lhs\ by $F\1{reg}(x,y)$, $F\1{hom}((s/\mu)^\theta,t)$ and $o\mn({ \| (s^\theta,t) \|^{\gamma_0} }$, respectively. (Recall that $s:=1-\frac{x}{x_c}$ and $t:=1-\frac{y}{Y_c}$.)
Since $y\mapsto F\1{reg}(x,y)$ is analytic in a neighborhood of $Y_c$, one can deform the second component of the contour of integration of $I\1{reg}$ from $\mathcal V$ to $\mathcal C^c := \partial \disk_{1+\epsilon} \setminus \mathcal C$. Moreover, $F\1{reg}(x_c u,Y_c v)$ is bounded on $\mathcal V \times \mathcal C^c$. So the same argument as before implies that $I\1{reg} = O((1+\marg)^{-p})$. We conclude that when $n,p\to \infty$ at any speed, we have
\begin{equation}\label{eq:coeff asymptotics exp decay}
x_c^n Y_c^p \cdot F_{n,p} \ =\ \cst_F \cdot I\1{hom} + I\1{rem} + O\m({ (1+\marg)^{-n} } + O\m({ (1+\marg)^{-p} } \,.
\end{equation}

Now assume $n\sim v\cdot p^{1/\theta}$. The change of variable $u=1-s$ maps $\mathcal V$ to $\Set{s\in \partial \cone}{|s|\le \tilde{\marg}}$, where $\tilde{\marg} = O(\marg)$. Therefore
\begin{equation}
I\1{hom} \,=\,
\m({ \frac{1}{2\pi i} }^2 \iint_{(\partial \cone)^2} 
     \frac{ F\1{hom} \m({ (s/\mu)^\theta,t }
         }{ (1-s)^{n+1} (1-t)^{p+1} } 
     \idd{|s|\le \tilde{\marg}, |t|\le \tilde{\marg}}
     \dd s \, \dd t \,.
\end{equation}
Using the fact that $F\1{hom}$ is homogenous of degree $\gamma_0$, we get after the rescaling $s \leftarrow s/n$ and $t \leftarrow t/p$ :
\begin{equation}
I\1{hom} \,=\,\frac{1}{n\, p^{1+\gamma_0}} 
\m({ \frac{1}{2\pi i} }^2 \iint_{(\partial \cone)^2} 
     \frac{ F\1{hom} \m({ p\cdot (s/(\mu n))^\theta,t }
         }{ \m({1-s/n}^{n+1} \m({1-t/p}^{p+1} } 
     \idd{|s|\le \tilde{\marg}n, |t|\le \tilde{\marg}p}
     \dd s \, \dd t  \,.
\end{equation}
For $s\in \partial \cone$ and $|s|\le \tilde{\marg} n$, we have $-\Re(s) = |s|\cdot \sin \ang \le \tilde{\marg} \sin \ang \cdot n$. Then, using the estimate $\log(1+x)\ge x-x^2/2$, one can show that
$\mn|{ \m({1-s/n}^{n+1} } \ge \mn|{ 1+\frac{-\Re(s)}{n} }^n \ge \exp (c_1 \cdot |s|)$ with $c_1=\mn({1-\frac12 \tilde{\marg} \sin \ang} \sin \ang$. The same bound holds for $(1-t/p)^{p+1}$. Then it follows from the upper bound \eqref{eq:H_alpha bounds} of $H_\alpha$ that there exists $M<\infty$ such that
\begin{equation}
\abs{ \frac{ F\1{hom} \m({ p\cdot (s/(\mu n))^\theta,t }
          }{ \m({1-s/n}^{n+1} \m({1-t/p}^{p+1} } 
     \idd{|s|\le \tilde{\marg}n, |t|\le \tilde{\marg}p}  }
\ \le\ M\cdot \m({ |\Lambda s|^\theta +|t| }^{\gamma_0}
                e^{-c_1\cdot (|s|+|t|)}
\end{equation}
for all $n,p$ such that $\frac{p^{1/\theta}}{\mu n} \le \Lambda$. The \rhs\ of the abouve inequality is integrable on $(\partial \cone)^2$ and independent of $n,p$. Thus by the dominanted convergence theorem, we have
\begin{equation}
n\, p^{1+\gamma_0} \cdot I\1{hom} 
\cv[n\sim v\cdot p^{1/\theta}]{n,p}
\m({ \frac{1}{2\pi i} }^2 \iint_{(\partial \cone)^2} 
     F\1{hom} \m({ \m({\frac{s}{\mu v}}^\theta,t }
     e^{s+t} \dd s \, \dd t,
\end{equation}
which gives after simplification
\begin{equation}\label{eq:I_alpha integral formula}
I\1{hom} \sim \mu\, I_\alpha(\mu v) 
              \cdot p^{-(\gamma_0+1+1/\theta)}
\qtq{with}
I_\alpha(\lambda) := \m({ \frac{1}{2\pi i} }^2 
                     \iint_{(\partial \cone)^2} 
F\1{hom} \m({s^\theta,t} e^{\lambda s+t} \dd s\,\dd t\,.
\end{equation}
It is not hard to see that $I_\alpha(\lambda)$ is a well-defined analytic function of $\lambda$ for all $\lambda>0$. (Recall that its expression is given without proof in the remark after \Cref{thm:coeff asymptotics}.)

By definition, for all $c_0>0$, one can find $\marg>0$ such that $\abs{ o\m({ \|(s^\theta,t)\|^{\gamma_0} } } \le c_0 \cdot (|s|^\theta+|t|)^{\gamma_0}$ for all $|s|\le \tilde{\marg}$ and $|t|\le \tilde{\marg}$. Then, using similar estimates as for $I\1{hom}$, it is not hard to see that
\begin{equation}
\abs{I\1{rem}} \le c_0 \cdot \m({ 
  M \iint_{(\partial \cone)^2}
  \m({ |s|^\theta+|t| }^{\gamma_0} e^{-c_1\cdot (|s|+|t|)}
  \dd s\,\dd t 
  } \cdot p^{-(\gamma_0+1+1/\theta)} \,.
\end{equation}
The integral is independent of $\marg$. The constant $c_0$ can be made arbitrarily small by taking smaller and smaller~$\marg$. It follows that $I\1{rem} = o\mn({ p^{-(\gamma_0 +1+1/\theta)} }$. Combining this estimate with \eqref{eq:coeff asymptotics exp decay} and \eqref{eq:I_alpha integral formula}, we obtain the bivariate asymptotics \eqref{eq:coeff bivariate asymptotics} of $F_{n,p}$ when $n,p\to \infty$ and $n\sim v\cdot p^{1/\theta}$. This concludes the proof of \Cref{thm:coeff asymptotics}.
\end{proof}

\section{$\Delta$-analyticity of $F(x,y)$}%
\label{sec:Delta-analyticity}

The rest of this paper is devoted to the proof of the following $\Delta$-analyticity result used in the proof \Cref{thm:coeff asymptotics}. 

\begin{proposition}\label{prop:Delta-analyticity}
The function $F(x,y)$ has an analytic continuation in some double $\Delta$-domain $\ddom_{x_c} \times \ddom_{Y_c}$ which is continuous on the boundary. And $y\mapsto F(x_c,y)$ and $y\mapsto \partial_x F(x_c,y)$ are analytic in $\ddom_{Y_c}$. 
\end{proposition}

We prove this result under the general assumptions specified in the introduction (in particular, we still restrict ourselves to the generic and the dilute phases). The proof comes in three steps, which are organized as follows:
In \Cref{sec:Delta-analyticity/existence}, we prove that $F(x,y)$ is absolutely convergent on the double disk $\cdisk_{x_c}\times \cdisk_{Y_c}$, so in this sense $(x_c,Y_c)$ is indeed a dominant singularity of $F(x,y)$.
In \Cref{sec:Delta-analyticity/uniqueness}, we check that $(x_c,Y_c)$ is essentially the only dominant singularity, in the sense that $F(x,y)$ is analytic everywhere on the boundary of $\cdisk_{x_c}\times \cdisk_{Y_c}$ except when $x=x_c$. This part relies crucially on \Cref{lem:technical} which, in spite of its simple statement, has a quite long and technical proof. We postpone the proof of \Cref{lem:technical} to \Cref{sec:technical lemma}. 
\Cref{sec:technical lemma} make use of some analysis results in \Cref{sec:variational method,sec:generalized IFTs}, which are organized separately because they are not specific to the parking model, and is of independent interest.
Finally, in \Cref{sec:Delta-analyticity/extension}, we combine the conclusion of \Cref{sec:Delta-analyticity/uniqueness} with some asymptotic expansions from \Cref{sec:func asymptotics} to construct the global analytic continuation of $F(x,y)$ claimed in \Cref{prop:Delta-analyticity}.

\subsection{Domain of convergence of $F(x,y)$}\label{sec:Delta-analyticity/existence}

Recall from \Cref{sec:parametric solution} that $\hat Y$ is a power series with nonnegative coefficients, and is the functional inverse of $\hat x$. From the definition of $Y_c$, we see that $\hat Y$ induces a homeomorphism from $[0,x_c]$ to $[0,Y_c]$ that is analytic on $[0,Y_c)$. In particular, the series $\hat Y$ converges absolutely at $x_c$ and $\hat Y(x_c)=Y_c$. 

\begin{lemma}[Domain of convergence of $F(x,y)$]\label{lem:domain of convergence}
The power series $F(x,y)$ and $\partial_y F(x,y)$ are absolutely convergent on $\cdisk_{x_c} \times \cdisk_{Y_c}$, and $\partial_x F(x,y)$ is absolutely convergent on $\disk_{x_c} \times \cdisk_{Y_c}$.
\end{lemma}

\begin{proof}
Since $F(x,y)$ has nonnegative coefficients, to prove the lemma it suffices to show that the series $\partial_y F(x,y) + y\cdot F(x,y)$ converges absolutely at $(x_c,Y_c)$. Thanks to the parametrization of $F(x,y)$ and Lemma~\refp{4}{lem:algebraic properties}, we have
\begin{equation}
\partial_y F(\hat x(Y),y) + y\cdot F(\hat x(Y),y)
\ =\ \partial_y \hat F(Y,y) + y\cdot \hat F(Y,y)
\ =\ \frac1{2y} \cdot \frac{-\partial_y Q(Y,y)}{2(Y-y)\sqrt{q(Y,y)}} 
\,.
\end{equation}
Using the identities in Lemma~\refp{2}{lem:algebraic properties}, it is not hard to see that $q(Y,Y) = \phi'(Y)$ and $-\frac{\partial_y Q(Y,y)}{2(Y-y)} \to \phi'(Y)$ as $y\to Y$. Therefore
\begin{equation}
\partial_y F(\hat x(Y),Y) + Y\cdot F(\hat x(Y),Y) 
\ =\ \frac{1}{2Y} \cdot \sqrt{\phi'(Y)} \,.
\end{equation}
By taking $Y=\hat Y(x)$ in the above equation, we obtain that $f(x) := \partial_y F(x,\hat Y(x)) + \hat Y(x)\cdot F(x,\hat Y(x)) = \frac{ \sqrt{\phi'(\hat Y(x))} }{ 2\hat Y(x) }$. It is clear that the series $f(x)$ also has nonnegative coefficients.

By Lemma~\refp{1}{lem:algebraic properties}, we have $\phi'(Y) = (B(Y)+YB'(Y))\cdot \hat x'(Y)$, which is analytic and strictly positive on $[0,Y_c)$. Therefore $f(x)$ has an analytic continuation on $[0,x_c)$ with a finite limit at $x_c^-$. It follows that it converges absolutely at $x_c$. 
This implies that the double series $\partial_y F(x,y) + y\cdot F(x,y)$ converges absolutely at $(x_c,Y_c)$, and completes the proof of the lemma.
\end{proof}

\begin{figure}[t]
\centering
\includegraphics[scale=1]{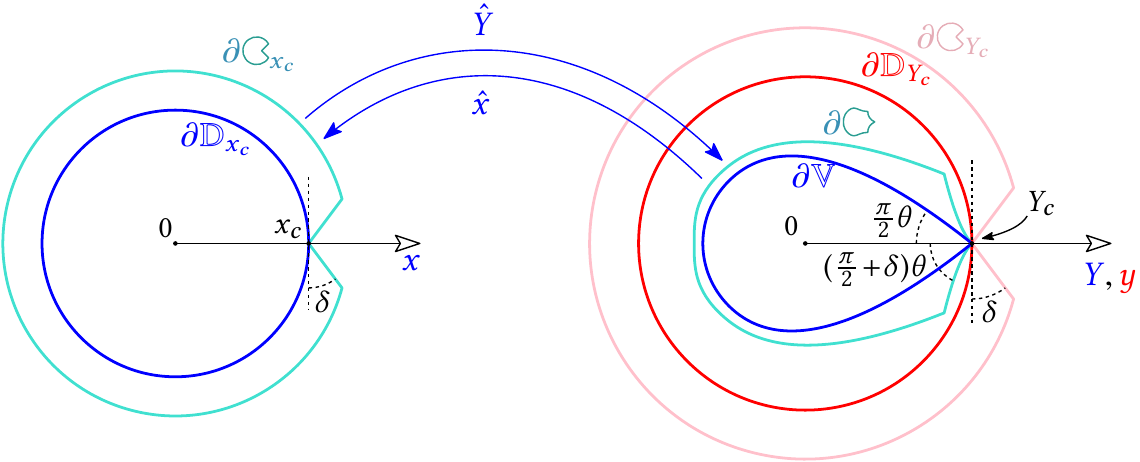}
\caption{The boundaries of various domains. The angles indicate the directions of their half tangents at $x_c$ (for $\ddom_{x_c}$) or at $Y_c$ (for $\V, \tear, \ddom_{Y_c}$). The function $\hat x$ induces a conformal bijection from $\V$ to $\disk_{x_c}$, and a conformal bijection from $\tear\equiv \tear^{\marg,\ang}$ to $\ddom_{x_c} \equiv \ddom_{x_c}^{\marg,\ang}$. Its inverse is $\hat Y$. 
}\label{fig:delta-domain-etc}
\end{figure}

\subsection{Uniqueness of dominant singularity of $F(x,y)$}\label{sec:Delta-analyticity/uniqueness}

By convention, we say that a function is \emph{holomorphic} (resp.\ \emph{meromorphic}) on an arbitrary set $D \subseteq \complex^n$ if it is continuous on $D$ and holomorphic (resp.\ meromorphic) in the interior of $D$. 
%(In the meromorphic case, the continuity should be understood as that of a function from $D$ to $\complex\cup \{\infty\}$.) 
A function is a \emph{conformal bijection} from $D$ to $D'$ if it is bijective and holomorphic on $D$, and its inverse is holomorphic on $D'$.

As the series $\hat Y$ has nonnegative coefficients and converges  at $x_c$, it defines a holomorphic function on $\cdisk_{x_c}$. Let $\V \!=\! \hat Y(\disk_{x_c})$ and $\cV \!=\! \hat Y(\cdisk_{x_c})$. It is a simple exercise to show that $\V$ is open and $\cV$ is indeed the closure of $\V$. 
%By continuity, the image $\hat Y (\cdisk_{x_c})$ is a compact set containing $\V$ and contained in $\cV$, thus it is equal to $\cV$.
\Cref{fig:delta-domain-etc} depicts the shape of $\V$ and its relation to various other domains, some of which will be defined later.
The set $\cV$ is a natural domain for the variable $Y$, in the following sense:

\begin{lemma}[Analyticity w.r.t.\ $Y\in \cV$]\label{lem:def V} We have $\cV \subseteq \cdisk_{Y_c}$, and $\hat x$ induces a conformal bijection from $\cV$ to $\cdisk_{x_c}$. 
Moreover, the function $\phi$ is holomorphic on $\cV$, and $q(Y,y)$ is holomorphic on $\cV \times \cdisk_{Y_c}$.
\end{lemma}

\begin{proof}
The series $\hat Y$ has nonnegative coefficients. Hence $|\hat Y(x)|\le \hat Y(x_c) = Y_c$ for all $x\in \cdisk_{x_c}$, that is, $\cV \subseteq \cdisk_{Y_c}$. 
By Assumption~\eqref{*}, $B(y)$ has an analytic continuation on $\cdisk_{Y_c}$ that is $C^2$-continuous at $Y_c$. Since $\hat x(Y)$ is a rational function of $Y$, $B(Y)$ and $B'(Y)$, it is well-defined and meromorphic in $\cdisk_{Y_c}$. We have $\hat x(\hat Y(x))=x$ for all $x$ in some neighborhood of $0$. Thanks to the uniqueness of analytic continuation, the same identity holds for all $x\in \cdisk_{x_c}$. It follows that $\hat Y$ is injective on $\cdisk_{x_c}$, hence defines a bijection from $\cdisk_{x_c}$ to $\cV$. This implies that its inverse $\hat x$ has no pole on $\cV$, and therefore induces a conformal bijection from $\cV$ to $\cdisk_{x_c}$. 

Recall that $\hat x(Y) \!=\! \frac{YB(Y)}{(B(Y)+YB'(Y))^2}$ and $\phi(Y) \!=\! Y\frac{B(Y)-YB'(Y)}{B(Y)+YB'(Y)}$. Like $\hat x$, the function $\phi$ is also meromorphic on $\cdisk_{Y_c} \supset \cV$. Assume that it has a pole $Y_*\in \cV$, that is, $B(Y_*)+Y_*B'(Y_*)=0$. Then $Y_*$ is at least a double zero of $(B(Y)+YB'(Y))^2$. Since $\hat x(Y_*)$ is finite, we must have $B(Y_*)=0$. This implies that $Y_*\ne 0$ and therefore $Y_*$ is a zero of the same multiplicity of $B(Y)-YB'(Y)$ and of $B(Y)+YB'(Y)$. It follows that $\phi$ is finite at $Y_*$. This contradicts the assumption that $Y_*$ is a pole of $\phi$. Hence $\phi$ is holomorphic on $\cV$.

Recall that $Q(Y,y)=(\phi(Y)+y)^2 - 4yB(y)\cdot \hat x(Y)$ and $q(Y,y) = \frac{Q(Y,y)}{(Y-y)^2}$ for $Y\ne y$. Since $\hat x$ and $\phi$ are holomorphic on $\cV$ and $B$ is holomrphic on $\cdisk_{Y_c}$, the function $q$ is holomorphic on $\cV \times \cdisk_{Y_c}$ away from the diagonal. 
For $(y,y)\ne (Y_c,Y_c)$ on the diagonal of $\cV \times \cdisk_{Y_c}$, the function $Q$ is analytic in a neighborhood of $(y,y)$. This is true even if $(y,y)$ is on the boundary of $\cV \times \cdisk_{Y_c}$, because by assumption $B$ is analytic in a $\Delta$-domain~$\ddom_{Y_c}$.
Then the integral formula of Lemma~\refp{3}{lem:algebraic properties} shows that $q$ is also analytic at $(y,y)$.
Finally, by \Cref{lem:q asymptotics} we have $q(Y,y)\to q(Y_c,Y_c)=0$ when $(Y,y)\to (Y_c,Y_c)$ in $\cV \times \cdisk_{Y_c}$. 
Hence $q$ is continuous at $(Y_c,Y_c)$ in $\cV \times \cdisk_{Y_c}$.
This shows that $q(Y,y)$ is analytic in the interior, and continuous on the boundary of $\cV \times \cdisk_{Y_c}$.
\end{proof}

The proof of the following lemma is rather long and technical, and is deferred to \Cref{sec:technical lemma}.

\begin{lemma}\label{lem:technical}
We have $\cV \setminus \{Y_c\} \subseteq \disk_{Y_c}$. Moreover,
$\hat x'(Y)\ne 0$ on $\cV \setminus \{Y_c\}$ and $q(Y,y)\ne 0$ on $\cV \times \cdisk_{Y_c} \setminus \{(Y_c,Y_c)\}$.
\end{lemma}

It is not hard to deduce from \Cref{lem:technical} that $F(x,y)$ is analytic everywhere on the boundary of $\cdisk_{x_c} \times \cdisk_{Y_c}$ except when $x=x_c$. But we shall not insist on this here, since the next subsection will provide stronger results.

\subsection{Analytic continuation of $F(x,y)$ to a double $\Delta$-domain}\label{sec:Delta-analyticity/extension}

The first part of \Cref{lem:technical} tells us that $\hat x$ is analytic and locally invertible at every point of $\cV$ except~$Y_c$. This~allows us to analytically extend its inverse $\hat Y$ to a neighborhood of each point on the circle $\partial \disk_{x_c}$ except~$x_c$. The following lemma says that we can also extend $\hat Y$ analytically to a neighborhood of $x_c$ \emph{in some $\Delta$-domain~$\ddom_{x_c}$}, and thus extend the conformal bijection $\hat x:\cV \to \cdisk_{x_c}$ to a conformal bijection onto $\cdom_{x_c}$.

%First, let us clarify its domain of definition:
%
%\begin{lemma}[Meromorphic extensions of $\hat x(Y)$\note{?}]\label{lem:dom of def of hat F}
%There exists some $\Delta$-domain $\ddom_{Y_c} \equiv \ddom_{Y_c}^{\marg,\ang}$ such that
%%\begin{enumerate}[nolistsep]
%%\item
%$\hat x$ and $\phi$ are meromorphic on $\cdom_{Y_c}$, and $q(Y,y)$ is meromorphic on $\cdom_{Y_c} \times \cdom_{Y_c}$.
%%\item
%%$\hat F(Y,y)$ is locally meromorphic on $\cdom_{Y_c} \times \cdom_{Y_c}$ except at the zeros and poles of $q(Y,y)$.
%%\end{enumerate}
%\end{lemma}
%
%\begin{proof}
%Recall that $\rho$ is the radius of convergence of $B(y)$, and $Y_c<\rho$ in the generic phase and $Y_c=\rho$ in the non-generic phase.
%By assumption, $B(y)$ has an analytic continuation on $\ddom_{Y_c}$ that is $C^2$-continuous at $Y_c$. Since $\hat x(Y)$ and $\phi(Y)$ are rational functions of $Y$, $B(Y)$ and $B'(Y)$, they are both meromorphic in $\ddom_{Y_c}$ and continuous at $Y_c$. By decreasing the values of $\marg, \ang>0$, one can make $\hat x$ and $\phi$ continuous on the boundary of $\ddom_{Y_c} \equiv \ddom^{\marg,\ang}_{Y_c}$, hence meromorphic on $\cdom_{Y_c}$. Recall that $Q(Y,y)=(\phi(Y)+y)^2-4yB(y)\cdot \hat x(Y)$. It follows that $q(Y,y)=\frac{Q(Y,y)}{(Y-y)^2}$ is also meromorphic on $\cdom_{Y_c} \times \cdom_{Y_c}$. 
%%When $(Y_*,y_*)\in \cdom_{Y_c} \times \cdom_{Y_c}$ is not a zero or a pole of $q(Y,y)$, the function $\sqrt{q(Y,y)}$ is holomorphic in a neighborhood of $(Y_*,y_*)$. Hence $\hat F(Y,y)=\frac12+\frac1{2y}\m({(Y-y)\sqrt{q(Y,y)}-\phi(Y)}$ is locally meromorphic at $(Y_*,y_*)$.
%\end{proof}

\begin{lemma}[Definition of $\ctear^{\marg,\ang}$]\label{lem:def tear}
$\hat Y$ extends to a holomorphic function on some closed $\Delta$-domain $\cdom_{x_c} \!\equiv \cdom_{x_c}^{\marg,\ang}$.\\
Let $\ctear \equiv \ctear^{\marg,\ang} \!= \hat Y(\cdom_{x_c}^{\marg,\ang})$. Then for all $\marg,\ang>0$ small enough, $\hat x$ induces a conformal bijection from $\ctear$ to $\cdom_{x_c}$. 
\end{lemma}

\begin{proof}
When $Y\to Y_c$ in $\cdisk_{Y_c}$, we have by \Cref{lem:x asymptotics}:
\begin{equation}
x_c - \hat x(Y) \sim c \cdot (Y_c-Y)^{\alpha-1}
\qtq{and}
\hat x'(Y) \sim c\cdot (\alpha-1) \cdot (Y_c-Y)^{\alpha-2}
\end{equation}
for some $c>0$. Under the change of variables $z=(Y_c-Y)^{\alpha-1}$ and $\hat w(z) = x_c - \hat x(Y)$, the above asymptotics imply that $\hat w(z) \to 0$ and $\hat w'(z) \to c$ as $z\to 0$ in the cone $K=\setn{re^{i \tau}}{r>0, |\tau|<\tau_0}$, for any $\tau_0<(\alpha-1)\frac \pi2$. 
If $\hat w$ were defined in a neighborhood of $0$, then the inverse function theorem would imply that it has a local inverse $\hat z$ such that $\hat z(w)\to 0$ and $\hat z(w)\to c^{-1}$ when $w\to 0$. In \Cref{sec:generalized IFTs}, we will show that this is still true when $\hat w$ is only defined in a cone. More precisely, \Cref{lem:generalized InverseFT} implies that for any $K'=\setn{re^{i\tau}}{ r>0,|\tau|<\tau_0'}$ with $\tau_0'<\tau_0$, there exist a neighborhood $\mathcal V$ of $0$ and an analytic function $\hat z:\mathcal V \cap K' \to K$, such that $\hat w(\hat z(w))=w$ for all $w\in \mathcal V \cap K'$. By going back through the change of variables $z=(Y_c-Y)^{\alpha-1}$ and $w=x_c-x$, we see that $\hat Y(x) = Y_c-(\hat z(x_c-x))^{\frac1{\alpha-1}}$ is an analytic continuation of $\hat Y$ on $\Set{x}{x_c-x\in \mathcal V\cap K'}$. 
Since $\alpha>2$, we can choose $\tau_0>\tau_0'>\frac\pi2$. Then the set $\Set{x}{x_c-x\in \mathcal V\cap K'}$ can be written as $\mathcal U_{x_c} \!\cap \ddom_{x_c}^{\marg,\ang}$, where $\marg>0$ is arbitrary, $\ang = \tau_0'-\frac\pi2$, and $\mathcal U_{x_c}$ is a sufficiently small neighborhood of $x_c$. 

We have constructed an analytic continuation of $\hat Y$ on $\mathcal U_{x_c} \!\cap \ddom_{x_c}^{\marg,\ang}$. On the other hand, by the remark preceding \Cref{lem:def tear}, $\hat Y$ also has an analytic continuation in a neighborhood $\mathcal U_x$ of each point $x \in \cdisk_{x_c}\setminus \{x_c\}$. When $\marg>0$ is sufficiently small, we have $\ddom_{x_c}^{\marg,\ang} \subseteq \bigcup _{x\in \cdisk_{x_c}} \!\mathcal U_x$. Then $\hat Y$ has an analytic continuation on $\ddom_{x_c}^{\marg,\ang}$. Moreover, \Cref{lem:generalized InverseFT} used in the previous paragraph also ensures that $\hat z(w)\to 0$ as $w\to 0$ in $K'$, or equivalently $\hat Y(x)\to Y_c$ as $x\to x_c$ in $\ddom_{x_c}^{\marg,\ang}$. Thus by decreasing slightly both $\marg$ and $\ang$, we may assume that $\hat Y$ is continuous on the boundary of $\ddom_{x_c}^{\marg,\ang}$, hence holomorphic on $\cdom_{x_c}^{\marg,\ang}$ in our terminology.

By the uniqueness of analytic continuation, we have $\hat x(\hat Y(x))=x$ for all $x\in \cdom_{x_c}^{\marg,\ang}$. It follows that $\hat x$ is injective on $\ctear^{\marg,\ang}:=\hat Y(\cdom_{x_c}^{\marg,\ang})$ and induces a conformal bijection from $\ctear^{\marg,\ang}$ to $\cdom_{x_c}^{\marg,\ang}$.
\end{proof}

In the proof of \Cref{lem:def V}, we deduced the holomorphicity of $\phi$ and $q(Y,y)$ on their respective domains from the holomorphicity of $\hat x$ on $\cV$. Now we know that $\hat x$ is holomorphic on the larger domain $\ctear$. The exact same argument can be used to show the following corollary. We leave the reader to check the details.

\begin{corollary}\label{cor:conformal bijection onto V/phi q holomorphic}
For $\marg,\ang>0$ small enough,
$\phi$ is holomorphic on $\ctear$ and $q(Y,y)$ is holomorphic on $\ctear \times \cdom_{Y_c}$.
\end{corollary}

%\begin{proof}
%Recall that $Q(Y,y)=(\phi(Y)+y)^2 - 4yB(y)\cdot \hat x(Y)$ and $q(Y,y) = \frac{Q(Y,y)}{(Y-y)^2}$ for $Y\ne y$. By \Cref{lem:def tear} and the remark after it, $\hat x$ and $\phi$ are holomorphic on $\ctear$. On the other hand, $B(y)$ is holomrphic on $\cdom_{Y_c}$ by assumption. It follows that $Q$ is holomorphic on $\ctear \times \cdom_{Y_c}$, and $q$ is holomorphic on $\ctear \times \cdom_{Y_c}$ away from the diagonal. Without loss of generality, we can assume that $Q$ is holomorphic on $\tear^{\marg',\ang'} \times \ddom_{Y_c}^{\marg',\ang'}$ for some $\marg'>\marg$ and $\ang'>\ang$. Observe that the open set $\tear^{\marg',\ang'} \times \ddom_{Y_c}^{\marg',\ang'}$ contains all diagonal points of $\ctear \times \cdom_{Y_c} \equiv \ctear^{\marg,\ang} \times \cdom_{Y_c}^{\marg,\ang}$ except $(Y_c,Y_c)$. Hence each of these diagonal points has a neighborhood in which $Q$ is analytic. Then the integral formula of Lemma~\refp{3}{lem:algebraic properties} shows that $q$ is also analytic at this point.
%Finally, $q$ is continuous at $(Y_c,Y_c)$ because by \Cref{lem:q asymptotics}, $q(Y,y)\to q(Y_c,Y_c)=0$ when $(Y,y)\to (Y_c,Y_c)$ in $\ctear \times \cdom_{Y_c}$. This completes the proof of the fact that $q(Y,y)$ is analytic in the interior, and continuous on the boundary of $\ctear \times \cdom_{Y_c}$, for $\marg,\ang>0$ small enough.
%\end{proof}

The second part of \Cref{lem:technical} asserts that $q(Y,y)$ has no zero on $\cV \times \cdisk_{Y_c}$ except $(Y_c,Y_c)$. By continuity, for any neighborhood $\mathcal U$ of $(Y_c,Y_c)$, there exists $\marg>0$ such that $q(Y,y)$ has no zero on $\ctear \times \cdom_{Y_c} \setminus \mathcal U$ neither. The following lemma states that for $\marg,\ang>0$ small enough, $(Y_c,Y_c)$ is actually the only zero.

\begin{lemma}[Analyticity of $\sqrt{q(Y,y)}$]\label{lem:q sqrt analytic}
For $\marg,\ang>0$ small enough, $(Y_c,Y_c)$ is the only zero of $q(Y,y)$ on $\ctear \times \cdom_{Y_c}$.
\end{lemma} 

\begin{proof}
Recall from \Cref{sec:func asymptotics} that $\cone$ is the cone $\Setn{z}{|\arg(z)|<\frac\pi2+\ang}$ and $\theta=\frac{1}{\alpha-1}$. As shown in \Cref{fig:delta-domain-etc}, the boundaries of $\cdom_{Y_c}$ and $\ctear$ each have two half tangents at $Y_c$ forming an angle of $2(\pi+\ang)$ and $2(\pi+\ang)\theta$, respectively. It follows that for any $\ang'>\ang$, there exists a neighborhood $\mathcal U_0$ of $(Y_c,Y_c)$ such that
\begin{equation}
\ctear \times \cdom_{Y_c} \cap \mathcal U_0 \ \subseteq \
\setB{(Y,y)}{ \mb({1-\frac{Y}{Y_c},1-\frac{y}{Y_c}}\in \Cone[\ang']^\theta \times \Cone[\ang'] } \,.
\end{equation}
Therefore by \Cref{lem:q asymptotics}, for $\ang$ and $\ang'$ small enough, we have $q(Y,y) \sim \cst_q \cdot H_\alpha(1-\frac{Y}{Y_c}, 1-\frac{y}{Y_c})$ when $(Y,y)\to (Y_c,Y_c)$ in $\ctear \times \cdom_{Y_c}$, where $\cst_q>0$ and $H_\alpha(S,t) = \frac{t^\alpha-S^\alpha - \alpha S^{\alpha-1} (t-S)}{(t-S)^2}$. The lower bound in \Cref{lem:H_alpha bounds} implies that $(S,t)=(0,0)$ is the only zero of $H_\alpha(S,t)$ in $\Cone[\ang']^\theta \times \Cone[\ang']$.
It follows that there exists a neighborhood $\mathcal U$ of $(Y_c,Y_c)$, such that $(Y_c,Y_c)$ is the only zero of $q$ in $\ctear \times \cdom_{Y_c} \cap \mathcal U$. 
On the other hand, by the remark preceding \Cref{lem:q sqrt analytic}, there exists $\marg>0$ such that $q$ has no zero on $\ctear \times \cdom_{Y_c} \cap \mathcal U$. It follows that for $\marg,\ang>0$ small enough, $(Y_c,Y_c)$ is the only zero of $q$ in $\ctear \times \cdom_{Y_c}$.
\end{proof}

\begin{proof}[Proof of \Cref{prop:Delta-analyticity}]
By \Cref{cor:conformal bijection onto V/phi q holomorphic} and \Cref{lem:q sqrt analytic}, the function $\hat F(Y,y) = \frac12 + \frac{1}{2y} \m({ (Y-y)\sqrt{q(Y,y)} - \phi(Y)}$ is holomorphic on $\ctear \times \cdom_{Y_c}$. (The factor $y$ in the denominator is not a problem, since we know that $\hat F(Y,y)$ defines a formal power series in $y$.)
And $\hat Y$, the inverse of $\hat x$, induces a holomorphic function from $\cdom_{x_c}$ onto $\ctear$ according to \Cref{lem:def tear}. 
It follows that $F(x,y) = \hat F(\hat Y(x),y)$ is holomorphic on (i.e.\ analytic in the interior and continuous on the boundary of) $\cdom_{x_c} \times \cdom_{Y_c}$.

When $x=x_c$, we have $F(x_c,y) = \hat F(Y_c,y)$ and $\partial_x F(x_c,y) = \partials_x \hat F(Y_c,y)$, where $\partials_x \hat F(Y,y)$ is given by \eqref{eq:partials x Fhat}. Then, thanks to the $\Delta$-analyticity of $B(y)$ and the fact that $q(Y_c,y)\ne 0$ for all $y\in \ddom_{Y_c}$, both $y\mapsto F(x_c,y) $ and $y \mapsto \partial_x F(x_c,y)$ are analytic on $\ddom_{Y_c}$.
\end{proof}

\section{Proof of \Cref{lem:technical}}\label{sec:technical lemma}

\subsection{The inclusion $\cV \setminus \{Y_c\} \subseteq \disk_{Y_c}$}\label{sec:technical lemma/aperiodicity}

Recall that $|\hat Y(x)|\le \hat Y(x_c) = Y_c$ for all $x\in \cdisk_{x_c}$ because the series $\hat Y$ has nonnegative coefficients. To prove the inclusion $\cV \setminus \{Y_c\} \subseteq \disk_{Y_c}$, it suffices to show that the inequality is strict for all $|x|\le x_c$ different from $x_c$. With a bit of thought, one sees that this is true \Iff\ the series $\hat Y$ is aperiodic.

Recall that $\supp B = \Set{l\in \natural}{b_l\ne 0}$ denotes the support of the coefficients of $B$. Since $b_0\ne 0$, the series $B$ is aperiodic \Iff\ $\supp B \not \subseteq m\integer$ for all $m\ge 2$. Similarly, since $\hat Y(0)=0$ and $\hat Y'(0)=\frac1{\hat x'(0)} \ne 0$, the series $\hat Y$ is aperiodic \Iff\ $\supp (x^{-1}\hat Y) \not \subseteq m\integer$ for all $m\ge 2$. 
The following simple lemma provides a method to relate the (a)periodicity of one formal power series to another. 
We will use it to deduce the aperiodicity of $\hat Y(x)$ from that of $B(Y)$.

\begin{lemma}[Heredity of periodicity]\label{lem:aperiodicity transfer}
If\, $\Phi: \Omega \subseteq \complex[[x]] \to \complex[[x]]$ is a mapping between formal power series such that $\Phi(S(\omega x)) = \Phi(S)(\omega x)$ for all roots of unity $\omega \!\in\! \setn{e^{i 2\pi q}}{q\in \rational}$, then for all integer  $m\ge 1$, $\,\supp S \subseteq m \integer$ \,implies\, $\supp \Phi(S) \subseteq m \integer$.
\end{lemma}

\begin{proof}
Observe that $\supp S \subseteq m \integer$ \Iff\ $S(e^{i \frac{2\pi}m} x) = S(x)$. The property of the mapping $\Phi$ ensures that if $S(e^{i \frac{2\pi}m} x) = S(x)$, then $\Phi(S)(e^{i \frac{2\pi}m} x) = \Phi(S(e^{i \frac{2\pi} m} x)) = \Phi(S)(x)$. Hence $\supp S \!\subseteq\! m \integer$ implies $\supp \Phi(S) \!\subseteq\! m \integer$.
\end{proof}

\begin{lemma}\label{lem:aperiodicity Y}
The power series $\hat Y(x)$ is aperiodic, and therefore $\cV \setminus \{Y_c\} \subseteq \disk_{Y_c}$.
\end{lemma}

\begin{proof}
The fact that $\hat x(Y)$ and $\hat Y(x)$ are inverse of each other can be written as $\hat Y(x) = x\cdot W(\hat Y(x))$, where $W(Y) = \frac{Y}{\hat x(Y)} = B(Y)\cdot \mb({ 1+ \frac{YB'(Y)}{B(Y)} }^2$. The Lagrange inversion formula states that $[x^n] \hat Y(x) = \frac{1}{n} [Y^{n-1}] W(Y)$ for all $n\ge 1$. Therefore the series $\hat Y(x)$ is aperiodic \Iff\ $W(Y)$ is. 
Moreover, since $W(0)=B(0)=b_0$ is nonzero by assumption, the series $W(Y)$ is aperiodic \Iff\ $\supp W \not \subseteq m \integer$ for all $m\ge 2$.

We know that $B(Y)$ is aperiodic. In particular, 
$\supp B \not \subseteq m \integer$ for all $m\ge 2$.
Hence to prove that $W(Y)$ is also aperiodic, it suffices to show that $\Phi:W(Y)\mapsto B(Y)$ is a well-defined mapping that satisfies the assumption of \Cref{lem:aperiodicity transfer}. The mapping $\Phi$ is well-defined if the relation between $B(Y)$ and $W(Y)$, which can be written as
\begin{equation}\label{eq:relation W->B}
B(Y) = W(Y)\cdot \m({ 1+\frac{YB'(Y)}{B(Y)} }^{-2} \,,
\end{equation}
uniquely determines the coefficients of $B(Y)$ for any given $W(Y)$. Let us prove this by induction: Write $W(Y) = \sum_{n\ge 0} w_n Y^n$ and $B(Y) = \sum_{n\ge 0} b_n Y^n$. We see easily that $b_0=w_0$. For $n\ge 1$, \Cref{eq:relation W->B} gives:
\begin{equation}
b_n = \sum_{m=0}^n w_{n-m} \cdot [Y^m] \m({ 1+\frac{YB'(Y)}{B(Y)} }^{-2} 
= \sum_{m=0}^n w_{n-m} \sum_{k=0}^\infty (-1)^k (k+1) \cdot [Y^m] \m({ \frac{Y B'(Y)}{B(Y)} }^k \,.
\end{equation}
Since $\frac{YB'(Y)}{B(Y)} = Y \frac{\hspace{12.5pt}b_1\hspace{6.0pt} + \cdots + n b_n Y^{n-1} + \cdots}{b_0 + b_1 Y + \cdots + \hspace{5.8pt} b_n Y^n \hspace{5.8pt} + \cdots}$ and $b_0\ne 0$, it is not hard to see that the coefficient $[Y^m] \m({ \frac{YB'(Y)}{B(Y)} }^k$ in the double sum is a function of $(b_0,\dots, b_{n-1})$ unless $m=n$ and $k=1$. It follows that $b_n$ can be written as
\begin{equation}
b_n = f(b_0,\ldots,b_{n-1};w_0,\ldots,w_n) - 2 w_0 \cdot [Y^n] \frac{YB'(Y)}{B(Y)} \,.
\end{equation}
Since $[Y^n] \frac{YB'(Y)}{B(Y)} = \frac{nb_n}{b_0}$ and $w_0=b_0$, the above formula gives $(1+2n) b_n = f(b_0,\ldots,b_{n-1};w_0,\ldots,w_n)$. By induction, this implies that all the coefficients $b_n$ are determined by $W(Y)$, i.e.\ the mapping $\Phi$ is well-defined.
By repalcing $B(Y)$ with $B(\omega W)$ in the definition of $W(Y)$, we see that $\Phi(W(\omega Y)) = B(\omega Y) = \Phi(W)(\omega x)$ for any $\omega \in \complex \setminus \{0\}$. This shows that $\Phi$ satisfies the assumption of \Cref{lem:aperiodicity transfer}. It follows that the series $W$, and hence $\hat Y$, is aperiodic.
As explained at the beginning of \Cref{sec:technical lemma/aperiodicity}, this implies that $\cV \setminus \{Y_c\} \subseteq \disk_{Y_c}$.
\end{proof}

\subsection{$\hat x$ has no critical point in $\cV \setminus \{Y_c\}$}\label{sec:technical lemma/unique critical point}

We have seen in \Cref{lem:def V} that $\hat x$ induces a conformal bijection from $\cV$ to $\cdisk_{x_c}$, so it has no critical point in the interior $\V$. In this subsection, we check that the same is true on the boundary $\partial \V$ except at $Y_c$. We will use a variational method which provides additional equations on the critical points of $\hat x$ on $\partial \V$ by considering perturbations of the parameters $b_k=[y^k]B(y)$. 
This variational method uses very little information on the specific function $\hat x$ and applies in a much more general setting in analytic combinatorics. For this reason, we will discuss it in full detail in \Cref{sec:variational method}. The method itself is summarized as \Cref{prop:variational method}.

We highlight the fact that this variational method can provide an additional equation by perturbing $b_k$ only if $b_k>0$. When $|\supp B|=\infty$, we obtain an infinite sequence of equations. It turns out that the asymptotics of these equations as $k\to \infty$ is quite simple, and the proof of \Cref{lem:x unique critical point} below make use of this asymptotics. 
When $|\supp B|<\infty$, our method provides only finitely many equation.
While there is still in theory enough equations for eliminating the critical points of $\hat x$ on $\partial \V \setminus\{Y_c\}$ (see \Cref{rmk:applications in AC} for a detailed count), we did not find a proof that works in general (we verified that \Cref{lem:x unique critical point} remains true when $\supp B=\{0,1,2\}$, $\{0,1,3\}$ and $\{0,2,3\}$). This is why we assumed in Assumption~\eqref{*} that $|\supp B|=\infty$. 

\begin{lemma}\label{lem:x unique critical point}
$\hat x'(Y)\ne 0$ for all $Y\in \cV \setminus \{Y_c\}$.
\end{lemma}

\begin{proof}
Assume that $\hat x$ has a critical point $Y_*\in \cV\setminus \{Y_c\}$. 
Fix an integer $k$ such that $b_k\equiv [y^k]B(y)>0$, and consider a perturbation $\varepsilon$ to the weight $b_k$. The perturbed model has a weight generating function~$B(y,\varepsilon) = B(y)+\varepsilon y^k$. Let $\hat x(Y,\varepsilon)$, $x_c(\varepsilon)$, $Y_c(\varepsilon)$ and $\V(\varepsilon)$ denote the perturbed versions of $\hat x(Y)$, $x_c$, $Y_c$ and $\V$, respectively. 

For all $\varepsilon \in \mathcal I:=(-b_k,b_k)$, the perturbed weight sequence $(b_l+\varepsilon \delta_{k,l})_{l\ge 0}$ remains nonnegative and satisfies the same assumptions (in particular, Assumption~\eqref{*}) as the non-perturbed one. Hence we can apply \Cref{lem:def V} to conclude that $\hat x(\,\cdot\,,\varepsilon)$ induces a conformal bijection from $\V(\varepsilon)$ to $\disk_{x_c(\varepsilon)}$. 
Notice that $\hat x(Y,\varepsilon)$ is meromorphic in $Y\in \disk_{x_c}$, rational in $\varepsilon$, adn finite at $(Y,\varepsilon)=(Y_*,0)$. Hence it is analytic  in an open neighborhood $\mathcal U \subseteq \complex \times \mathcal I$ of $(Y_*,0)$. \Wlg, assume that $\set{(Y,\varepsilon)}{Y\in \V(\varepsilon), \varepsilon \in \mathcal I} \subseteq \mathcal U$. Then, one can check that $\hat x(Y,\varepsilon)$, $x_c(\varepsilon)$ and $\V(\varepsilon)$ satisfy the conditions of \Cref{prop:variational method} (see also \Cref{rmk:global assumption on V}), provided that $x_c(\varepsilon)$ is differentiable at $0$.

Let us show that $x_c(\varepsilon)$ is indeed differentiable at $0$: By definition, $x_c(\varepsilon)=\hat x(Y_c(\varepsilon),\varepsilon)$, where $Y_c(\varepsilon)<\rho$ is a critical point of $\hat x(\,\cdot\,,\varepsilon)$ in the generic phase, and $Y_c(\varepsilon)=\rho$ in the non-generic phase. Recall the characterization of the phases from \Cref{prop:phase diagram}. 
\begin{itemize}
\item
If the weight generating function $B(\,\cdot\,,0)$ is in the generic phase, then so is $B(\,\cdot\,,\varepsilon)$ for all $\varepsilon$ close to zero. In this case, $Y_c(\varepsilon)$ is a critical point of $\hat x(\,\cdot\,,\varepsilon)$, which is analytic in a neighborhood of  $Y_c(\varepsilon)$. Hence we can apply \Cref{lem:critical locus}, which implies that $x_c(\varepsilon)= \hat x(Y_c(\varepsilon),\varepsilon)$ is differentiable at $0$, with $x_c'(0) = \partial_\varepsilon \hat x(Y_c(0),0)$. 

\item
If $B(\,\cdot\,,0)$ is in the non-generic dilute phase, then $\hat x(\,\cdot\,,0)$ is not analytic at $Y_c(0)=\rho$, so \Cref{lem:critical locus} no longer applies. But the proof of \Cref{lem:critical locus} can be adapted as follows: The functions $\hat x(Y,\varepsilon)$ and $\partial_Y \hat x(Y,\varepsilon)$, though not analytic, are still $C^1$ at $(Y_c(0),0)$, in particular, we have
\begin{equation}\label{eq:x dYx expansion bis}
\hat x(Y,\varepsilon) = \hat x(Y,0) + \partial_\varepsilon \hat x(Y_c(0),0) \cdot \varepsilon + o(\varepsilon)
\qtq{and}
\partial_Y \hat x(Y,\varepsilon) = \partial_Y \hat x(Y,0) + O(\varepsilon)
\end{equation} 
as $(Y,\varepsilon) \to (Y_c(0),0)$.
Let $\mathcal I_0$ be the set of values of $\varepsilon$ for which the perturbed model is in the generic phase. For $\varepsilon \in \mathcal I \setminus \mathcal I_0$, the value $Y_c(\varepsilon)=\rho$ is independent of $\varepsilon$. It follows that $\frac{ \hat x(Y_c( \varepsilon),\varepsilon) - \hat x(Y_c(0),0) }\varepsilon \to \partial_\varepsilon \hat x(Y_c(0),0)$ when $\varepsilon \to 0$ in $\mathcal I \setminus \mathcal I_0$. For $\varepsilon \in \mathcal I_0$, we have $\partial_Y \hat x(Y_c(\varepsilon),\varepsilon)=0$, hence the second expansion in \eqref{eq:x dYx expansion bis} implies that $\partial_Y \hat x(Y_c(\varepsilon),0) = O(\varepsilon)$. But from the asymptotic expansion of $\hat x(Y,0)$ and $\partial_Y \hat x(Y,0)$ in \Cref{lem:x asymptotics}, we can see that $\hat x(Y_c(0),0) - \hat x(Y,0) = o\m({ \partial_Y \hat x(Y,0) }$. Therefore we have $\hat x(Y_c(0),0) - \hat x(Y,0) = o(\varepsilon)$. Plugging this into the first expansion in \eqref{eq:x dYx expansion bis}, we obtain that $\hat x(Y_c(\varepsilon),\varepsilon) = \hat x(Y_c(0),0) + \partial_\varepsilon \hat x(Y_c(0),0)\cdot \varepsilon + o(\varepsilon)$, that is, $\frac{ \hat x(Y_c( \varepsilon),\varepsilon) - \hat x(Y_c(0),0) }\varepsilon \to \partial_\varepsilon \hat x(Y_c(0),0)$ when $\varepsilon \to 0$ in $\mathcal I_0$ as well. 
It follows that $x_c(\varepsilon) = \hat x(Y_c(\varepsilon),\varepsilon)$ is differentiable at $0$ and we have $x_c'(0) = \partial_\varepsilon \hat x(Y_c(0),0)$ as well.
\end{itemize}

We conclude that $x_c(\varepsilon)$ is indeed differentiable at $0$, and we always have $x_c'(0)=\partial_\varepsilon \hat x(Y_c(0),0)$. (This is also obviously true in the dense phase, though we do not need this fact here.) Then, \Cref{prop:variational method} states that $Y_*$ must satisfy $\Re\m({ \frac{\partial_\varepsilon \hat x(Y_*,0) }{ \hat x(Y_*,0) } } = \frac{\partial_\varepsilon \hat x(Y_c,0) }{ \hat x(Y_c,0) }$ for every $k$ such that $b_k>0$. A straightforward computation gives the explicit equation
\begin{equation}
\Re\m({ 
\frac{Y_*^k}{B(Y_*)} \frac{ 2k+1-\psi(Y_*) }{ 1+\psi(Y_*) } } = 
\frac{Y_c^k}{B(Y_c)} \frac{ 2k+1-\psi(Y_c) }{ 1+\psi(Y_c) } \,, \text{ ~where }\psi(Y) := \frac{YB'(Y)}{B(Y)}\,.
\end{equation}
In particular, we have that 
\begin{equation}
\abs{
\frac{Y_*^k}{B(Y_*)} \frac{ 2k+1-\psi(Y_*) }{ 1+\psi(Y_*) } } \ge 
\frac{Y_c^k}{B(Y_c)} \frac{ 2k+1-\psi(Y_c) }{ 1+\psi(Y_c) }
\end{equation}
By Assumption~\eqref{*}, there are infinitely many $k$ such that $b_k>0$. Hence we can take the limit $k\to \infty$ in the above inequality, which implies that $|Y_*|\ge Y_c$. But according to the previous subsection, we have $|Y_*|<Y_c$ for all $Y_*\in \cV\setminus \{Y_c\}$. Therefore $\hat x$ cannot have a critical point in $\cV\setminus \{Y_c\}$.
\end{proof}

Before moving on, let us register a useful fact whose proof uses a similar variational argument as \Cref{lem:x unique critical point}.

\begin{lemma}\label{lem:B no zero}
$B(Y)\ne 0$ and $\partials_x \phi(Y) \equiv B(Y)+YB'(Y)\ne 0$ for all $Y\in \cV$. 
\end{lemma}

\begin{proof}
Assume that $B(Y_*) = 0$ for some $Y_*\in \cV$. Since $B(0)>0$, $Y_*\ne 0$. Using the fact that $\hat x(Y)=\frac{YB(Y)}{(B(Y)+YB'(Y))^2}$ is bounded and nonzero on $\cV\setminus \{0\}$, it is not hard to see that $Y_*$ must be a zero of $B$ of multiplicity exactly 2.
Now let us show that this is impossible using the variational method:

Consider a perturbation $B(y,\varepsilon) = B(y)+\varepsilon$ to the constant term of the weight generating function, and denote by $\hat x(Y,\varepsilon)$, $x_c(\varepsilon)$ and $\V(\varepsilon)$ the perturbed versions of $\hat x(Y)$, $x_c$ and $\V$. (This is the special case $k=0$ of the perturbation considered in the proof of \Cref{lem:x unique critical point}.)
When $\varepsilon>-B(0)$, we can still apply the argument of the first paragraph to the perturbed model. It follows that the zeros of $B(\,\cdot\,,\varepsilon)$ in $\cV(\varepsilon)$ are all double zeros. 
But the critical points of $B(\,\cdot\,,\varepsilon)$ do not depend on $\varepsilon$, while its zeros do.
More precisely, since $Y_*$ is a double zero of $B(Y)$, the equation $B(Y,\varepsilon)=B(Y)+\varepsilon=0$ has two solutions $Y_*^+(\varepsilon)$ and $Y_*^-(\varepsilon)$ such that $Y_*^\pm(\varepsilon) - Y_* \!\sim \pm c \cdot \varepsilon^{1/2}$ as $\varepsilon \to 0$, where~$c =\! \sqrt{\frac2{B''(Y_*)}}$. It follows that $Y_*^\pm(\varepsilon)$ are both simple zeros of $B(\,\cdot\,,\varepsilon)$, and hence $Y_*^\pm(\varepsilon) \not \in \cV(\varepsilon)$ for all $\varepsilon\ne 0$ close to zero.

By continuity, $Y_*$ must be on the boundary of $\cV$. It is clear that $Y_*\ne Y_c$. By \Cref{lem:x unique critical point}, we have $\hat x'(Y_*)\ne 0$ and thus $\hat x$ is locally injective at $Y_*$. It follows that in a small neighborhood of $Y_*$, the preimage $\hat x^{-1}(\cdisk_{x_c})$ coincides with $\cV$, that is, $Y\in \cV$ \Iff\ $|\hat x(Y)|\le x_c$. By continuity, the same is true in the perturbed model when $\varepsilon$ is small enough. Hence $Y_*^\pm(\varepsilon) \not \in \cV(\varepsilon)$ implies that $\abs{ \hat x(Y_*^\pm(\varepsilon),\varepsilon)} > x_c(\varepsilon)$ for all $\varepsilon \ne 0$ close to $0$.
However, the asymptotics $Y_*^\pm(\varepsilon)-Y_* \sim \pm c\cdot \varepsilon^{1/2}$ implies that 
\begin{equation}
\hat x(Y_*^\pm(\varepsilon),\varepsilon)
\,=\, \hat x(Y_*) + \hat x'(Y_*) \cdot
      (Y_*^\pm(\varepsilon) - Y_*) + O(\varepsilon) 
\,=\, \hat x(Y_*) \pm c\, \hat x'(Y_*) \cdot
      \varepsilon^{1/2} + o(\varepsilon^{1/2}).
\end{equation}
Since $Y_*\in \partial \V$, we have $\abs{\hat x(Y_*)} = x_c$. It follows that
\begin{equation}\label{eq:sqrt epsilon asymp bound}
\abs{\hat x(Y_*^\pm(\varepsilon),\varepsilon) } 
\,=\, x_c \cdot \abs{ 1 \pm \tilde c \cdot \varepsilon^{1/2} + o(\varepsilon^{1/2})}
\,=\, x_c \cdot \m({ 1 \pm \Re \mn({\tilde c \cdot \varepsilon^{1/2}} + o(\varepsilon^{1/2}) },
\end{equation}
where $\tilde c = \frac{c\, \hat x'(Y_*)}{{\hat x(Y_*)}}\ne 0$. In the proof of \Cref{lem:x unique critical point}, we have shown that $x_c(\varepsilon)$ is differentiable at $\varepsilon=0$. Hence the asymptotic expansion of $\abs{ \hat x(Y_*^\pm(\varepsilon),\varepsilon)}$ and the inequality $\abs{ \hat x(Y_*^\pm(\varepsilon),\varepsilon)} > x_c(\varepsilon)$ implies that $ \Re(\tilde c \cdot \varepsilon^{1/2})=0$ for all $\varepsilon\ne 0$ close to $0$. But this is impossible, because $\arg(\tilde c\cdot \varepsilon^{1/2})$ changes by $\pi/2$ when $\varepsilon$ changes sign. We conclude by contradiction that $B(Y)\ne 0$ for all $Y \in \cV$. Since $\hat x(Y)=\frac{YB(Y)}{(B(Y)+YB'(Y))^2}$ is bounded on $\cV$, we also have $B(Y)+YB'(Y) \ne 0$ for all $Y\in \cV$.
\end{proof}

\subsection{$q(Y,y)$ does not vanish on $\cV \times \cdisk_{Y_c}$ except at $(Y_c,Y_c)$}
\label{sec:technical lemma/unique zero}

Recall that $q$ is a holomorphic function on $\cV \times \cdisk_{Y_c}$ such that $q(Y,Y)=\phi'(Y)$  and $q(Y,y)=\frac{Q(Y,y)}{(Y-y)^2}$ when $Y\ne y$. 
We start by showing that a zero of the function $Q$ in $\cV \times \cdisk_{Y_c}$ must also be a zero of both $\partials_x Q$ and $\partial_y Q$, using a variant of the quadratic method. 
The proof is complicated by the fact that $\partial_x F(x,y)$ is absolutely convergent only on $\disk_{x_c} \times \cdisk_{Y_c}$ and not on $\partial \disk_{x_c} \times \cdisk_{Y_c}$ (c.f.\ \Cref{lem:domain of convergence}). We solve this problem with a continuity argument by studying the local geometry of the zero set $\Set{(Y,y)}{Q(Y,y)=0}$.

\begin{lemma}\label{lem:Q double zero}
If $Q(Y,y)=0$ for some $(Y,y)\in \cV \times \cdisk_{Y_c}$, then $\partial_y Q(Y,y) = \partials_x Q(Y,y) = 0$ as well.
\end{lemma}

\begin{proof}
Let $T(Y,y) = 2y \cdot \m({F(\hat x(Y),y) - \frac12} + \phi(Y)$. A simple rearrangement of \eqref{eq:F:parametrization} shows that $Q(Y,y) = T(Y,y)^2$. 
By \Cref{lem:domain of convergence,lem:def V}, 
the power series $F(x,y)$ and $\partial_y F(x,y)$ are absolutely convergent on $\cdisk_{x_c} \times \cdisk_{Y_c}$, and $\hat x$ maps $\cV$ continuously to $\cdisk_{x_c}$. It follows that $\partial_y T(Y,y) = 2y\cdot  \partial_y F(\hat x(Y),y) + 2\m({ F(\hat x(Y),y)-\frac12 }$ is bounded on $\cV \times \cdisk_{Y_c}$. Hence $Q(Y,y)=T(Y,y)^2=0$ implies $\partial_y Q(Y,y) = 2 T(Y,y) \cdot \partial_y T(Y,y) =0$, for all $(Y,y)\in \cV \times \cdisk_{Y_c}$.
Similarly, since $\partial_x F(x,y)$ is absolutely convergent on $\disk_{x_c} \times \cdisk_{Y_c}$, the function $\partials_x T(Y,y) = 2y \cdot \partial_x F(\hat x(Y),y) + \partials_x \phi(Y)$ takes finite values on $\V \times \cdisk_{Y_c}$. Therefore $Q(Y,y)=T(Y,y)^2=0$ implies $\partials_x Q(Y,y) = 2 T(Y,y) \cdot \partials_x T(Y,y) =0$ for all $(Y,y)\in \V \times \cdisk_{Y_c}$.

It remains to show that $Q(Y,y)=0$ also implies $\partials_x Q(Y,y)=0$ for $(Y,y)\in \partial \V \times \cdisk_{Y_c}$. 

When $Y_*=y_*=Y_c$, we have $\partials_x Q(Y_c,Y_c)=0$ directly by Lemma~\refp{2}{lem:algebraic properties}. When $(Y_*,y_*) \in \partial \V \times \disk_{Y_c}$, the mapping $y\mapsto Q(Y,y)$ is analytic in a neighborhood of $y_*$ for all $Y\in \cV$. By  the generalization of the implicit function theorem in \Cref{lem:generalized ImplicitFT}, there exists a continuous function $\tilde y:\cV \to \complex$ such that $\tilde y(Y_*)=y_*$ and $Q(Y,\tilde y(Y))=0$ for all $Y\in \cV$ close enough to $Y_*$. Since $y_*$ is in the interior of $\cdisk_{Y_c}$, the graph of this function $\tilde y$ contains a sequence $(Y_j,y_j)\in \V \times \disk_{Y_c}$ that converges to $(Y_*,y_*)$. But the first paragraph of the proof ensures that $Q(Y_j,y_j) = \partials_x Q(Y_j,y_j) = 0$ for all $j$. Therefore $\partials_x Q(Y_*,y_*) = 0$ by continuity.

It remains the case where $(Y_*,y_*) \in \partial \V \times \partial \disk_{Y_c}$ and $(Y_*,y_*)\ne (Y_c,Y_c)$. Thanks to the $\Delta$-analyticity of $B(y)$, the function $Q(Y,y)$ is analytic in $Y$ when $Y\in \partial \V \setminus \{Y_c\}$, and analytic in $y$ when $y\in \partial \disk_{Y_c} \setminus \{Y_c\}$. In both cases, we can apply \Cref{lem:generalized ImplicitFT} to express locally the zero set of $Q$ as the graphs of some functions.
We will use the asymptotics of these functions provided in \Cref{lem:generalized ImplicitFT} to show that their graphs contain a sequence $(Y_j,y_j)\in \V \times \disk_{Y_c}$ that converges to $(Y_*,y_*)$. As in the previous paragraph, this implies $\partials_x Q(Y_*,y_*) = 0$ by continuity.
Actually, we will proceed by contradiction: Assume that $Q(Y_*,y_*)=0$ and $\partials_x Q(Y_*,y_*)\ne 0$ for some $(Y_*,y_*)\in \partial \V \times \partial \disk_{Y_c} \setminus \{(Y_c,Y_c)\}$. We have two cases:

When $Y_*\ne Y_c$, the mapping $Y\mapsto Q(Y,y)$ is analytic at $Y_*$ for all $y\in \cdisk_{Y_c}$. Moreover:
\begin{itemize}[topsep=0.5ex,itemsep=0ex]
\item
$\partial_Y Q(Y_*,y_*) = \partials_x Q(Y_*,y_*) \cdot \hat x'(Y_*) \ne 0$ by \Cref{lem:x unique critical point}, hence $Y_*$ is a simple zero of $Y\mapsto Q(Y,y_*)$.
\item
We have $Q(Y_*,y)\sim c\cdot (y-y_*)^\gamma$ as $y\to y_*$ in $\cdisk_{Y_c}$ for some $\gamma\ge 2$ and $c\ne 0$.

Indeed, according to the first paragraph of the proof, we have $Q(Y_*,y_*) = \partial_y Q(Y_*,y_*)=0$. If $y_*\ne \rho$, then $y\mapsto Q(Y_*,y)$ is analytic in a neighborhood of $y_*$, so its Taylor expansion gives $Q(Y_*,y)\sim c\cdot (y-y_*)^\gamma$ for some integer $\gamma\ge 2$. If $y_*=\rho$ (i.e.\ we are in the non-generic phase \emph{and} $y_*=Y_c$), then $B(y)$ is the sum of an analytic function at $\rho$ and the singular term $\cst_B\cdot (1-y/\rho)^{ \tilde \alpha} (1+o(1))$ according to Assumption~\eqref{*}. 
Since $Q(Y_*,y)$ is a polynomial of $y$ and $B(y)$, we also have $Q(Y_*,y)\sim c\cdot (y-y_*)^\gamma$ for $\gamma=\tilde \alpha>2$.
\end{itemize}
Then, \Cref{lem:generalized ImplicitFT} applied to $f(z,s) = Q(Y_*+z, y_*+s)$, $n=1$, $S=\setn{s\in \complex}{y_*+s \in \cdisk_{Y_c}}$ and $h(s)=s^\gamma$ tells us that the zero set of $Q$ coincides in a neighborhood of $(Y_*,y_*)\in \complex \times \cdisk_{Y_c}$ with the graph $\setn{(\tilde Y(y),y)}{y\in \cdisk_{Y_c}}$ of a continuous function $\tilde Y:\cdisk_{Y_c}\to \complex$ such that $\tilde Y(y)-Y_* \sim \frac{c\cdot (y-y_*)^\gamma}{\partial_Y Q(Y_*,y_*)}$ as $y\to y_*$. The last asymptotics implies that the tangent of the disk $\disk_{Y_c}$ at $y_*$ is mapped by $\tilde Y$ to an angle of size $\gamma \pi \ge 2\pi$ at $Y_*$. This means that the image $\tilde Y(\disk_{Y_c})$ contains a neighborhood of $Y_*$, possibly with a cone of arbitrarily small angle removed. In particular, $\V \cap \tilde Y(\disk_{Y_c})$ contains a cone of positive angle at $Y_*$. It follows that there exists a sequence $\seq yj$ such that $(\tilde Y (y_j),y_j)\in \V\times \disk_{Y_c}$ for all $j$ and $(\tilde Y (y_j),y_j) \to (Y_*,y_*)$ as $j\to \infty$. 

When $Y_*=Y_c$ and $y_*\ne Y_c$, the mapping $y\mapsto Q(Y,y)$ is analytic at $y_*$ for all $Y\in \cV$. Moreover:
\begin{itemize}[topsep=0.5ex,itemsep=0ex]
\item
$Q(Y_c,y_*)=\partial_y Q(Y_c,y_*) = 0$, that is, $y_*$ is a zero of $y\mapsto Q(Y_c,y)$ of some multiplicity $n\ge 2$.
\item
$\partial_y^k Q(Y,y_*) = O((Y_c-Y)^{\alpha-1})$ for all $0\le k<n$ and $Q(Y,y_*)\sim c \cdot (Y_c-Y)^{ \alpha-1}$ as $Y\to Y_c$ in $\V$.

Indeed, $\partial_y^k Q(Y,y_*)$ is a rational function of $(Y,B'(Y),B(Y))$. Due to Assumption~\eqref{*}, it is $C^1$-continuous in a neighborhood of $Y_c$ in $\cV$.
By the definition of $n$, we have $\partial_y^k Q(Y_c,y_*)=0$ for all $0\le k<n$. On the other hand, we have $\partial_Y \partial_y^k Q(Y,y_*) = \partials_x \partial_y^k Q(Y,y_*) \cdot \hat x'(Y)$ and by \Cref{lem:x asymptotics}, $\hat x'(Y)\sim \tilde c \cdot (Y_c-Y)^{\alpha-2}$ for some $\tilde c>0$ as $Y_c\to Y$. It follows that $\partial_Y \partial_y^k Q(Y,y_*) = O((Y_c-Y)^{\alpha-2})$ and therefore after integration, $\partial_y^k Q(Y,y_*) = O((Y_c-Y)^{\alpha-1})$ for all $0\le k<n$. 
When $k=0$, since $\partials_x Q(Y_c,y_*)\ne 0$ by assumption, we have $\partial_Y Q(Y,y_*) \sim \partials_x Q(Y_c,y_*) \cdot \tilde c\cdot (Y_c-Y)^{\alpha-2}$ and hence $Q(Y,y_*) \sim c\cdot (Y_c-Y)^{\alpha-1}$ for some $c\ne 0$. 
\end{itemize}
Then, \Cref{lem:generalized ImplicitFT} applied to $f(z,s)=Q(Y_c-s,y_*+z)$, $S=\setn{s\in \complex}{Y_c-s\in \V}$ and $h(s) = s^{\frac{\alpha-1}{n}}$ tells us that the zero set of $Q$ coincides in a neighborhood of $(Y_c,y_*)\in \V\times \complex$ with the graphs $\setn{ (Y,\tilde y_k(Y)) }{ Y\in \V, 1\le k\le n }$ of $n$ continuous functions $\tilde y_k:\cV \to \complex$ such that $\tilde y_k (Y)-y_* \sim \omega_k\cdot c'\cdot (Y_c-Y)^{\frac{\alpha-1}{n}}$ as $Y\to Y_c$, where $c'\ne 0$ and $\omega_1,\ldots,\omega_n$ are all the $n$-th roots of unity. The asymptotics of $\tilde y_k$ implies that, in a neighborhood of $y_*$, the union $\bigcup_{k=1}^n \tilde y_k(\V)$ contains a cone of positive angle at $y_*$, and all of its images under the rotations $z\mapsto \omega_k z$ ($k=1,\ldots,n$). Since $n\ge 2$, there is at least one $k$ for which $\disk_{Y_c} \cap \tilde y_k(\V)$ contains a cone of positive angle at $y_*$. It follows that there exists a sequence $\seq Yj$ such that $(Y_j,\tilde y_k(Y_j))\in \V\times \disk_{Y_c}$ for all $j$ and $(Y_j,\tilde y_k(Y_j)) \to (Y_*,y_*)$ as $j\to \infty$.

In both cases, the set $\V\times \disk_{Y_c}$ contains a sequence of zeros of $Q$ that converges to $(Y_*,y_*)$. As discussed before, this implies $\partials_x Q(Y_*,y_*)=0$. This completes the proof by contradiction for $(Y_*,y_*) \in \partial \V \times \partial \disk_{Y_c} \setminus \{(Y_c,Y_c)\}$.
\end{proof}

\begin{lemma}\label{lem:q no zero}
$q(Y,y)$ does not vanish on $\cV \times \cdisk_{Y_c} \setminus \{(Y_c,Y_c)\}$.
\end{lemma}

\begin{proof}
We prove that the set $\mathcal Z := \setn{(Y,y)\in \cV\times \cdisk_{Y_c}}{ (Y,y)\ne (Y_c,Y_c) \text{ and }q(Y,y)=0 }$ is empty in two steps: First, we derive from \Cref{lem:Q double zero} that all points in $\mathcal Z$ satisfy $\phi(Y)+y=B(y)=B'(y)=0$, hence $\mathcal Z$ is a discrete set. Then, we show that a solution of the system $\phi(Y)+y=B(y)=B'(y)=0$ cannot be an isolated point in $\mathcal Z$, hence $\mathcal Z$ must be empty.

Consider $(Y_*,y_*)\in \mathcal Z$. 
We have seen in Lemma~\refp{3}{lem:algebraic properties} that $q(Y,Y) = \phi'(Y) = \partials_x \phi(Y) \cdot \hat x'(Y)$, which does not vanish on $\cV\setminus \{Y_c\}$ by \Cref{lem:x unique critical point,lem:B no zero}. Since $Y_*\in \cV$ and $(Y_c,Y_c)\not \in \mathcal Z$, we have $Y_*\ne y_*$.
One can check that $q(0,y)=1$ for all $y$, so $Y_*\ne 0$ as well. On the other hand, we have $Q(Y_*,y_*)=(Y_*-y_*)^2 q(Y_*,y_*)=0$ and hence $\partials_x Q(Y_*,y_*) = \partial_y Q(Y_*,y_*)=0$ by \Cref{lem:Q double zero}. Explicitly,
\begin{align}
           Q(Y,y) &= 0 \qquad \Leftrightarrow \hspace{-18mm}&
\m({\phi(Y)+y}^2 &= 4yB(y) \cdot \hat x(Y) 
\label{eq:Q=0} \\
\partials_x Q(Y,y) &= 0 \qquad \Leftrightarrow \hspace{-18mm}&
2\partials_x \phi(Y) \cdot \m({\phi(Y)+y} &= 4yB(y) \\
\partial_y Q(Y,y) &= 0 \qquad \Leftrightarrow \hspace{-18mm}&
2\m({\phi(Y)+y} &= 4(B(y)+yB'(y))\cdot \hat x(Y) \,.
\label{eq:partial_y Q=0}
\end{align}
If $\phi(Y_*)+y_*\ne 0$, then the quotient of the first two equations gives that $\phi(Y_*) + y_* = 2\partials_x \phi(Y_*) \cdot \hat x(Y_*)$. One can check that this simplifies to $y_*=Y_*$. This contradicts what we have shown before. Hence $\phi(Y_*) + y_* = 0$. Plugging this into \eqref{eq:Q=0} and \eqref{eq:partial_y Q=0} gives that $B(y_*)=B'(y_*)=0$. (We have $\hat x(Y_*)\ne 0$ because $Y_*\ne 0$ and $\hat x$ is injective on $\cV$.)
These equations shows that $\mathcal Z$ is a discrete set, hence $(Y_*,y_*)$ is an isolated point of $\mathcal Z$.

If $Y_*\in \V$, then $Y\mapsto q(Y,y)$ is analytic in a neighborhood of $Y_*$, so \Cref{lem:generalized ImplicitFT} tells us that in a neighborhood of $(Y_*,y_*)$, the zero set of $q$ contains the graph of a continuous function $\tilde Y:\cdisk_{Y_c} \to \complex$ such that $\tilde Y(y_*)=Y_*$. But this implies that for any sequence $\seq yj$ that converges to $y_*$ in $\cdisk_{Y_c}$, the pair $(\tilde Y(y_j),y_j)$ is in $\mathcal Z$ for $j$ large enough, and converges to $(Y_*,y_*)$ as $j\to \infty$. This contradicts the fact that $(Y_*,y_*)$ is an isolated point of $\mathcal Z$. Therefore $Y_*\not\in \V$. The same argument also shows that $y_* \not \in \disk_{Y_c}$. In addition, if $Y_*=Y_c$, then $y_* = -\phi(Y_*) = -Y_c \frac{B(Y_c)-Y_c B'(Y_c)}{B(Y_c)+ Y_c B'(Y_c)}$ would be in $\disk_{Y_c}$. So $Y_*\ne Y_c$. Since $B(y_*)=0$, we also have $y_*\ne Y_c$.

The previous paragraph proves that $Y_*\in \partial \V \setminus \{Y_c\}$ and $y_*\in \partial \disk_{Y_c} \setminus \{Y_c\}$. Thanks to the $\Delta$-analyticity of $B$, the function $Q$ is analytic at $(Y_*,y_*)$. We have $Q(Y_*,y_*) = \partial_Y Q(Y_*,y_*) = \partial_y Q(Y_*,y_*) = 0$ by \Cref{lem:Q double zero}. Moreover, the using the equations $\phi(Y_*) + y_* = B(y_*) = B'(Y_*) = 0$, one can simplify $\partial_Y^2 Q(Y_*,y_*)$ and $\partial_Y \partial_y Q(Y_*,y_*)$ to
\begin{equation}
\partial_Y^2 Q(Y_*,y_*) = 2 \phi'(Y_*)^2       \qtq{and}
\partial_Y \partial_y Q(Y_*,y_*) = 2 \phi'(Y_*) 
%\partial_y^2 Q(Y_*,y_*) &= 2 - 4\m({2B'(y_*) + y_* B''(y_*)} \cdot \hat x(Y_*) 
\,.
\end{equation}
We have seen that $\phi'$ does not vanish on $\cV \setminus \{Y_c\}$. Therefore $Y_*$ is a double zero of $Y\mapsto Q(Y,y_*)$. According to the Newton-Puiseux theorem (see e.g.\ \cite{CasasAlvero2000}), the zero set of $Q$ coincides in a neighborhood of $(Y_*,y_*)$ with the graphs of two analytic functions $\tilde Y_1,\tilde Y_2$ such that $\tilde Y_1(y_*)=\tilde Y_2(y_*)=Y_*$ and $\tilde Y_1'(y_*), \tilde Y_2'(y_*)$ are the two roots of the polynomial
\begin{equation}
\partial_Y^2 Q(Y_*,y_*) \cdot r^2 + 2\partial_Y \partial_y Q(Y_*,y_*) \cdot r + \partial_y^2 Q(Y_*,y_*) \,.
\end{equation}
In particular, $\tilde Y_1'(y_*) + \tilde Y_2'(y_*) = -\frac{2\partial_Y \partial_y Q(Y_*,y_*)}{\partial_Y^2 Q(Y_*,y_*)} = -\frac{2}{\phi'(Y_*)}$.
Since $Y_*\ne y_*$, the zero sets of $Q$ and $q$ coincide near $(Y_*,y_*)$. We have seen that $\mathcal Z$ does not contain points in $\V \times \disk_{Y_c}$. It follows that the graphs of the functions $\tilde Y_j$ do not intersect 
$\V \times \disk_{Y_c}$ in a neighborhood of $(Y_*,y_*)$, or equivalently, $\V \cap \tilde Y_j(\disk_{Y_c}) = \varnothing$ locally near $Y_*$. 
When $y\to y_*$ along a half-line, $\tilde Y_j(y)\to Y_*$ also along a half-line. In this asymptotic regime, we have
\begin{align*}
Y\in \V ~~\iff~~
\abs{\frac{\hat x(Y)}{\hat x(Y_*)}}<1  ~~\iff~~
\Re\m({ \frac{\hat x(Y)}{\hat x(Y_*)}-1 }< 0
\qtq{and}
y\in \disk_{Y_c} ~~\iff~~
\abs{\frac{y}{y_*}} < 1 ~~\iff~~
\Re\m({ \frac{y}{y_*}-1 }<0    
\end{align*}
We cannot have $\tilde Y'_j(y_*)=0$ because otherwise $\tilde Y_j(\disk_{Y_c})$ would contain a cone of angle arbitrarily close to $2\pi$ near $Y_*$, which would intersect $\V$. It follows that when $y\to y_*$,
\begin{equation}
\frac{\hat x(\tilde Y_j(y))}{\hat x(Y_*)} - 1 
\ \sim\ \frac{\hat x'(Y_*)}{\hat x(Y_*)} (\tilde Y_j(y)-Y_*)
\ \sim\ \frac{\hat x'(Y_*)}{\hat x(Y_*)} y_* \, \tilde Y_j'(y_*) \cdot \frac{y-y_*}{y_*} \,.
\end{equation}
It is not hard to see that the condition $\V \cap \tilde Y_j(\disk_{Y_c})=\varnothing$ constraints the coefficient $\frac{\hat x'(Y_*)}{\hat x(Y_*)} y_*\, \tilde Y_j'(y_*)$ to be negative. It follows that $\frac{\hat x'(Y_*)}{\hat x(Y_*)} y_* \cdot (\tilde Y_1'(y_*) + \tilde Y_2'(y_*)) < 0$. 
Using $\tilde Y_1'(y_*) + \tilde Y_2'(y_*)=-\frac{2}{\phi'(Y_*)}$ and $y_*=-\phi(Y_*)$, one can check that the inequality simplifies to $\psi_*:=\frac{Y_*B'(Y_*)}{B(Y_*)}>1$. It follows that $y_* = -\phi(Y_*) = Y_*\frac{\psi_*-1}{\psi_*+1} \in (0,Y_*)$. But since $Y_*\in \cV \subseteq \cdisk_{Y_c}$, this contradicts the result that $y_*\not \in \disk_{Y_c}$. Therefore $\mathcal Z$ must be empty.
\end{proof}

\appendix

\section{Variational method for finding the dominant singularities of an inverse}
\label{sec:variational method}

In this appendix, we discuss the variational method used in the proof of \Cref{lem:x unique critical point} to find additional constraints on the critical points of $\hat x$ on the boundary of $\V$. 
We will describe the method in a general context:
%In the proof of \Cref{lem:x unique critical point}, we invoked a variational method for finding additional equations satisfied by the critical points of $\hat x$ on the boundary of $\V$. The intuitive idea behind this method has already been explained in the discussion leading to the lemma.
%In this appendix, we state and prove the corresponding formal result in a general setting:
\Cref{prop:variational method} states the result of the variational method under the minimal conditions for its application, and we discuss in \Cref{rmk:global assumption on V,rmk:applications in AC} how the setting of \Cref{prop:variational method} arises naturally in analytic combinatorics.

\begin{proposition}\label{prop:variational method}
Let $x_c:\mathcal I\to \real_{>0}$ be a continuous function on an open interval $\mathcal I$ that is differentiable at~$0\in \mathcal I$. 
Let $\hat x:\mathcal U \to \complex$ be a $C^1$ function on an open domain $\mathcal U\subseteq \complex \times \mathcal I$ that is analytic in its first variable.
For each $\varepsilon \in \mathcal I$, 
let $\V(\varepsilon)$ be a connected component of the (lower) level set $L_{x_c}(\varepsilon) := \Set{Y \in \mathcal U_\varepsilon}{|\hat x(Y,\varepsilon)|<x_c(\varepsilon)}$ that does not contain any critical point of $\hat x(\,\cdot\,,\varepsilon)$, 
where $\mathcal U_\varepsilon$ denotes the set $\Set{Y\in \complex}{(Y,\varepsilon)\in \mathcal U}$. 
In addition, we assume that the family $(\V(\varepsilon))_{ \varepsilon \in \mathcal I}$ contains a continuous function $Y_0:\mathcal I \to \complex$ in the sense that $Y_0(\varepsilon) \in \V(\varepsilon)$ for all $\varepsilon \in \mathcal I$. 

Under the above conditions, if $\partial_Y \hat x(Y_*,0)=0$ for some $Y_*\in \mathcal U_0$ on the boundary of $\,\V(0)$, then we have
\begin{equation}\label{eq:variational method}
\Re\m({
\frac{\partial_\varepsilon \hat x(Y_*,0)}{\hat x(Y_*,0)} 
} = \frac{x_c'(0)}{x_c(0)} \,.
\end{equation}
%In addition, if $x_c(\varepsilon) = \hat x(Y_c(\varepsilon),\varepsilon)$ for some continuous function $Y_c:\mathcal I\to \complex$ that either satisifes $\partial_Y \hat x(Y_c(\varepsilon),\varepsilon)=0$, or is independent of $\varepsilon$, then $x_c'(0) = \partial_\varepsilon \hat x(Y_c(0),0)$. 
\end{proposition}

\begin{remark}[Global version of the assumptions on $(\V(\varepsilon))_{\varepsilon \in \mathcal I}$]\label{rmk:global assumption on V}
\Cref{prop:variational method} states the \emph{local} version of the assumptions on the family $(\V(\varepsilon))_{\varepsilon\in \mathcal I}$. A stronger \emph{global} version goes as follows: we assume that for each $\varepsilon \in \mathcal I$, $\hat x(\,\cdot\,,\varepsilon)$ induces a conformal bijection from $\V(\varepsilon)$ to the disk $\disk_{x_c(\varepsilon)}$ such that the preimage of $0$, characterized by $\hat x(Y_0(\varepsilon),\varepsilon) = 0$ and $Y_0(\varepsilon) \in \V(\varepsilon)$, is a continuous function of $\varepsilon$.
It is clear that the global version of the assumptions implies the local one: if $\hat x(\,\cdot\,,\varepsilon)$ induces a conformal bijection from $\V(\varepsilon)$ to $\disk_{x_c(\varepsilon)}$, then $\V(\varepsilon)$ is a connected component of the lower level set $L_{x_c}(\varepsilon) = \Set{Y\in \mathcal U_\varepsilon}{|\hat x(Y,\varepsilon)|<x_c(\varepsilon)}$ that does not contain any critical point of $\hat x(\,\cdot\,,\varepsilon)$.

As explained in the proof of \Cref{lem:x unique critical point}, the domain $\V$ of the parking model and its perturbation~$\V(\varepsilon)$ satisfy the global version of the assumptions (hence also the local one).
\end{remark}

\begin{remark}[Applications in analytic combinatorics]\label{rmk:applications in AC}
The situation addressed in this appendix has also appeared in the enumeration of Ising-decorated triangulations in \cite{ChenTurunen2020}. In the proof of \cite[Lemma~13]{ChenTurunen2020}, the authors used a simplified version of the variational method discussed here to find one extra equation satisfied by the critical points of the function $\check x_R$ (the counterpart of $\hat x(\,\cdot\,,\varepsilon)$ in \cite{ChenTurunen2020}) on the boundary of the domain~$\mathcal H_0(R)$ (the counterpart of $\V(\varepsilon)$). The main simplification in \cite{ChenTurunen2020} comes from the fact that the non-trivial critical points of $\check x_R$ are known to be simple.

More generally, whenever we have a power series $\hat Y$ of radius of convergence $x_c$ whose inverse $\hat x:=\hat Y^{-1}$ has an analytic continuation on $\V:= \hat Y(\disk_{x_c})$, the function $\hat x$ will induce a conformal bijection from $\V$ to $\disk_{x_c}$. In the context of analytic combinatorics, a natural question is to ask where are the singularities of $\hat Y$ on its circle of convergence $\partial \disk_{x_c}$. If $x_*\in \partial \disk_{x_c}$ is a point such that $Y_*=\hat Y(x_*) \in \partial \V$ is well-defined and that $\hat x$ has an analytic continuation in a neighborhood $\mathcal U$ of $Y_*$, then $\hat Y$ has a singularity at $x_*$ \Iff\ $Y_*$ is a critical point of $\hat x$. 
The fact that $Y_*$ is a critical point of $\hat x(\,\cdot\,,0)$ with a critical value in $\partial \disk_{x_c}$ implies the equations
\begin{equation}\label{eq:pre-variational method}
\partial_Y \hat x(Y_*,0)=0
\qtq{and}
|\hat x(Y_*,0)|=x_c(0)\,.
\end{equation}
This is a system of three real equations on two real variables $\Re(Y_*)$ and $\Im(Y_*)$. So generically, this system should be able to eliminate all the ``unexpected'' singularities of $\hat Y$ on $\partial \disk_{x_c}$. However, if the function $\hat x$ (thus also $\hat Y$ and $\V$) depends on one or more extra (real) parameters $\varepsilon$, then the system \eqref{eq:pre-variational method} would generically have ``unexpected'' solutions for some subset of $\varepsilon$ of codimension one. 
\Cref{prop:variational method} provides a solution to this problem when the depence of $\hat x$ on the extra parameters $\varepsilon$ is $C^1$. More precisely, it provides one additional (real) equation on $Y_*$ for each (real) parameter $\varepsilon$, which is generically enough for eliminating all the ``unexpected'' solutions.
\end{remark}

The basic idea behind \Cref{prop:variational method} is the following: Since $\hat x(\,\cdot\,,\varepsilon)$ has no critical point in  $\V(\varepsilon)$ for any~$\varepsilon$, if $\hat x(\,\cdot\,,0)$ has a critical point $Y_*$ on the boundary of $\V(0)$, then the perturbation $\varepsilon$ must ``move $Y_*$ away from $\V(0)$'' for both positve and negative values of $\varepsilon$, and this gives a stationarity equation that $Y_*$ must satisfy.

The implementation of the above idea is complicated by the fact that both $\hat x(\,\cdot\,,\varepsilon)$ and the domain $\V(\varepsilon)$ change with the perturbation. For this we need to understand how the critical points and the level sets of $\hat x(\,\cdot\,,\varepsilon)$ depends on $\varepsilon$, and the interplay between the two. This is the subject of the two lemmas below. More precisely, \Cref{lem:critical locus} defines the branches of the critical points of $\hat x(\,\cdot\,,\varepsilon)$ near $(Y_*,0)$, and computes the derivative of $\hat x(\,\cdot\,,\varepsilon)$ along these branches. 
\Cref{lem:connected level set} establishes the connectedness of the level set of $\hat x(\,\cdot\,,\varepsilon)$ near $Y_*$, when the level is higher than all the critical values of $\hat x(\,\cdot\,,\varepsilon)$.

\begin{lemma}\label{lem:critical locus}
If $Y_*$ is a critical point of $\hat x(\,\cdot\,,0)$ of multiplicity $n\ge 1$, then there exists a neighborhood $\,\mathcal V \times \mathcal J$~of~$(Y_*,0)$ in which the critical points of $\hat x(\,\cdot\,,\varepsilon)$ in $\mathcal V \times \mathcal J$  are parametrized by $n$ (not necessarily distinct) continuous functions, that is, there exist $n$  continuous functions $Y_*\0k\! :\mathcal J\to \mathcal V$ such that 
\begin{equation}\label{eq:critical locus}
Y_*\01(0)= \cdots= Y_*\0n(0)=Y_* \qtq{and} 
\setb{Y\in \mathcal V}{\partial_Y \hat x(Y,\varepsilon)=0} = \m.{ Y_*\01(\varepsilon),\cdots, Y_*\0n(\varepsilon) }
\end{equation}
for all $\varepsilon \in \mathcal J$.
Moreover, for all $1\le k\le n$, we have $\od{}\varepsilon \!\! \left.\hat x(Y_*\0k(\varepsilon),\varepsilon)\right|_{\varepsilon=0} = \partial_\varepsilon \hat x(Y_*,0)$.
\end{lemma}

\begin{proof}
The generalization of implicit function theorem given in \Cref{lem:generalized ImplicitFT} ensures that the zero set of $\partial_Y \hat x$ defines $n$ continuous functions $Y_*\01,\ldots,Y_*\0n$ satisfying \eqref{eq:critical locus}. 

By Cauchy's integral formula, $\partial_Y \hat x(Y,\varepsilon) = \frac{1}{2\pi i}\oint \frac{\hat x(\eta,\varepsilon)}{(Y-\eta)^2}\dd \eta$. So the $C^1$-continuity of $\hat x$ implies that of $\partial_Y \hat x$.
Since both $\hat x$ and $\partial_Y \hat x$ are $C^1$ \wrt\ $(Y,\varepsilon)$, and $Y_*$ is a zero of $ Y \mapsto \partial_Y \hat x(Y,0)$ of multiplicity $n$, we have 
\begin{align*}
\hat x(Y,\varepsilon) 
   &\,=\, \hat x(Y,0) 
+ \partial_\varepsilon \hat x(Y_*,0)\cdot \varepsilon 
+ o(\varepsilon) 
\\ &\,=\, \hat x(Y_*,0) + O((Y-Y_*)^{n+1})
+ \partial_\varepsilon \hat x(Y_*,0)\cdot \varepsilon 
+ o(\varepsilon)
\\ \tq{and} 
\partial_Y \hat x(Y,\varepsilon) 
&\,=\, \partial_Y \hat x(Y,0) 
+ \partial_\varepsilon \partial_Y \hat x(Y_*,0) \cdot \varepsilon 
+ o(\varepsilon)  
\\ &\,=\, c\cdot (Y-Y_*)^n + O((Y-Y_*)^{n+1}) 
+ \partial_\varepsilon \partial_Y \hat x(Y_*,0) \cdot \varepsilon 
+ o(\varepsilon)
\end{align*}
as $(Y,\varepsilon) \to (Y_*,0)$, where $c = \frac1{n!} \partial_Y^{n+1} \hat x(Y_*,0) \ne 0$. Applying the above expansion of $\partial_Y \hat x$ to the equation $\partial_Y \hat x(Y_*\0k(\varepsilon),\varepsilon) = 0$ shows that $Y_*\0k(\varepsilon) - Y_* = O(\varepsilon^{1/n})$ as $\varepsilon \to 0$. Plugging this into the expansion of $\hat x$ then gives $\hat x(Y_*\0k(\varepsilon),\varepsilon) = \hat x(Y_*,0) + \partial_\varepsilon \hat x(Y_*,0) \cdot \varepsilon + o(\varepsilon)$, that is, the function $\varepsilon \mapsto \hat x(Y_*\0k(\varepsilon),\varepsilon)$ is differentiable at $0$, with a derivative equal to $\partial_\varepsilon \hat x(Y_*,0)$.
\end{proof}

\begin{lemma}\label{lem:connected level set}
Assume that $\hat x(Y_*,0)\ne 0$.
Then the neighborhood $\,\mathcal V\times \mathcal J$ in \Cref{lem:critical locus} can be chosen in such a way that for all $\varepsilon \in \mathcal J$ and $\,h>\!\! \max \limits_{k=1,\dots,n} \abs{\hat x(Y_*\0k(\varepsilon),\varepsilon)}$, the local level set $\,L^{\mathcal V}_h(\varepsilon) := \Set{Y\in \mathcal V}{ |\hat x(Y,\varepsilon)| < h }$ is connected.
\end{lemma}

\begin{proof}
\Wlg, we assume that $Y_*=0$ and $\hat x(0,0)=1$. We choose $\mathcal V$ to be the closed ball of radius $r$ centered at $0$ and $\mathcal J=(-\varepsilon_0,\varepsilon_0)$, for some $r>0$ and $\varepsilon_0>0$ to be specified later. We assume that $r$ and $\varepsilon_0$ are small enough so that by continuity, $\hat x$ does not vanish on $\mathcal{V\times J}$.
In the rest of the proof, unless otherwise mentioned, we fix an $\varepsilon \in \mathcal J$ and drop it from the notations.

\newcommand{\lbar}[1][]{\overline L\hspace{1pt}_{h#1}^{\mathcal V}}

Let $H(Y)=\abs{\hat x(Y)}$ and $h_c = \max_{1\le k\le n} H(Y_*\0k)$. Then we have $L_h^{\mathcal V} = \Set{Y\in \mathcal V}{H(Y) < h}$, and the lemma claims that $L_h^{\mathcal V}$ is connected for all $h>h_c$. For technical reasons, we will prove the claim for the closed level set $\lbar\!:= \Set{Y\in \mathcal V}{\!H(Y) \le h}$ instead of $L_h^{\mathcal V}$. This is clearly equivalent, since we have $\lbar['] \subseteq L_h^{\mathcal V} \subseteq \lbar$ for all~$h'<h$.
Notice that the maximum $h_0=\max H(\mathcal V)$ is finite, and for all $h\ge h_0$, we have $\lbar=\mathcal V$, which is connected. 
For~the other values of $h$, we will construct a continuous mapping $\varPhi:\mathcal V \times (h_c,h_0] \to \mathcal V$ with the property:
\begin{equation}\label{eq:retraction property}
\varPhi(Y,h)=Y\text{ ~when }h\ge H(Y) \qtq{and} H(\varPhi(Y,h))=h\text{ ~when }h\le H(Y).
\end{equation}
Since $h_0$ is the maximum of $H$ on $\mathcal V$, the above property dictates that $\varPhi(\,\cdot\,,h_0):\mathcal V \to \mathcal V$ is the identity map. For general $h\in (h_c,h_0]$, it says that $\varPhi(\,\cdot\,,h)$ is equal to the identity on $\lbar$, while projects the complement of $\lbar$ to the level line $\Set{Y\in \mathcal V}{H(Y)=h} \subseteq \lbar$. These facts ensure that for each $h\in (h_c,h_0)$, the restriction $\varPhi|_{\mathcal V\times [h,h_0]}$ defines a deformation retraction from $\mathcal V$ to $\lbar$. We refer to \cite{Hatcher2002} for the definition and properties of deformation retractions. In particular, it implies that $\lbar$ is homotopy equivalent to $\mathcal V$, therefore also connected.

We construct $\varPhi$ by defining its marginals $\tilde Y\equiv \varPhi(Y,\,\cdot\,):(h_c,h_0] \to \mathcal V$ using the following backward ODE: for all $h\ge H(Y)$, let $\tilde Y(h)=Y$ (this is the initial condition), and for $h\in (h_c,H(Y)]$, let $\tilde Y(h)$ satisfy
\begin{equation}\label{eq:backward ODE}
\od{\tilde Y}{h} = \mathcal F(\tilde Y) := \begin{cases}
\frac1{ \mathbf t\cdot \nabla H(\tilde Y) }\, \mathbf t
& \text{if }  \tilde Y\in \partial \mathcal V  \text{ and } \tilde Y\cdot \nabla H(\tilde Y) <0 , \\
\frac{1}{ \norm{\nabla H(\tilde Y)}^2 }\, \nabla H(\tilde Y)
& \text{otherwise}. 
\end{cases}
\end{equation}
Here we identify $\tilde Y$ with a vector in $\real^2$ and use the notations of real vector analysis: $\nabla H(\tilde Y)$ is the gradient of the scalar function $H(\tilde Y)$, $\mathbf t$ is any nonzero vector orthogonal to $\tilde Y$ (i.e.\ a tangent vector of the circle $\partial \mathcal V$), $\mathbf a \cdot \mathbf b$ stands for the inner product of two vectors $\mathbf a$ and $\mathbf b$, and $\norm{\,\cdot\,}$ denotes  the Euclidean norm on $\real^2$. 

Intuitively, the above ODE describes how a point $Y\in \mathcal V$ should move when we lower the height $h$ from $h_0$ to $h_c$, and force $Y$ to remain in the level set $\lbar$. For large values of $h$, the point $Y$ is already in $\lbar$, so it does not have to move, that is, $\tilde Y(h)=Y$ for $h\ge H(Y)$ (the first half of property \eqref{eq:retraction property}). When $h$ decreases below $H(Y)$, we move the point $Y$ to new positions $\tilde Y(h)$ by gradient descent: In general, $\tilde Y$ moves in the direction of $-\nabla H(\tilde Y)$ as $h$ decreases (the second case in \eqref{eq:backward ODE}). But when $\tilde Y$ is on the boundary of the disk $\mathcal V$ and $-\nabla H(\tilde Y)$ points to the exterior of $\mathcal V$, we project the vector $-\nabla H(\tilde Y)$ onto the tangent of $\partial \mathcal V$, and move $\tilde Y$ in that direction instead (the first case in \eqref{eq:backward ODE}). In both cases, the movement speed is adjusted so that $\od{}h H(\tilde Y(h)) = \nabla H(\tilde Y) \cdot \od{\tilde Y}{h} \equiv 1$, which implies $H(\tilde Y(h))=h$ for all $h\le H(Y)$ (the second half of property \eqref{eq:retraction property}). 

Due to the identity $H(\tilde Y(h))=h$ for $h\le H(Y)$, only the vector field $\mathcal F$ on $\mathcal V \setminus \lbar[_c]$ is involved in determining the solution of the ODE for $h\in (h_c,h_0]$. Let us show that the vector field $\mathcal F$ is locally bounded on $\mathcal V \setminus \lbar[_c]$. More precisely, let us show that when $r>0$ and $\varepsilon_0>0$ are chosen appropriately, $\mathcal F$ is bounded on $\mathcal V \setminus L_h^{\mathcal V}$ for all $h>h_c$ and $\varepsilon\in \mathcal J$: Since $\nabla H$ is continuous on the closed set $\mathcal V \setminus L_h^{\mathcal V}$, it is not hard to see from \eqref{eq:backward ODE} that $\mathcal F$ is bounded on $\mathcal V \setminus L_h^{\mathcal V}$ \Iff\ 
\begin{equation}\label{eq:nabla H non-vanishing}
\nabla H(\tilde Y)\ne 0 \text{ for all }\tilde Y\in \mathcal V \setminus L_h^{\mathcal V}\,, 
\qtq{and}
\nabla H(\tilde Y) \not \in \real_{<0} \cdot \tilde Y \text{ for all }\tilde Y\in \partial \mathcal V \setminus L_{h}^{\mathcal V}\,. 
\end{equation}
Recall that $H=|\hat x|$. It is a simple exercise to check that under the canonical identification of $\complex =\real^2$, we have
\begin{equation}\label{eq:nabla H in terms of x}
\nabla H = \frac{\hat x}{|\hat x|} \cdot (\partial_Y \hat x)^*,
\end{equation}
where $z^*$ denotes the complex conjugate of $z$. 
By assumption, $\hat x$ do not vanish on $\mathcal V$, for all $\varepsilon \in \mathcal J$.
Moreover, $h>h_c=\max_{1\le k\le n}H(Y_*\0k)$ implies that the $n$ critical points $Y_*\01,\cdots,Y_*\0n$ of $\hat x$ are all in $L_h^{\mathcal V}$, hence $\partial_Y\hat x$ does not vanish on $\mathcal V\! \setminus L_h^{\mathcal V}$.
It follows that $\nabla H(\tilde Y)\ne 0$ for all $\tilde Y \in \mathcal V \! \setminus L_h^{\mathcal V}$, i.e.\ the first half of the condition \eqref{eq:nabla H non-vanishing} is true.
On the other hand, by \eqref{eq:nabla H in terms of x} and the definition of $L_h^{\mathcal V}$, the second half of \eqref{eq:nabla H non-vanishing} is true \Iff\
\begin{equation}
\frac{\tilde Y \partial_Y \hat x(\tilde Y)}{\hat x(\tilde Y)} \not \in \real_{<0} \text{ for all }
\tilde Y \in \partial \mathcal V \text{ such that }
|\hat x(\tilde Y)|\ge h \,.
\end{equation}
Taking the limit $h\searrow h_c$, we see that the above condition holds \emph{for all} $h>h_c$ \Iff\
\begin{equation}\label{eq:nabla H non-vanishing 2}
\m({ \frac{\tilde Y \partial_Y \hat x(\tilde Y)}{\hat x(\tilde Y)} , \abs{ \hat x(\tilde Y) } -h_c } \not \in \real_{<0} \times \real_{>0} 
\text{ for all } \tilde Y \in \partial \mathcal V .
\end{equation}
To find $r$ and $\varepsilon_0$ such that the above condition is satisfied \emph{for all} $\varepsilon\in \mathcal J= (-\varepsilon_0,\varepsilon_0)$, we would like to obtain a non-trivial uniform limit of the pair 
in \eqref{eq:nabla H non-vanishing 2} when $(r,\varepsilon_0) \to (0,0)$. A bit of thought reveals that it is convenient to take the limit $r\to 0$ after $\varepsilon_0\to 0$, and we should renormalize both components of the pair by $1/r^{n+1}$. 
Indeed, since $\hat x$, $\partial \hat x$ and $h_c$ are all (uniformly) continuous in $\varepsilon$ and $h_c(\varepsilon=0)=\hat x(Y_*,0)=1$, we have
\begin{equation}
\frac{1}{r^{n+1}} \m({ \frac{\tilde Y \partial_Y \hat x(\tilde Y,\varepsilon)}{\hat x(\tilde Y,\varepsilon)} , \abs{ \hat x(\tilde Y,\varepsilon) }-h_c(\varepsilon) }
\xrightarrow[\varepsilon_0\to 0]{}
\frac{1}{r^{n+1}} \m({ \frac{\tilde Y \partial_Y \hat x(\tilde Y,0)}{\hat x(\tilde Y,0)} , \abs{\hat x(\tilde Y,0)} -1 }
\end{equation}
uniformly in $\varepsilon\in (-\varepsilon_0,\varepsilon_0)$.
Next, since $Y_*=0$ is a zero of multiplicity $n$ of $\partial_Y \hat x(\,\cdot\,,0)$, we have $\partial_Y \hat x(\tilde Y,0) \sim c\cdot \tilde Y^n$ and $\hat x(\tilde Y,0) = 1+\frac{c}{n+1}\tilde Y^{n+1} + o(\tilde Y^{n+1})$ for some $c\in \complex \setminus \{0\}$ when $\tilde Y\to 0$. 
We parametrize the point $\tilde Y \in \partial \mathcal V$ by $\tilde Y=r e^{i\tau}$ with $\tau \in [0,2\pi)$. Then taking the limit $r\to 0$ of the previous display gives
\begin{equation}\label{eq:pair uniform cvg}
\lim_{r\to 0} \lim_{\varepsilon_0 \to 0} 
\frac{1}{r^{n+1}} \m({ \frac{
   \tilde Y \partial_Y \hat x(\tilde Y,\varepsilon)
}{ \hat x(\tilde Y,\varepsilon) } , 
\abs{\hat x(\tilde Y,\varepsilon)} -h_c(\varepsilon) }
=\
\m({ c \cdot e^{i (n+1)\tau}, 
     \frac{1}{n+1} \Re(c \cdot e^{i(n+1)\tau}) }
\end{equation}
uniformly in $\tau\in [0,2\pi)$ and $\varepsilon\in (-\varepsilon_0,\varepsilon_0)$. It is not hard to see that for any fixed $c\in \complex \setminus \{0\}$ and $n\ge 1$, the set $\set{\mn({ c e^{i (n+1)\tau}, \frac{1}{n+1} \Re(c e^{i(n+1)\tau}) }}{\tau \in [0,2\pi)}$ is bounded away from $\real_{<0} \times \real_{>0}$ in $\complex \times \real$. Then the uniform convergence \eqref{eq:pair uniform cvg} implies that there exist $r>0$ and $\varepsilon_0>0$ such that \eqref{eq:nabla H non-vanishing 2} is true for all $\varepsilon \in (-\varepsilon_0,\varepsilon_0)$. 
With this choice of $r$ and $\varepsilon_0$, the second half of \eqref{eq:nabla H non-vanishing} is also true for all $\varepsilon \in (-\varepsilon_0,\varepsilon_0)$. We conclude that the vector field $\mathcal F$ is bounded on $\mathcal V\setminus L_h^{\mathcal V}$ for each $h>h_c$.

Due to the difference between the 2 cases on the \rhs\ of \eqref{eq:backward ODE}, the vector field $\mathcal F$ is not continuous. But the discontinuity only occurs on the circle $\partial \mathcal V$, and can be avoided using the following regularization:
For~$\sigma\in (0,r)$, let $\mathcal V_\sigma$ denote the closed disk of radius $r-\sigma$ centered at $0$.  Define $\mathcal F^\circ_\sigma:\mathcal V\to \real^2$ by $\mathcal F^\circ_\sigma(\tilde Y)=\mathcal F(\tilde Y)$ for all $\tilde Y \in \mathcal V_\sigma \cup \partial \mathcal V$, and by linear interpolation on the segment $\setn{r' e^{i\tau}}{r'\in [r-\sigma,r]}$ for each $\tau\in [0,2\pi)$.
Recall that the vector field $\mathcal F$ has the property that $\mathcal F(\tilde Y) \cdot \nabla H(\tilde Y)=1$ for all $\tilde Y\in \mathcal V$ (except at its singularities), which ensures that every solution of the backward ODE \eqref{eq:backward ODE} must satisfy $H(\tilde Y(h))=h$ for all $h\in (h_c,H(Y)]$. 
The linear interpolation in the definition of $\mathcal F^\circ_\sigma$ breaks this property. To restore it, we define $\mathcal F_\sigma = \frac{1}{\mathcal F_\sigma^\circ \cdot \nabla H} \mathcal F_\sigma^\circ$. 
One can check that $\mathcal{F_\sigma = F_\sigma^\circ = F}$ on $\mathcal{V_\sigma \cup \partial V}$. In particular, $\mathcal F_\sigma$ converges to $\mathcal F$ pointwise on $\mathcal V$ as $\sigma \to 0$. 
With a close look at the proof in the previous paragraph, it is not hard to see that for $\sigma>0$ small enough, the vector field $\mathcal F_\sigma$ is bounded and Lipschitz continuous on $\mathcal V\setminus L_h^{\mathcal V}$ for all $h>h_c$. 
Then, by the Cauchy-Lipschitz theorem (a.k.a.\  Picard–Lindel\"of theorem, see e.g.\ \cite[Theorem~2.9]{Teschl2012}), the backward ODE $\od{\tilde Y}{h}=\mathcal F_\sigma(\tilde Y)$ with the initial condition $\tilde Y(h)=Y$ for $h\in [H(Y),h_0]$ has a unique solution $\tilde Y_\sigma: (h_c,h_0] \to \mathcal V$,
such that $\varPhi_\sigma(Y,h) = \tilde Y_\sigma(h)$ defines a continuous function $\varPhi_\sigma: \mathcal V \times (h_c,h_0] \to \mathcal V$. 
By construction, the vector field $\mathcal F_\sigma$ satisfies $\mathcal F_\sigma \cdot \nabla H \equiv 1$ on $\mathcal V$. Hence we have $H(\tilde Y_\sigma(h))=h$ for all $h\in (h_c,H(Y)]$ and $Y\in \mathcal V$. It follows that $\varPhi_\sigma$ satisfies the condition~\eqref{eq:retraction property}. As discussed in the second paragraph of the proof, this implies the conclusion of the lemma.
\end{proof}

\begin{remark*}
In the above proof, we expect $\tilde Y_\sigma$ to converge to a solution of the backward ODE \eqref{eq:backward ODE} when $\sigma\to 0$. However this is not needed for proving \Cref{lem:connected level set}. 
\note{Add a picture to illustrate the ODE ?}

This proof is an adaptation of the proof of a classical theorem \cite[Theorem~3.1]{Milnor1963} of Morse theory in differential topology, which states that if $H:\mathcal V\to \real$ is a smooth function on a manifold $\mathcal V$ (without boundary), and $a<b$ are such that $H^{-1}([a,b])$ is compact and contains no critical point of $H$, then there exists a deformation retraction from $H^{-1}((-\infty,b])$ to $H^{-1}((-\infty,a])$. The main difficulty in adapting the classical proof to \Cref{lem:connected level set} is that now $\mathcal V$ has a boundary.
\note{There seem to be generalizations of Morse theory that deals with manifolds with boundary. For example \href{https://www.hse.ru/data/2013/12/15/1338379877/morall6.pdf}{here} and \href{https://arxiv.org/pdf/1207.3066.pdf}{here}.
But I cannot locate a theorem that clearly implies \Cref{lem:connected level set}.
}
\end{remark*}

\begin{proof}[Proof of \Cref{prop:variational method}]
Let $Y_*\in \mathcal U$ be a critical point of $\hat x(\,\cdot\,,0)$ on the boundary of $\,\V(0)$, and let $(Y_*\0k(\varepsilon))_{1\le k\le n}$ be the critical points of $\hat x(\,\cdot\,,\varepsilon)$ in a neighborhood of $Y_*$ as defined in \Cref{lem:critical locus}.
We prove \Cref{prop:variational method} in two steps: First, we show that for each $\varepsilon $ close enough to~$0$, there exists $k\in \{1,\ldots,n\}$ such that 
$\abs{\hat x(Y_*\0k(\varepsilon),\varepsilon)} \ge x_c(\varepsilon)$. Then, we derive \eqref{eq:variational method} from the previous inequality.
The first step is topological in nature and makes crucial use of the connectedness result of \Cref{lem:connected level set},
while the second step basically calculates the derivative of the ratio $|\hat x(Y_*\0k(\varepsilon),\varepsilon)| / x_c(\varepsilon)$ at $\varepsilon=0$.

We start by showing that for each neighborhood $\mathcal V$ of $Y_*$, we have $\mathcal V \cap \V(\varepsilon) \ne \varnothing$ for all $\varepsilon$ close enough~to~$0$:
Since $Y_*\in \partial \V(0)$, we have $\mathcal V \cap \V(0)\ne 0$. 
Recall that for all $\varepsilon \in \mathcal I$, $\V(\varepsilon)$ is a connected component of the level set $L_{x_c}(\varepsilon) = \Set{Y\in \mathcal U}{|\hat x(Y,\varepsilon)|<x_c(\varepsilon)}$. In particular, $\V(\varepsilon)$ is open and connected, thus path connected. Let $\Gamma \subseteq \V(0)$ be a path that connects $Y_0(0)$ to an arbitrary point $Y_1\in \mathcal V\cap \V(0)$ (recall that $Y_0:\mathcal{I\to U}$ is a continuous function such that $Y_0(\varepsilon)\in \V(\varepsilon)$ for all $\varepsilon$), and let $K\subseteq \V(0)$ be a compact neighborhood of $Y_0(0)$. By construction, $\Gamma \cup K$ is a compact subset of $L_{x_c}(0)$.
Then the continuity of $\hat x$ and $x_c$ implies that $\Gamma \cup K \subseteq L_{x_c}(\epsilon)$ for all $|\varepsilon|<\varepsilon_0$. Up to decreasing $\varepsilon_0$, we also have $Y_0(\varepsilon) \in K$ for all $|\varepsilon|<\varepsilon_0$. Since $Y_0(\varepsilon)\in \V(\varepsilon)$ and $\Gamma \cup K$ is connected, it implies $\Gamma \cup K \subseteq \V(\varepsilon)$. In particular, we have $Y_1\in \V(\varepsilon)$ and therefore $\mathcal V \cap \V(\varepsilon) \ne \varnothing$ for all $|\varepsilon|<\varepsilon_0$.

Let $\mathcal{V\times J}$ be a neighborhood of $(Y_*,0)$ having the properties stated in \Cref{lem:connected level set}. Without loss of generality, we assume that $\mathcal J\subseteq (-\varepsilon_0,\varepsilon_0)$. Now fix $\varepsilon \in \mathcal J$.
By the previous paragraph, $\V(\varepsilon)\cap \mathcal V \ne \varnothing$. But $\V(\varepsilon)$ is a connected component of $L_{x_c}(\varepsilon)$, so it contains at least one connected component of the local level set~$L_{x_c}^{\mathcal V}(\varepsilon) = L_{x_c}(\varepsilon) \cap \mathcal V$.
If $x_c(\varepsilon)>|\hat x(Y_*\0k(\varepsilon),\varepsilon)|$ for all $k\in \{1,\ldots,n\}$, then on the one hand, \Cref{lem:connected level set} states that $L_{x_c}^{\mathcal V}(\varepsilon)$ is connected, which implies $L_{x_c}^{\mathcal V}(\varepsilon) \subseteq \V(\varepsilon)$, and on the other hand, $Y_*\0k(\varepsilon) \in L_{x_c}^{\mathcal V}(\varepsilon)$ by the definition of the level set. It follows that $\V(\varepsilon)$ must contain all the $Y_*\0k(\varepsilon)$. This contradicts the assumption that $\V(\varepsilon)$ does not contain any critical point of $\hat x(\,\cdot\,,\varepsilon)$. Therefore we must have  $x_c(\varepsilon)\le |\hat x(Y_*\0k(\varepsilon),\varepsilon)|$, or equivalently 
\begin{equation}\label{eq:epsilon variation bound}
\Re \log \mB({ \frac{\hat x(Y_*\0k (\varepsilon),\varepsilon) }{ x_c(\varepsilon) } } \ge 0
\end{equation}
for at least one $k\equiv k(\varepsilon)\in \{1,\ldots,n\}$.

Since $Y_*$ is on the boundary of $\V(0)$, we have $|\hat x(Y_*,0)|=x_c(0)$, that is, the above inequality becomes an equality at $\varepsilon=0$. So the derivative of the \lhs\ at $\varepsilon=0$, if exists, is equal to zero. But by \Cref{lem:critical locus}, we have
\begin{equation}
\left. \od{}\varepsilon  \hat x(Y_*\0k(\varepsilon),\varepsilon) \right|_{\varepsilon=0} 
=\ \lim_{\varepsilon \to 0} \frac{\hat x(Y_*\0k(\varepsilon),\varepsilon) - \hat x(Y_*,0)}{\varepsilon}
\ =\ \partial_\varepsilon \hat x(Y_*,0) \,,
\end{equation}
regardless of the choice of $k\equiv k(\varepsilon)\in \{1,\ldots,n\}$. It follows that the \lhs\ of \eqref{eq:epsilon variation bound} is indeed differentiable at $\varepsilon=0$, and the vanishing of the derivative gives Equation~\eqref{eq:variational method}.
\end{proof}

\section{Modified inverse/implicit function theorem}
\label{sec:generalized IFTs}

In this appendix we prove two analytic lemmas used in this paper. They can be viewed as modifications of the inverse function theorem and of the implicit function theorem, respectively. 

\begin{lemma}[Inverse function theorem in a cone]\label{lem:generalized InverseFT}
Let $\hat w:K\to \complex$ be a holomorphic function such that $\hat w(z)\to 0$ and $\hat w'(z)\to c$ when $z\to 0$ in $K$, where $c\in \complex \setminus \{0\}$, and $K=\setn{re^{i\tau}}{r\in (0,r_0),\, \tau\in (\tau_1,\tau_2)}$ is a truncated cone with an angle $\tau_2-\tau_1 \in (0,2\pi)$. Then there is a neighborhood $\mathcal U$ of $0$ such that $\hat w|_{K\cap \mathcal U}$ is injective and its inverse $\hat z: \hat w(K\cap \mathcal U) \to K\cap \mathcal U$ is an analytic function such that $\hat z(w)\to 0$ and $\hat z(w)\to c^{-1}$ when $w\to 0$ in $\hat w(K\cap \mathcal U)$.
Moreover, for all $K'\subseteq K$ of the form $K'=\setn{re^{i\tau}}{r\in (0,r_0),\, \tau \in (\tau_1',\tau_2')}$ with $\tau_1<\tau_1'<\tau_2'<\tau_2$, there exists a neighborhood $\mathcal V$ of $0$ such that $c\cdot (K'\cap \mathcal V) \subseteq \hat w(K \cap \mathcal U)$.
%\note{higher dimensional generalization: the main difficulty lies in finding a condition to describe the cones for which the inversion is possible (holomorphy should be replaced by $C^1$-continuity.}
\end{lemma}

\begin{proof}
Since the cone $K$ has an angle strictly smaller than $2\pi$, one can find a constant $M<\infty$ such that for all $z_1,z_2\in K$, there exists a smooth path $\Gamma_{z_1,z_2} \subseteq K \cap \disk_{\max(|z_1|,|z_2|)}$ of length $\abs{\Gamma_{z_1,z_2}} < M\cdot \abs{z_2-z_1}$ which connects $z_1$ and $z_2$. Let $\mathcal U$ be a ball centered at the origin such that $\abs{\hat w'(z)-c} \le M^{-1} |c|$ for all $z\in K \cap \mathcal U$. 
Then for all $z_1,z_2 \in K \cap \mathcal U$, we have $\Gamma_{z_1,z_2} \subseteq K \cap \mathcal U$. It follows that
\begin{equation}\label{eq:inverse in a cone bound}
\mb|{\hat w(z_2)-\hat w(z_1) -c(z_2-z_1)} 
\  = \ \abs{\int_{\Gamma_{z_1,z_2}} (\hat w'(z)-c)\dd z}
\ \le\ \abs{\Gamma_{z_1,z_2}} \cdot \!\! 
       \sup_{z\in \Gamma_{z_1,z_2}} \!\! \abs{\hat w'(z)-c} 
\  < \ |c|\cdot |z_2-z_1| \,.
\end{equation}
The above bound implies that $\hat w$ is injective on $K\cap \mathcal U$. By the classical inverse function theorem, the inverse $\hat z$ of $\hat w|_{K\cap \mathcal U}$ is analytic on $\hat w(K\cap \mathcal U)$. Moreover, taking $z_1=z$ and $z_2\to 0$ in the above display gives $|\hat w(z)-cz|<|cz|$, which implies that $\hat w(z)\to 0$ \Iff\ $z\to 0$. It follows that $\hat z(w)\to 0$ and $\hat z'(w) = \frac{1}{\hat w'(\hat z(w))} \to c^{-1}$ when $w\to 0$ in $\hat w(K\cap \mathcal U)$.

Given $\tau_1<\tau_1'<\tau_2'<\tau_2$, let $\delta = \frac12 \min (\tau_2-\tau_2',\tau_1'-\tau_1)$. The limits $\hat w(z)\to 0$ and $\hat w'(z)\to c$ imply that $\frac{\hat w(z)}{c z} \to 1$ when $z\to 0$ in $K$. So there exists $r_\delta>0$ such that $\abs{\frac{\hat w(z)}{c z}-1} < \delta$ for all $z\in K \cap \cdisk_{r_\delta}$. \Wlg, we assume that $\cdisk_{r_\delta} \subset \mathcal U$. 
Now take $\mathcal V=\disk_{(1-\delta)r_\delta}$ and $\tilde K = \setn{r e^{i\tau}}{ r\in (0,r_\delta], \tau\in [\tau_1+\delta,\tau_2-\delta] }$. Then we have $K'\cap \mathcal V\subseteq \tilde K \subseteq K \cap \cdisk_{r_\delta}$, where $K'=\setn{re^{i\tau}}{r\in (0,r_0),\, \tau \in (\tau_1',\tau_2')}$ as stated in the lemma. Thanks to the estimate $\abs{\frac{\hat w(z)}{c z}-1} < \delta$ for $z\in K \cap \cdisk_{r_\delta}$, the boundary of $\tilde K$ is mapped by $\hat w$ to a curve that encloses the truncated cone $c\cdot (K'\cap \mathcal V)$. Since $\hat w(\tilde K)$ is simply connected, this implies that $c\cdot (K'\cap \mathcal V) \subseteq \hat w(\tilde K) \subseteq \hat w(K\cap \mathcal U)$. 
\end{proof}

\begin{lemma}[Modified implicit function theorem]
\label{lem:generalized ImplicitFT}
Let $S$ be a topological space containing $0$ that is locally connected at $0$, and let $\mathcal U$ be a neighborhood of $(0,0)$ in $\complex \times S$.
Assume that $f:\mathcal U\to \complex$ is a continuous function which is analytic in its first variable, such that $z\mapsto f(z,0)$ has a zero of multiplicity $n\ge 1$ at $z=0$. Then the zeros set of $f$ can be parametrized by $n$ continuous functions near $(0,0)$, that is, 
there exist a neighborhood $\mathcal{V\times S}$ of $(0,0)$ and $n$ continuous functions $z_k\!:\mathcal{S\to V}$ such that for each $s\in \mathcal S$, the function $z\mapsto f(z,s)$ has exactly $n$ zeros in $\mathcal V$ (counted with multiplicity), given by $z_1(s),\cdots,z_n(s)$.

In addition, if there is a continuous function $h:S\to \complex$ such that for all $0\le m\le n$, the limit
\begin{equation}\label{eq:ImplicitFT/coeff expansion}
c_m = \frac1{m!} \, \lim_{s\to 0} \frac{\partial_z^m f(0,s)}{h(s)^{n-m}}
\end{equation}
exists in $\complex$, then we have 
\begin{equation}
\lim_{s\to 0} \frac{z_k(s)}{h(s)} = r_k
\end{equation}
for $1\le k\le n$, where $r_1,\ldots,r_n \in \complex$ are the $n$ roots of the polynomial $c_0 + c_1 r + \cdots + c_n r^n$ listed in some order.
\end{lemma}

\begin{proof}
Fix any $r>0$ such that $z=0$ is the only zero of $f(\,\cdot\,,0)$ in $\disk_{2r}$. By the continuity of $f$, there exists a neighborhood $\mathcal S\subseteq S$ of $0$ such that $f(z,s)\ne 0$ for all $(z,s)\in \partial \disk_r \times \mathcal S$. \Wlg, we assume that $\mathcal S$ is connected. According to Cauchy's argument principle, for each $s\in \mathcal S$, the number of zeros of $f(\,\cdot\,,s)$ in $\disk_r$ is given by $N(s)=\frac{1}{2\pi i} \oint_{\partial \disk_r}\! \frac{\partial_z f(z,s)}{f(z,s)}\dd z$. Since $N(s)$ is continuous and integer, it is equal to $n$ for all $s\in \mathcal S$.
For each $s\in \mathcal S$, let $z_1(s),\cdots,z_n(s)$ be the $n$ zeros of $f(\,\cdot\,,s)$, and define 
\begin{equation}
p(z,s) \,\equiv\, 
z^n + a_{n-1}(s) z^{n-1} + \cdots + a_0(s) \,=\,
\prod_{k=1}^n \m({z-z_k(s)} \,.
\end{equation}
Similarly to Cauchy's argument principle, we can compute a sum over the zeros $z_1(s),\cdots,z_n(s)$ by an integral:
\begin{equation}
p(z,s)
\,=\, \exp \m({ \sum_{k=1}^n \log(z-z_k(s)) } 
\,=\, \exp \m({ \frac{1}{2\pi i} \oint_{\partial \disk_r}\! \frac{\partial_z f(\zeta,s)}{f(\zeta,s)} \log(z-\zeta) \dd \zeta } \,.
\end{equation}
Notice that, since the logarithm is well-defined modulo $2\pi i\integer$, the exponential is well-defined for all $z\in \disk_r$. The above expression shows that $p(z,s)$ is continuous in $(z,s)$. Thanks to Cauchy's differentiation formula, all of its $z$-derivatives are also continuous \wrt\ $(z,s)$. In particular, $a_k(s) = \frac{1}{k!}\partial_z^k p(0,s)$ is a continuous function of $s\in \mathcal S$, for all $0\le k<n$. It is well-known that the zeros of polynomial are continuous functions of its coefficients. It follows that, up to adjusting the ordering of the roots $z_1(s),\cdots,z_n(s)$ for each $s\in \mathcal S$, the functions $z_k:\mathcal S\to \disk_r$ are continuous. This proves the first part of the lemma with $\mathcal V = \disk_r$. 

Notice that since $p(z,s)$ and $f(z,s)$ have the same set of zeros with multiplicity on $\mathcal V \times \mathcal S$, their quotient $c(z,s) = \frac{f(z,s)}{p(z,s)}$ is continuous on $\mathcal V \times \mathcal S$ and does not vanish there (the joint continuity in $(z,s)$ follows from Cauchy's integral formula).

Now assume that $f$ satisfies \eqref{eq:ImplicitFT/coeff expansion}. When $m=n$, the formula gives $c_n=\frac1{n!}\partial_z^n f(0,0)$. Since $z=0$ is a zero of multiplicity $n$ for $f(\,\cdot\,,0)$, this implies $c_n \ne 0$. If $h(0)\ne 0$, then we have $c_m=\frac{1}{m!}\frac{\partial_z^m f(0,0)}{h(0)^{n-m}}=0$ for all $m<n$ and $r_k = \frac{z_s(0)}{h(0)}=0$ for all $1\le k\le n$, so the claims of the lemma are trivially true. When $h(0)=0$, the limits \eqref{eq:ImplicitFT/coeff expansion} imply that $\partial_z^j f(0,s) = O(h(s)^{n-j}) = o(h(s)^{n-m})$ for all $j<m\le n$. By differetiating $p(z,s)=\frac{f(z,s)}{c(z,s)}$, we get
\begin{equation}
a_m(s) \,=\, \frac{\partial_z^m p(0,s)}{m!}
\,=\, \frac{1}{c(0,s)} \frac{ \partial_z^m f(0,s) }{m!} + o(h(s)^{n-m})
\end{equation}
and therefore $\lim_{s\to 0} \frac{a_m(s)}{h(s)^{n-m}} = \frac{c_m}{c(0,0)}$ for all $0\le m<n$. The same calculation for $m=n$ shows that $c(0,0)=c_n$. By construction, $z_k(s)$ is a zero of the polynomial $p(z,s) = z^n + a_{n-1}(s)z^{n-1} + \cdots + a_0(s)$. It follows that
\begin{equation}\label{eq:ImplicitFT/asymp equation}
0\,=\, \frac{p(z_k(s),s)}{h(s)^n} \,=\, 
\m({ \frac{z_k(s)}{h(s)} }^n \!+ \sum_{m=0}^{n-1} \frac{a_m(s)}{h(s)^{n-m}} \cdot \m({ \frac{z_k(s)}{h(s)} }^m 
=\, \m({ \frac{z_k(s)}{h(s)} }^n \!+ \sum_{m=0}^{n-1} \m({ \frac{c_m}{c_n} +o(1)} \cdot \m({ \frac{z_k(s)}{h(s)} }^m
\end{equation}
when $s\to 0$. When $\mb|{\frac{z_k(s)}{h(s)}}$ is large, the \rhs\ is dominated by the term $\mb({ \frac{z_k(s)}{h(s)} }^n$, thus it cannot vanish. Hence the ratio $\frac{z_k(s)}{h(s)}$ must stay bounded when $s\to 0$. 
If $\frac{z_k(s)}{h(s)}\to r\in\complex$ as $s\to 0$ along some subsequence, then \eqref{eq:ImplicitFT/asymp equation} implies that $\sum_{m=0}^n c_m r^m = 0$. In other words, all limit points of the function $\frac{z_k(s)}{h(s)}$ as $s\to 0$ are roots of the polynomial $c_0+\cdots+ c_n r^n$. 
Since $\frac{z_k}{h}$ is continuous in a neighborhood of $0\in S$ and $S$ is locally connected at $0$, it is a simple exercise to check that the limit point must be unique, i.e.\ the limit $r_k=\lim_{s\to 0}\frac{z_k(s)}{h(s)}$ exists in $\complex$. 
By computing $\lim \limits_{s\to 0} \frac{p(h(s)\cdot r,s)}{h(s)^n}$ with the two expressions $\prod \limits_{k=1}^n \!\mb({z-z_k(s)} = z^n\!+\! \sum \limits_{m=0}^{n-1} a_m(s)\, z^m$ of $p(z,s)$, we see that
\begin{equation}
c_0 + c_1 r + \cdots + c_n r^n = c_n \cdot \prod_{k=1}^n (r-r_k)
\end{equation}
for all $r\in \complex$. Therefore $r_1,\ldots,r_n$ are the $n$ roots of $c_0 + c_1 r + \cdots + c_n r^n$ listed in some order.
\end{proof}

\begin{remark*}
In the proof of \Cref{lem:generalized ImplicitFT}, we constructed a local factorization of the function $f$ of the form 
\begin{equation}
f(z,s) = c(z,s) \cdot \mb({z^n + a_{n-1}(s)z^{n-1} + \cdots + a_0(s)}\,,
\end{equation}
where the functions $c$ and $a_m$ are continuous, $z\mapsto c(z,s)$ is analytic, $a_m(0)=0$ and $c(0,0)\ne 0$. This is a version of the Weierstrass preparation theorem. The classical version usually assumes that $s\in \complex^k$ and $f(z,s)$ is analytic in both variables.
Our proof can be easily amended to show that if $f$ is $C^n$-continuous or analytic \wrt\ $(z,s)$, then so are the function $c$ and $a_m$.

\note{This proof is adapted from a proof of the classicial Weierstrass preparation theorem. But I lost the reference to that proof... It was probably a book or a blog.}
\end{remark*}

\note{To do: find counter examples which demonstrate the necessity of all these complicated methods for finding singularities on the boundary of a domain}

\bibliographystyle{/Users/linxiao/tex-lib/abbrev-2021-03-29}
\bibliography{/Users/linxiao/tex-lib/bibdata-2021-03-29}
\Addresses

\end{document}